\documentclass{article}
\usepackage[nottoc]{tocbibind}
\usepackage[titletoc,title]{appendix}
\usepackage[utf8]{inputenc}

\usepackage[pdftex]{graphicx}
\usepackage[pdftex,dvipsnames,usenames]{color}

\usepackage{enumitem}
\usepackage{mathtools}

\usepackage{cancel}

\usepackage[english]{babel}
\usepackage[hyperindex,breaklinks]{hyperref}

\hypersetup{
  colorlinks   = true, 
  urlcolor     = blue, 
  linkcolor    = blue, 
  citecolor   = red 
}

\usepackage{amsmath,amssymb,amsfonts,mathrsfs, amsthm}

\usepackage{graphicx}
\usepackage[margin=0.75in]{geometry}

\usepackage{soul}

\usepackage{pdfpages}

\usepackage{lipsum}
\usepackage{csquotes}
\usepackage{thmtools}

\usepackage[
  backend=biber,
  style=alphabetic,
  maxnames=4,       
  minnames=4,       
  maxbibnames=4,    
  minbibnames=4     
]{biblatex}
\usepackage{amsfonts,amsmath,amstext,amssymb,amsthm}
\usepackage{enumitem}
\usepackage[english]{babel}
\usepackage{fancybox}
\usepackage{xspace}
\usepackage{amscd}
\usepackage[all]{xy}
\usepackage[svgnames]{xcolor}
\usepackage{calc}

\usepackage{esint}

\usepackage{enumitem}

\usepackage{cases}
\usepackage{eqnarray}

\usepackage[compact]{titlesec}

\usepackage{accents}

\def\XXint#1#2#3{{\setbox0=\hbox{$#1{#2#3}{\int}$ }
\vcenter{\hbox{$#2#3$ }}\kern-.56\wd0}}

\newcommand{\eps}{\varepsilon}

\newcommand{\N}{\mathbb{N}}
\newcommand{\R}{\mathbb{R}}
\newcommand{\Z}{\mathbb{Z}}

\newcommand{\sign}{\textnormal{sign}\,}

\newcommand{\ceils}[1]{\lceil{#1}\rceil}

\newcommand{\dist}{{\rm dist}\, }
\newcommand{\supp}{{\rm supp}\, }

\newcommand{\cH}{\mathcal{H}}

\newcommand{\cX}{\mathcal{X}}
\newcommand{\cZ}{\mathcal{Z}}

\newcommand{\zz}{\mathbf{z}}

\newcommand{\epk}{{\widetilde \eps_k}}
\newcommand{\Rk}{{\widetilde R_k}}
\newcommand{\Rh}{{\hat R_k}}
\newcommand{\zk}{{\widetilde \zz_k}}

\newcommand{\cttb}{\beta_\circ}
\newcommand{\cttc}{\chi}

\newcommand{\Sp}{\mathbb{S}}

\newcommand{\ep}{\varepsilon}

\newcommand{\HH}{\operatorname{H}}

\usepackage{cleveref}

\newtheorem{theorem}{Theorem}[section]
\newtheorem{proposition}[theorem]{Proposition}
\newtheorem{lemma}[theorem]{Lemma}

\newtheorem{conjecture}[theorem]{Conjecture}

\theoremstyle{definition}

\newtheorem{definition}[theorem]{Definition}

\newtheorem{remark}[theorem]{Remark}

\usepackage{url}

\makeatletter
\renewenvironment{abstract}{
  \par\smallskip
  \noindent\textsc{Abstract.}\ %
}{
  \par\medskip
}
\makeatother

\makeatletter
\newcommand{\addressa}[1]{\gdef\@addressa{#1}}
\newcommand{\emaila}[1]{\gdef\@emaila{\url{#1}}}

\newcommand{\addressb}[1]{\gdef\@addressb{#1}}
\newcommand{\emailb}[1]{\gdef\@emailb{\url{#1}}}

\newcommand{\@endstuff}{\par\vspace{\baselineskip}\noindent
\begin{tabular}{@{}l}\scshape\@addressa\\\textit{E-mail address:} \@emaila\end{tabular} 

\vspace{12pt} \noindent
\begin{tabular}{@{}l}\scshape\@addressb\\ \textit{E-mail address:} \@emailb\end{tabular}
}

\AtEndDocument{\@endstuff}
\makeatother

\begin{document}

\setlength{\abovedisplayskip}{6pt}          
\setlength{\belowdisplayskip}{6pt}          
\setlength{\abovedisplayshortskip}{6pt}     
\setlength{\belowdisplayshortskip}{6pt}     

\setlength{\parskip}{0pt plus 1pt}
\setlength{\parindent}{15pt}
\setlength{\abovedisplayskip}{5pt plus 1pt minus 2pt}
\setlength{\belowdisplayskip}{5pt plus 1pt minus 2pt}

\setlist[itemize]{itemsep=2pt, topsep=3pt, partopsep=1pt, parsep=1pt}
\setlist[enumerate]{noitemsep, topsep=2pt, partopsep=0pt, parsep=0pt}

\titlespacing{\section}{0pt}{10pt plus 2pt minus 2pt}{6pt plus 1pt minus 1pt}
\titlespacing{\subsection}{0pt}{8pt plus 2pt minus 2pt}{5pt plus 1pt minus 1pt}

\title{\huge 
On stable solutions to the Allen--Cahn equation with bounded energy density in $\R^4$
}
\author{Enric Florit-Simon and Joaquim Serra}
\date{}
\maketitle

\begin{abstract}
We show that stable solutions $u:\R^4\to (-1,1)$ to the Allen–Cahn equation with bounded energy density (or equivalently, with cubic energy growth) are one-dimensional.

This is known to entail important geometric consequences, such as robust curvature estimates for stable phase transitions, and the multiplicity one and Morse index conjectures of Marques--Neves for Allen--Cahn approximations of minimal hypersurfaces in closed 4-manifolds.
\end{abstract}

\addressa{Enric Florit-Simon \\ Department of Mathematics, ETH Z\"{u}rich \\ Rämistrasse 101, 8092 Zürich, Switzerland}
\emaila{enric.florit@math.ethz.ch}
\newcommand{\enric}[1]{{\color{green}{#1}}}

\addressb{Joaquim Serra \\ Department of Mathematics, ETH Z\"{u}rich \\ Rämistrasse 101, 8092 Zürich, Switzerland }
\emailb{joaquim.serra@math.ethz.ch}
\newcommand{\joaq}[1]{{\color{blue}{#1}}}

\tableofcontents

\section{Introduction}

\subsection{The Allen–Cahn equation and its connection to minimal surfaces}

The theory of phase transitions naturally leads to the study of the Ginzburg--Landau-type energy functional
\begin{equation}\label{first_time_energy}
    \mathcal{E}^\varepsilon(u, \Omega) = \frac{1}{\sigma_{n-1}}\int_{\Omega} \left( \frac{\varepsilon}{2}|\nabla u|^2 + \frac{1}{\varepsilon} W(u) \right) dx,
\end{equation}
defined for scalar fields $u: \Omega \to \mathbb{R}$, where $\Omega \subset \mathbb{R}^n$ is open and $\sigma_{n-1}$ is defined in \eqref{eq:sig(n-1
)def}. The potential $W$ is typically chosen to be a symmetric double-well potential. Let us fix for concreteness the standard quartic potential $W(u) = \frac{1}{4}(1 - u^2)^2$.

Critical points $u_\eps$ of \( \mathcal{E}_\varepsilon \) satisfy the \emph{Allen--Cahn (A--C) equation}:
\begin{equation}\label{eq:epaceqintro}
    -\varepsilon \Delta u_\varepsilon + \frac{1}{\varepsilon}W'(u_\varepsilon) = 0.
\end{equation}

Originally introduced in the 1970s as a phase-field model for binary alloys \cite{AC72}, the Allen--Cahn equation has since become central in the study of variational problems, nonlinear PDEs, and geometric analysis, due in large part to its deep connection with the theory of minimal surfaces. In particular, as $\varepsilon \downarrow 0$, the narrow interfacial regions $\{-0.9 < u_\varepsilon < 0.9\}$ converge (in a suitable sense) to minimal hypersurfaces.

This connection has been extensively explored since the 1980s through several major developments:

\begin{itemize}
    \item In the 1970s and 1980s, foundational results established the connection between phase transitions and minimal surfaces \cite{MM77, Modica85b, ModicaMonotonicity}. During this period—and motivated by this connection—De Giorgi proposed a famous conjecture \cite{DeG78} on the classification of monotone solutions to the Allen--Cahn equation, which remains largely open to this day.

    \item The development of a regularity theory for energy-minimizing solutions—culminating around 2010—revealed a full analogy with the classical theory of area-minimizing hypersurfaces \cite{CC95, CC06, Sav09, Sav10}.

    \item In parallel to these developments, the varifold-based framework for the Allen–Cahn equation \cite{Hut86, HT00, Ton05, TonegawaWickramasekera2012} established that sequences of solutions with bounded energy converge (in a weak sense) to generalized minimal surfaces---more precisely, integral stationary varifolds.

    \item From the 2000s onwards, gluing methods were used to construct solutions whose interfaces converge to prescribed minimal surfaces \cite{delPinoKowalczykPacard2005, Pacard2009AllenCahn, delPinoKowalczykWei2007, delPinoKowalczykWei2008, delPinoKowalczykPacardWei2013}. This approach led to the resolution of De Giorgi’s conjecture in high dimensions \cite{dPKW11} and to the construction of non-flat minimizing solutions in dimension~8 \cite{PW13, LWW17b}.

    \item More recently, a new and powerful twist on gluing techniques has emerged: they have been used to develop a fine regularity theory for stable solutions with multiple, nearly parallel transition layers \cite{WW18, WW19, CM20}.
\end{itemize}

\subsection{Stable phase transitions and the regularity question} 

The natural functional domain for the energy \eqref{first_time_energy} is that of bounded functions in the Sobolev space $H^1(\Omega)$. 

We say that a $H^1$ function $u_\varepsilon:\Omega\to [-1, 1]$  is a \emph{minimizer} of $\mathcal{E}^\varepsilon$ in $\Omega$ if
\begin{equation}\label{minimizerconding}
    \mathcal{E}^\varepsilon(u_\varepsilon, \Omega) \le \mathcal{E}^\varepsilon(u_\varepsilon + \xi, \Omega) \quad \text{for all } \xi \in C^1_c(\Omega).
\end{equation}
That is, minimizers are understood here as \emph{absolute minimizers}—functions that minimize the energy among all admissible competitors with the same boundary data.

However, both from physical and geometric perspectives, it is natural to consider \emph{local minimizers}: functions $u_\varepsilon$ that satisfy \eqref{minimizerconding}  but only for variations $\xi$ that are small enough in the $H^1$ norm.

From the physical viewpoint, local minimizers correspond to stable equilibrium configurations—states toward which dissipative evolution processes (such as those governed by the time-dependent Allen--Cahn equation \cite{AC72} or the Cahn--Hilliard equation \cite{CH58}, for which $\mathcal{E}_\varepsilon$ serves as a Lyapunov functional) may evolve. Absolute minimizers describe only a restricted subclass of such configurations.

From the geometric viewpoint, in the context of Allen--Cahn approximations to minimal hypersurfaces in closed Riemannian manifolds \cite{TonegawaWickramasekera2012, Guaraco2018, GasparGuaraco2018, GGWeyl, CM20, ChodoshMantoulidis2023}, the relevant solutions are those of finite Morse index.
In particular, they are local—but not absolute—minimizers (in the function space, as above), in appropriate subdomains.

It is a standard fact that if $u_\varepsilon : \Omega \to [-1, +1]$ is a local minimizer of $\mathcal{E}_\varepsilon$ in $\Omega$, then $u_\varepsilon \in C^2(\Omega)$, satisfies the Euler--Lagrange equation \eqref{eq:epaceqintro}, and, in addition, the \emph{stability inequality} holds:
\begin{equation}\label{stability_1stintro}
   \int_\Omega \varepsilon |\nabla \xi|^2 + \frac{1}{\varepsilon} W''(u_\varepsilon) \xi^2 \, dx \ge 0 \quad \text{for all } \xi \in C^1_c(\Omega).
\end{equation}
Solutions of \eqref{eq:epaceqintro} that satisfy \eqref{stability_1stintro} are called \emph{stable}.

Thus, stability is a necessary condition for local minimality (in the function space, as defined above). In practice, except in degenerate and hence non-generic situations, the two notions are essentially equivalent\footnote{Indeed, it is a standard fact that a solution of A--C in a bounded domain 
$\Omega$ is \emph{stable} if and only if the first Dirichlet eigenvalue  of the 
Jacobi operator $Jv := -\Delta v + W''(u)\, v$
is nonnegative. Moreover, except for the degenerate (and non-generic) case in 
which the first eigenvalue is zero, stability implies that the solution is a 
local minimizer.}.

Given the central role of stability---both in the modeling of physical phase transitions and in the geometric study of minimal hypersurfaces---a fundamental question arises concerning the regularity of stable phase transitions:
\begin{quote}
    In dimension $n \le 7$, do sequences of stable solutions of Allen--Cahn converge smoothly, possibly with multiplicity, to minimal hypersurfaces in the singular limit $\varepsilon \downarrow 0$?
\end{quote}

A more precise formulation is:
\begin{conjecture}\label{conj:regularity}
    Let $\varepsilon_j \downarrow 0$, and suppose that $\{u_{\varepsilon_j}\}$ is a sequence of stable solutions of \eqref{eq:epaceqintro} in a given domain $\Omega \subset \mathbb{R}^n$, with uniformly bounded energy. If $n \leq 7$,  the curvatures of the level sets of $u_{\varepsilon_j}$ within the interfacial regions $\{ -0.9 < u_{\varepsilon_j} < 0.9 \}$  are uniformly bounded along the sequence in compact subdomains of $\Omega$.
\end{conjecture}

While the analogous question for embedded minimal hypersurfaces was resolved around 50 years ago in the classical works of Schoen, Simon, and Yau \cite{SSY75, SS81}---see also \cite{Wic14, Bel25} for important extensions removing a priori embeddedness assumptions---its counterpart in the context of phase transitions has remained elusive.

A positive answer was previously known only in dimension $n = 3$, as a consequence of the works \cite{AAC01, WW18, WW19, CM20} (see below for details). This paper establishes the result in dimension $n = 4$ and introduces ideas that may prove useful for addressing the remaining cases $n = 5, 6, 7$. The examples constructed in \cite{PW13} (see also \cite{LWW17b}) show that the conjecture cannot hold in dimensions $n \ge 8$, confirming that the restriction to dimensions $n \le 7$---originally motivated by the analogy with minimal surface theory---is indeed necessary.

For absolute minimizers, the regularity question is fully resolved through the theory developed in \cite{MM77, CC06,Sav09, Sav10}. However, these techniques do not carry over to the case of local minimizers, as they crucially depend on the convergence of absolute minimizers to classical minimal hypersurfaces with multiplicity one---a property that fails for local mimimizers.

Conjecture~\ref{conj:regularity} is known to have significant implications in geometric analysis. In particular, a positive resolution in a given dimension implies the validity of the Marques--Neves multiplicity one and Morse index conjectures \cite{MN16} for Allen--Cahn approximations in that same dimension, by the results in \cite{WW19, CM20}. See Section~\ref{sec:geomappintro} for a further discussion of this implication.

\subsection{The Wang–Wei reduction and the classification problem for stable solutions with bounded energy density}\label{sec:wwredclas}

Consider a solution $u : \mathbb{R}^n \to [-1, 1]$ of the Allen--Cahn equation with $\varepsilon = 1$:
\begin{equation}\label{eq:aceqintro}
    -\Delta u + W'(u) = 0 \quad \text{in } \mathbb{R}^n.
\end{equation}
For such a function, we define the \emph{energy density} on balls of radius $r > 0$ by
\begin{equation}\label{eq:densitydef}
    \mathbf{M}_r(u) := \frac{1}{r^{n-1}} \mathcal{E}^1(u, B_r)\,.
\end{equation}

By Modica's monotonicity formula \cite{Mod87, ModicaMonotonicity}, the map $r \mapsto \mathbf{M}_r(u)$ is nondecreasing in $r$.

We say that $u$ has \emph{bounded energy density} if ${\bf M}_\infty(u):=\lim_{r \to \infty} \mathbf{M}_r(u) < \infty$.

As we explain in more detail below (see Sections \ref{sec:wangwei} and \ref{sec:critsol}), the groundbreaking works of Wang--Wei \cite{WW18, WW19} and Chodosh--Mantoulidis \cite{ChodoshMantoulidis2023} reduce the regularity question for stable phase transitions, namely \Cref{conj:regularity}, to the following conjectural classification result:
\begin{conjecture}\label{conj:stabgiorgibnd}
    Let $u : \mathbb{R}^n \to [-1, 1]$ be a stable solution to \eqref{eq:aceqintro} with bounded energy density. Then, for $n \le 7$, $u$ must be one-dimensional,   that is either identically $\pm 1$ or of the form $u(x) = \tanh\left( \frac{e \cdot x - s_0}{\sqrt{2}} \right)$ for some unit vector $e \in \mathbb{S}^{n-1}$ and $s_0 \in \mathbb{R}$.
\end{conjecture}

A positive answer to this conjecture is currently known only for $n = 2$, by Ghossoub and Gui \cite{GG98}, and for $n = 3$, by Alberti, Ambrosio, and Cabré \cite{AAC01, AC00}. Both results date back over 25 years, and the question has remained open in higher dimensions since then.

In this paper, we address the case $n = 4$.

It is worth noting that \Cref{conj:stabgiorgibnd} is a special case of the so-called \emph{stable} (or \emph{strong}\footnote{The term “strong De Giorgi conjecture” is justified by the fact that it is known to imply the original De Giorgi conjecture for monotone solutions of the Allen–Cahn equation in $\mathbb{R}^{n+1}$.}) De Giorgi conjecture, which asserts the same classification result even without the assumption of bounded energy density. This stronger version of the conjecture is completely open even in dimension $n=3$.

The minimal surface analogue of this stronger conjecture---namely, that any complete, two-sided, stable minimal hypersurface in $\mathbb{R}^n$ must be flat for $n \le 7$---is a classical result for $n=3$ \cite{FS80, DP79, Pog81} and has recently been established in the breakthrough works \cite{CL24} for $n = 4$ (see also \cite{CL23, CMR24}), \cite{CLMS24} for $n = 5$, and \cite{Maz24} for $n = 6$. The final case $n = 7$ remains open. Unfortunately, the powerful and delicate techniques from intrinsic differential geometry used in these proofs appear to be inapplicable to the setting of phase transitions (strong De Giorgi) or other similar variational scaling-dependent problems.

\subsection{Main result}
In this paper we establish the following:
\begin{theorem}[{\bf Main result: Classification in $\R^4$}]\label{mainthm} 
Let $u:\R^4\to [-1,1]$ be a stable solution to \eqref{eq:aceqintro} with bounded energy density. Then, $u$ is either identically $\pm 1$ or of the form $\tanh\left( \frac{e \cdot x - s_0}{\sqrt{2}} \right)$ for some unit vector $e\in\Sp^{n-1}$ and $s_0\in\R$.
\end{theorem}
In other words, we establish \Cref{conj:stabgiorgibnd} in dimension $n=4$.

Combining \Cref{mainthm} with the main result in \cite{WW19}, local curvature estimates for stable solutions directly follow. Indeed, let us define
\begin{align*}
        \mathcal A^2(u)&=\begin{cases}
            \frac{|D^2u|^2-|\nabla|\nabla u||^2}{|\nabla u|^2}\quad &\mbox{ if } \nabla u\neq 0\\
            0 &\mbox{ otherwise\,;}
        \end{cases}
        \qquad \quad \mbox{and}\qquad \mathcal A(u)=\left(\mathcal A^2(u)\right)^{1/2}.
    \end{align*}
It is easy to see that if $\nabla u(x)\neq 0$ then $\mathcal A^2(u)(x)=|{\rm I\negthinspace I}_{\{u=u(x)\}}|^2(x)+|\nabla^T \log|\nabla u||^2(x)$, where ${\rm I\negthinspace I}_{\{u=u(x)\}}$ is the second fundamental form of the level set $\{u=u(x)\}$ and $\nabla^T$ denotes the gradient in the directions tangent to $\{u=u(x)\}$.
\begin{theorem}[\textbf{Regularity for level sets, $\ep$-version}]\label{thm:curvestseps}
    Assume that $u_\ep:B_1\subset\R^4\to (-1,1)$ is a stable solution to $\ep$-Allen--Cahn, satisfying that $\mathcal E^\ep(u_\ep,B_1)\leq \Lambda$. Then, there are $\ep_0>0$ and $C$  depending only on $\Lambda$ such that if $\ep\leq \ep_0$, then
    \begin{equation}\label{eq:curvestind}
    |\nabla u_\ep|\geq \frac{1}{C\ep} \quad \mbox{and} \quad \mathcal A(u_\ep)\leq C \quad \mbox{ in } \, \{|u_\ep(x)|\leq 0.9\}\cap B_{1/2}\,.
\end{equation}
In particular, $\{u_\ep=t\}\cap B_{1/2}$ is then a smooth hypersurface for all $|t|\leq 0.9$, with
\begin{equation}\label{lvl2ffind}
    |{\rm I\negthinspace I}_{\{u_\ep=t\}}|\leq C \quad \mbox{ in } \, B_{1/2}.
\end{equation}
\end{theorem}
This shows the validity of \Cref{conj:regularity} in dimension $n=4$.

Some remarks are in order: 

\begin{remark}   
    The regularity theory for stable solutions with multiple flat, nearly parallel interfaces developed in~\cite{WW18, CM20, WW19} is a fundamental element in the proof of~\Cref{mainthm}. This is recalled in \Cref{sec:wangwei}.\\
    To go beyond the scope of existing techniques, we introduce several new ingredients that may be of independent interest, including:
    \begin{itemize}
        \item A ``continuous induction'' argument on the value of ${\bf M}_\infty$, the energy density at infinity of the solutions, which reduces \Cref{mainthm} to the classification of a single {\it critical solution} with remarkably rigid properties.
        \item A tangential form of the stability inequality, which controls ``bad balls'' with large curvature by a notion of height excess (i.e. flatness).
        \item A new monotonicity-type formula that relates the height excess back with the energy density.
    \end{itemize}
    \Cref{sec:strucproof} presents all of the ingredients mentioned above and provides a detailed outline of the core of the proof. The goal is to give the reader a clear sense of both the main difficulties inherent in the problem and the strategy we develop to overcome them.
    \end{remark}
    \begin{remark}\label{rem:analogies}
    It is noteworthy that a central part of our proof (outlined in~\Cref{sec:contradictionintro}) shows a striking formal resemblance to the argument developed in the first part of~\cite{CFFS25} dealing with Bernoulli's free boundary problem. 

    Even more intriguingly, although the overarching structure of these two arguments (Bernoulli part of~\cite{CFFS25} and \Cref{sec:contradictionintro} on this paper) aligns closely, the specific ingredients involved in our proof are entirely different. 

    The connection becomes apparent only {\em a posteriori}, once one identifies the appropriate dictionary to translate objects and scalings between the two seemingly unrelated problems. In writing~\Cref{sec:contradictionintro}, we made a deliberate effort to highlight this analogy.

    That said, we believe that leveraging this analogy in a meaningful way is quite nontrivial, requiring several entirely new ideas. From this perspective, the strategy we develop---outlined in detail in Section~\ref{sec:strucproof}---is genuinely original.

    Finally, we emphasize that although~\cite{CFFS25} also contains results on the ``free boundary Allen--Cahn'' problem, those are completely unrelated to our work. Only the part of the  Bernoulli part exhibits meaningful analogies with our approach.
\end{remark}

\begin{remark}
    Several parts of our proof remain valid in higher dimensions. In fact, a variant\footnote{This is explored in a work in preparation by the authors.} of our overall strategy yields the classification of embedded stable minimal hypersurfaces with Euclidean area growth up to dimension~$7$, independently of~\cite{SS81}.

    In this paper, the case $n = 4$ of \Cref{conj:stabgiorgibnd} is established by using the main results of~\cite{WW19} as a ``black box''.
    We emphasize that while the regularity theory in~\cite{WW19} is formulated in all dimensions, its main results---which are optimal, at least as currently formulated---do not allow us to carry out our strategy in dimensions $5 \leq n \leq 7$. Nonetheless, we believe it is possible that delving deeper into the proofs in~\cite{WW19} and extracting suitable versions of certain intermediate steps, one could combine them with ideas introduced in this paper to tackle the  higher-dimensional cases. For this reason, we have written parts of the paper in general dimension~$n$.

    We give a more technical comment on this dimensional obstruction  in~\Cref{obstructionshd}, after the overview of the proofs.
\end{remark} 

\subsection{Geometric applications: Min-max solutions and the multiplicity one and Morse index conjectures}\label{sec:geomappintro}
There has been a growing interest in using Allen--Cahn approximations to construct geometric objects with special properties, including minimal hypersurfaces on closed manifolds. In particular, we highlight the results in \cite{Guaraco2018,GasparGuaraco2018,GGWeyl}, building on \cite{TonegawaWickramasekera2012,Wic14}:
\begin{itemize}
    \item A remarkably simple min-max construction, of mountain-pass type, exhibits the existence of rich families of $\ep$-Allen--Cahn solutions on manifolds. More precisely, fixed a closed $n$-dimensional manifold $M$, for every $p\in\N$ one obtains (for $\ep>0$ sufficiently small) solutions $u_{\ep}^p:M\to (-1,1)$ with energy $\sim p^{1/n}$ and Morse index $\leq p$.
    \item Passing them to the limit (via \Cref{thm:HutchTon}, obtained in \cite{HT00}) as $\ep\to 0$, one obtains integral stationary varifolds (i.e. generalised minimal hypersurfaces) $\Sigma^p$. Moreover, the Morse index bounds mean that the $\Sigma^p$ are locally stable.
    \item Using \Cref{thm:tonwick} (obtained in \cite{TonegawaWickramasekera2012}, which uses the deep and powerful regularity theory for stable integral varifolds in \cite{Wic14}), the limits are then seen to be of optimal regularity (i.e., as regular as in the case of area minimisers).
\end{itemize}
A main issue in the strategy above is that, without \Cref{conj:regularity} (or equivalently \Cref{conj:stabgiorgibnd}), in the second bullet we are forced to pass the $u_\ep^p$ to the limit using \cite{HT00}, i.e. just as {\it critical points}---using stability essentially only for the limit objects. There is then no geometric control in the convergence, which allows degeneration to occur\footnote{On manifolds with boundary, the situation is even worse, as transition level sets could potentially collapse onto the boundary, even when the latter is not minimal \cite{LPS24}. In the prescribed mean curvature (i.e. inhomogeneous A--C equation) case, the transition level sets could fold into a zero mean curvature submanifold instead, losing all the curvature information in the limit \cite{BW21,MZ25}.}: Indeed, several sheets of $\{u_\ep=0\}$ could collapse onto the same limit, like catenoids converging to a hyperplane, losing all information on energy, index or topology coming from the $u_\ep^p$. The {\it multiplicity one} and {\it Morse index} conjectures\footnote{Originally formulated in the Almgren--Pitts min-max setting.} \cite{MN16} state that this should not happen generically, and they were first confirmed for $n=3$ in the breakthrough article \cite{CM20} (via Allen--Cahn approximation). More precisely, \cite{CM20} shows that on three-dimensional closed manifolds with a generic Riemannian metric, the $\Sigma^p$ above arise as smooth, multiplicity one limits of the level sets $\{u_{\ep}^p=0\}$, and they have area $\sim p^{1/3}$ and Morse index exactly $p$. In particular, they are all distinct. This shows a strong form of a famous conjecture of Yau on the existence of infinitely many minimal surfaces on closed, three-dimensional manifolds.\\

\cite{WW19} extended the local estimates required in \cite{CM20} to higher dimensions (up to 10, surprisingly), for stable solutions satisfying a-priori curvature estimates (i.e., assuming precisely the thesis of \Cref{conj:regularity}). As explained in \cite[Remark 1.4]{CM20} and \cite[Remark 10.9]{WW19}, the only present bottleneck to showing the multiplicity one and Morse index conjectures for Allen--Cahn approximations in dimension $n$, with $4\leq n\leq 7$, is then a positive answer to \Cref{conj:stabgiorgibnd} in that dimension. Our \Cref{mainthm} confirms it for $n=4$.

\subsection{Acknowledgements}
It is a pleasure to thank Guido De Philippis for interesting discussions on the topic of this paper.

Both authors were supported by the European Research Council under the Grant Agreement No 948029.

\section{Previous results from the literature}
\subsection{Monotonicity formula}
\begin{lemma}[\textbf{Modica inequality, \cite{Modica85}}]\label{lem:modicaineq}
Let $u:\R^n\to\R$ be a bounded A--C solution on all of $\R^n$. Then, $|u|\leq 1$, and the inequality is strict unless $u\equiv\pm 1$. Moreover,
    $$\frac{|\nabla u|^2(x)}{2}\leq W(u(x))\,.$$
\end{lemma}
Throughout this article, unless otherwise indicated we assume that $u:\R^n\to(-1,1)$ is a solution to A--C on all of $\R^n$, to make use of the Modica inequality. Define
\begin{equation}\label{eq:sig(n-1
)def}
    \sigma_{n-1}:=\omega_{n-1}\int_{-1}^1 \sqrt{2 W(s)}\,ds\,,\quad\mbox{where } \omega_{n-1}\mbox{ is the volume of the unit ball in } \R^{n-1}.
\end{equation}
Recall the definition of the monotonic energy density ${\bf M}_r$ in \eqref{eq:densitydef}. We will more generally denote ${\bf M}_r(u,x_0):={\bf M}_r(u(\cdot-x_0))$, and we omit $u$ from the notation whenever it is clear from the context.
    \begin{remark}\label{rmk:rescdens}
    For a solution $u$ of $\ep$-A--C instead, we naturally set ${\bf M}_r^\ep(u):=\frac{1}{r^{n-1}}\mathcal E^\ep(u,B_r)={\bf M}_{r/\ep}(u(\ep x))$. Unless otherwise stated, we will always work with $\ep=1$.
\end{remark}
\begin{lemma}[\textbf{Monotonicity formula, \cite{Modica85b}}]\label{lem:monformula}
    Let $u:\R^n\to[-1,1]$ be an A--C solution on all of $\R^n$. Then, ${\bf M}_r$ is monotone nondecreasing in $r$. More precisely,
    \begin{equation}\label{eq:monfor}
        \frac{d}{dr}{\bf M}_r(u)=\frac{1}{r^n}\int_{B_r}\frac{1}{\sigma_{n-1}}\left[W(u)-\frac{|\nabla u|^2}{2}\right]\,dx+\frac{1}{r^{n+1}}\int_{\partial B_r}\frac{1}{\sigma_{n-1}}(x\cdot\nabla u)^2\,d\cH^{n-1}(x)\,.
    \end{equation}
\end{lemma}

\subsection{Some results on stable solutions}
There is a (weaker) form of the stability inequality \eqref{stability_1stintro}, which closely resembles the one for minimal hypersurfaces.
\begin{proposition}[\textbf{Sternberg--Zumbrun inequality, \cite{SZ98}}]\label{prop:SZstab}
    Let $u:\R^n\to \R$ be a stable solution to A--C, and let $\eta\in C_c^1(\R^n)$. Then,
    \begin{equation}\label{eq:SZineq1}
        \int\mathcal A^2\eta^2\,|\nabla u|^2\leq\int |\nabla \eta|^2|\nabla u|^2\,.
    \end{equation}
\end{proposition}
\begin{remark}\label{rmk:2ffsb}
    Assume that $\nabla u(x_0)\neq 0$. Then, in a Euclidean coordinate frame with $\frac{\nabla u}{|\nabla u|}(x_0)=e_1$, one has $|\nabla u|^2\mathcal A^2(x_0)=\sum_{j=2}^{n}\sum_{i=1}^n u_{ij}^2(x_0)$.
\end{remark}

We will repeatedly use:
\begin{lemma}\label{lem:A0imp1D}
    Let $u:\R^n\to\R$ be a solution to A--C, and assume that $\mathcal A\equiv 0$ in some open set $\Omega$. Then $u$ is one-dimensional in $\R^n$, i.e. $u=v(e\cdot x)$ for some $e\in\Sp^{n-1}$ and $v:\R\to\R$.
\end{lemma}
\begin{proof}
If $\nabla u\equiv 0$ in $\Omega$ then $u$ is constant by unique continuation. Otherwise, we find some open cube $Q\subset \Omega$ with $\nabla u\neq 0$.

\noindent {\bf Step 1.} $u$ is 1D in $Q$.\\
Indeed, we can compute
\begin{align}
    D \frac{\nabla u}{|\nabla u|}=\frac{D^2 u-\nabla|\nabla u| \otimes \frac{\nabla u}{|\nabla u|}}{|\nabla u|}\,,\quad\mbox{thus}\quad \mathcal A^2=\left\|D \frac{\nabla u}{|\nabla u|}\right\|^2\,,
\end{align}
which shows that $\frac{\nabla u}{|\nabla u|}$ is constant in $Q$ precisely when $\mathcal A\equiv 0$ there. Letting $e=\frac{\nabla u}{|\nabla u|}$, this means that $u=\tilde v(e\cdot x)$ in $Q$ for some $\tilde v:\R\to\R$.

\noindent {\bf Step 2.} Conclusion.
Given $w\in \Sp^{n-1}$ with $w\cdot e= 0$, we see that $\partial_w u\equiv 0$ in $Q$. By unique continuation, this shows that $\partial_w u\equiv 0$ in $\R^n$, thus $u= v(e\cdot x)$ in $\R^n$ for some $v:\R\to\R$.
\end{proof}

\begin{definition}
    Let $\phi:\R\to(-1,1)$ be defined as $\phi(s)=\tanh\big(\frac{s}{\sqrt{2}}\big)$; this is a monotone strictly increasing solution of the A--C equation in 1D, and ${\bf M}_\infty(\phi)=\lim_{R\to\infty}\mathcal E(\phi,B_R)=1$.
\end{definition}
Then, one has the following:

\begin{proposition}[\textbf{Classification in $\R$}]\label{prop:1dclas}
    Let $u:\R\to\R$ be a stable A--C solution. Then, $u$ is either $\pm 1$ or $\pm \phi(s-s_0)$ for some $s_0\in\R$. The same classification holds for solutions with $\lim_{R\to\infty}\mathcal E(u,B_R)<\infty$.
\end{proposition}
This can be shown by an ODE analysis. The best known results in higher dimensions are:
\begin{theorem}[\textbf{Classification in $\R^2$, \cite{GG98}}]\label{thm:2dclas}
    Let $u:\R^2\to [-1,1]$ be a stable A--C solution. Then, $u$ is either $\pm 1$ or $\phi(a\cdot x+b)$ for some $a\in\Sp^1$ and $b\in\R$.
\end{theorem}
\begin{theorem}[\textbf{Classification in $\R^3$, \cite{AAC01, AC00}}]\label{thm:3dclasenas}
    Let $u:\R^3\to [-1,1]$ be a stable A--C solution, and assume moreover that ${\bf M}_\infty(u)<\infty$. Then, $u$ is either $\pm 1$ or $\phi(a\cdot x+b)$ for some $a\in\Sp^2$ and $b\in\R$.
\end{theorem}
A modern proof of these two results consists in plugging in a log-cutoff $\eta$ into \eqref{eq:SZineq1} to see that $\mathcal A$ vanishes, so that \Cref{lem:A0imp1D} reduces the results to \Cref{prop:1dclas}.\\

Finally, we note the following simple but remarkable result:
\begin{theorem}[{\bf Discrepancy decay, \cite[Proposition 2.4]{Vil22}}]\label{thm:villegas}
Let $u:\R^n\to [-1,1]$ be a stable solution to A--C. Then, there is $C=C(n)$ such that
\begin{equation}\label{eq:deltamod3}
    \frac{1}{R^{n-1}}\int_{B_R} W(u)-\frac{|\nabla u|^2}{2}\leq \frac{C}{R^{1/3}} \,.
\end{equation}
\end{theorem}

\subsection{Wang--Wei ``a-priori'' estimates}\label{sec:wangwei}
\begin{definition}[{\bf Sheeting assumptions}]\label{def:sheetassump}
Let $u:B_R\to (-1,1)$ be a stable $\ep$-A--C solution. We say that $u$ satisfies the sheeting assumptions in $B_R$ (with constant $C_1>0$) if
\begin{equation}\label{eq:sheetassump}
     \ep|\nabla u|\geq \frac{1}{C_1}\quad \mbox{and} \quad R|\mathcal{A}_{u}|\leq C_1 \quad \mbox{in } \{|u|\leq 0.9\}\cap B_R.
\end{equation}
\end{definition}

In the lemma below---and throughout the paper---we will adopt the following notations: Given a point $x \in \mathbb{R}^n$, we write $x' \in \mathbb{R}^{n-1}$ for its first $n-1$ coordinates, so that $x = (x', x_n)$. Moreover, given a set $\Omega \subset \mathbb{R}^{n-1}$ and a function $g : \Omega \to \mathbb{R}$, we denote by ${\rm graph}\, g$ the set of points in $\Omega \times \mathbb{R}$ satisfying $x_n = g(x')$.\\

Set $\mathcal C_R:=B_R'\times [-R,R]$, where $B_R'\subset \R^{n-1}$. Since $\mathcal A$ controls the curvatures of the level sets, one has the following standard lemma:
\begin{lemma}\label{lem:rjbagbda}
    In the setting of \Cref{def:sheetassump}, assume that $u(0)= t\in [-0.9,0.9]$. There are $\delta\in(0,1)$ and $C_2$, depending on $C_1$ and $n$, such that:\\
    Let $\{\Gamma^t_i\}_{i=1}^N$ denote the connected components of $\{u=t\}\cap B_{\delta^2 R}$. Then there are $g^t_1<...< g^t_N$, $g^t_i\in C^\infty(B_{\delta R}')$, such that---after choosing a suitable Euclidean coordinate frame:
    \begin{equation*}
        \Gamma_i^t={\rm graph} \,g_i^t\quad \mbox{in }\ \mathcal C_{\delta R}\quad\mbox{and}\quad |D g_i^t|+R|D^2 g_i^t|\leq C_2\,.
    \end{equation*}
\end{lemma}
Wang and Wei showed that a-priori $C^2$ bounds can be upgraded to $C^{2,\vartheta}$ bounds, obtaining moreover improved separation and mean curvature estimates.
\begin{theorem}[\textbf{\cite{WW19}, $C^2$ implies $C^{2,\vartheta}$}]\label{thm:sheetimpc2alpha}
    Let $n\leq 10$. Let $u:B_R\to (-1,1)$ be a stable $\ep$-A--C solution, and assume the sheeting assumptions hold in $B_R$ for some $C_1$. Then, in the conclusion of \Cref{lem:rjbagbda}, for every $\vartheta\in(0,1)$ there are $C_2$ and $\delta_0>0$ depending additionally on $\vartheta$ such that if $\frac{\ep}{R}\leq \delta_0$, then
    \begin{equation}\label{eq:c2alphasheet}
        \|D^2 g_i^t\|_{L^\infty}+R^\theta[D^2 g_i^t]_{C^\vartheta}\leq \frac{C_2}{R}\,.
    \end{equation}
    Moreover, letting $\HH[f]:={\rm div}(\frac{\nabla f}{\sqrt{1+|\nabla f|^2}})$ denote the mean curvature operator,
    \begin{equation}\label{eq:meancurvsheet}
         \|\HH[g_i^t]\|_{L^\infty}+ \big[\HH[ g_i^t]\big]_{C^\vartheta}\leq \frac{\ep C_2}{R^2}\,.
    \end{equation}

    Additionally, the separation between layers satisfies
    \begin{equation}\label{eq:wwsepbnd}
        g_{i+1}^t-g_i^t\geq \frac{1+\vartheta}{2}\sqrt{2}\ep\log{\left(\frac{R}{\ep}\right)}.
    \end{equation}
\end{theorem}
\begin{remark}
    We emphasise that we will consider $\ep=1$ in the vast majority of the article.
\end{remark}

\subsection{Wang's classification result}
Wang developed an analogue of Allard's regularity theory for stationary varifolds in the Allen--Cahn case. In particular, he obtained the following theorem, which allows to give a new proof of Savin's result \cite{Sav09}.
\begin{theorem}[\textbf{\cite{Wang17}, Allard-type theorem for A--C}]\label{thm:ACallard}
    Let $u:\R^n\to [-1,1]$ be a solution to A--C. Then, there is $\delta=\delta(n)>0$ such that if ${\bf M}_\infty(u)\leq 1+\delta$, then $u$ is either $\pm 1$ or $\phi(a\cdot x+b)$ for some $a\in\Sp^{n-1}$ and $b\in\R$.
\end{theorem}

\section{The core of the proof}\label{sec:strucproof}
\subsection{Reduction to a critical solution}\label{sec:critsol}
{\bf Two regularity results.} \Cref{thm:sheetimpc2alpha} may be regarded as an \emph{a priori estimate}: it provides strong compactness and regularity information under suitable assumptions, but its applicability appears to rely on the very classification result we aim to establish. In contrast, \Cref{thm:ACallard} shows that the classification holds—in all dimensions and without assuming stability—for solutions whose densities at infinity are sufficiently close to $1$. This motivates:
\begin{definition}[{\bf Subcritical density}]\label{def:induction}
    Let $n\in\N$, $n\le 7$, and let $K\in\R_+$. We say that $K$ is a {\em subcritical density} in $\R^n$ if the only bounded stable solutions to the A--C equation in $\R^n$, with density at infinity ${\bf M}_\infty\leq K$, are $\pm 1$ and $\phi(a\cdot x+b)$ for some $a\in\Sp^{n-1}$ and $b\in\R$.
\end{definition}
With this perspective, \Cref{mainthm} amounts to showing that \emph{every} $K > 0$ is a subcritical density. Intuitively, our strategy is to prove \Cref{mainthm} via a form of \emph{continuous induction} on the density parameter $K$.

Using \Cref{thm:sheetimpc2alpha} we readily obtain regularity in regions of subcritical density (and just as in the Allard case, we get an extra $\delta$) -- this is a general feature in geometric variational problems:
\begin{theorem}[{\bf Regularity in subcritical regions}]\label{prop:curvestindeps}
    Let $n\leq 7$, and let $K$ be a subcritical density in $\R^n$. Let $u:\R^n\to(-1,1)$ be a stable solution to $\ep$-A--C. Then, there are $\delta=\delta(K)>0$ and $C=C(K)$ such that the following holds:\\
    Assume that
    $${\bf M}_R^\ep(u)\leq K+\delta\quad\quad \mbox{and}\quad\quad \frac{\ep}{R} \leq \delta\,.$$
    Then, the sheeting assumptions (recall \Cref{def:sheetassump}) hold in $B_{\delta R}$ with this $C$.
\end{theorem}
\begin{remark}
    We have stated this result for $\ep$-A--C solutions in general just for ready applicability when performing rescaling arguments; unless otherwise stated, everything else will be only for solutions to A--C with parameter $1$.
\end{remark}
\begin{proof}[Sketch of proof (see detailed proof in \Cref{subsec:regpfview} below)]
The reasoning closely follows that of \cite[Corollary 1.3]{WW19}.

We proceed in the spirit of B.~White \cite{white2016lecturesminimalsurfacetheory}. Suppose, for contradiction, that there exists a sequence of solutions $u_i$ with $\delta_i \to 0$ as in the statement, yet with no uniform curvature bounds. By zooming in at points where the curvature is nearly maximal, we obtain a rescaled sequence $\widetilde{u}_i$ that now has \emph{uniform curvature bounds} on expanding domains. 

At this point, \Cref{thm:sheetimpc2alpha} applies: the $C^{2,\vartheta}$ bounds combined with the Arzelà--Ascoli theorem yield $C^2$ convergence to a global solution of the Allen--Cahn equation or to a complete minimal hypersurface. This limit object has density bounded by $K$ and unit curvature at the origin. Then, recalling \Cref{def:induction}---or using the flatness of complete, stable minimal hypersufaces with bounded density \cite{SSY75, SS81}---this leads to a contradiction.
\end{proof}

Examining the proof carefully, we find another natural condition with which we already know the curvature estimates. A similar observation (in a somewhat different form) was already suggested by Wang--Wei.
\begin{theorem}[\textbf{Regularity in good balls}]\label{thm:goodimpsheet}
    Let $n\leq 7$. Let $u:\R^n\to(-1,1)$ be a stable solution to A--C, with ${\bf M}_R(u)\leq C_0$. Then, there are constants $\delta_{bad}>0$, $C$ and $R_0$, depending only on $C_0$, such that the following holds.\\
    Assume that $R\geq R_0$, and that
    \begin{equation}\label{prox1d}
        \int_{B_1(x)}\mathcal A_u^2|\nabla u|^2<\delta_{bad}\quad\quad\mbox{for every}\ x\in B_R\cap\{|u|\leq 0.9\}\,.
    \end{equation}
    Then, the sheeting assumptions (recall \Cref{def:sheetassump}) hold in $B_{\frac{R}{2}}$ with this $C$.
\end{theorem}
\begin{proof}[Sketch of proof (see detailed proof in  \Cref{subsec:regpfview} below)]
We run the same contradiction argument, with $\delta_{bad,i}\to 0$ for contradiction. The only difference is that, in the global Allen--Cahn limit case, passing \eqref{prox1d} also to the limit we would get a solution with $\mathcal A_{u_\infty}\equiv 0$ in $B_1$. But then it is one-dimensional by \Cref{lem:A0imp1D}, reaching a contradiction just as before.
\end{proof}

{\bf The critical solution.} We set:
\begin{definition}
    Let $n\leq 7$. We call $K_* : = \sup \{K>0 \mbox{ subcritical density in } \R^n\}$ the {\bf critical density}. We say that $u:\R^n\to(-1,1)$ is a {\bf critical solution} if $u$ is a stable solution to A--C such that ${\bf M}_\infty=K_*$ and $u$ is not 1D.
\end{definition}
Unless every $K>0$ is a subcritical density---as we ultimately want to establish---we must have $K_*<\infty$. We can then show:
\begin{proposition}[{\bf Critical density is attained}]\label{prop:critexist}
    Subcritical densities form an open set. In other words, assuming that $K_*<\infty$, there exists a critical solution.
\end{proposition}
\begin{proof}[Proof 1]
     We show that the set of subcritical densities in $\R^n$ is open. Let $K<\infty$ be subcritical. Then, \Cref{prop:curvestindeps} gives some $\delta=\delta(K,n)>0$ such that, for any stable solution with ${\bf M_\infty}(u)\leq K+\delta$, we have that $\mathcal A_u\leq \frac{C}{R}$ for every $R$ large enough. Then obviously $\mathcal A_u\equiv 0$, thus $u$ is one-dimensional, which shows that $K+\delta$ is subcritical as well.
\end{proof}
\begin{proof}[Proof 2] 
    We show that the complement is closed instead. Assume that $K_*<\infty$, so that there exists a sequence $K_i\searrow K_*$ of densities which are not subcritical; up to passing to a subsequence, $K_i\leq K_*+1$. By definition, there exist stable solutions $u_i:\R^n\to (-1,1)$ with ${\bf M}_\infty(u_i)\leq K_i$ but which are not 1D. By \Cref{thm:goodimpsheet} there need to be some $\zz_i\in\R^n$ such that 
    \begin{equation}\label{eq:eibpbgvpabbvv}
        \int_{B_1(\zz_i)}\mathcal A_{u_i}^2|\nabla u_i|^2\geq \delta_{bad}\,,
    \end{equation}
    as otherwise \Cref{thm:goodimpsheet} would give that $\mathcal A_{u_i}\leq \frac{C}{R}$ for every $R$ large enough, showing that $u_i$ is one-dimensional. Up to a translation, $\zz_i=0$. Passing to a subsequential limit (using standard interior $C^3$ estimates for solutions of Allen--Cahn), we obtain a bounded stable solution $u_\infty$ to A--C with
    $${\bf M_\infty}(u_\infty)=\lim_{r\to\infty}{\bf M}_r(u_\infty)=\lim_{r\to\infty}\lim_{i\to\infty}{\bf M}_r(u_i)\leq \lim_i K_i=K_*\,.$$
    On the other hand, by \eqref{eq:eibpbgvpabbvv} we see that
    \begin{equation*}
        \int_{B_1} \mathcal A_{u_\infty}^2|\nabla u_\infty|^2\geq \delta_{bad}\,.
    \end{equation*}
    But then $u_\infty$ is not one-dimensional, thus $K_*$ cannot be subcritical either.
\end{proof}
From the above, we are compelled to define:
\begin{definition}[\textbf{Bad centers and balls}]\label{def:badballintro}
    Fix $n\leq 7$, and let $u:\R^n\to (-1,1)$ be a critical solution. Let $\delta_{bad}>0$ be given by \Cref{thm:goodimpsheet}, with $C_0=K_*$.
    
    We say that $\zz\in\R^n$ is a {\bf bad center}, and that $B_1(\zz)$ is a {\bf bad ball} (of radius one), if $|u(\zz)|\leq 0.9$ and
    \begin{equation}\label{prox1d2intro2}
       \int_{B_1(\zz)} \mathcal A_u^2 |\nabla u|^2 \geq \delta_{bad}.
    \end{equation}
    Moreover, we call $\cZ(u):=\{\zz\in\R^n: \zz\mbox{ bad center}\}$ the {\bf bad set}, and denote with
    $$B_R(\cZ(u)\cap \Omega):=\bigcup_{\zz\in\cZ(u)\cap\Omega}B_R(\zz)=\{x:{\rm dist}(x,\cZ(u)\cap\Omega)\leq R\}$$
    the $R$-neighborhood of the bad centers in a set $\Omega$. We will write $\cZ$ instead of $\cZ(u)$ when there is no risk of confusion.
\end{definition}
Then, \Cref{thm:goodimpsheet} gives curvature estimates in the absence of bad centers. In particular:
\begin{lemma}[{\bf Existence of bad centers}]\label{lem:badnonempty}
    If $u$ is a critical solution, then $\cZ(u)\neq\emptyset$.
\end{lemma}
\begin{proof}
    Otherwise, \Cref{thm:goodimpsheet} would give that $\mathcal A_u\leq \frac{C}{R}$ for every $R$ large enough. Then obviously $\mathcal A_u\equiv 0$, thus $u$ would be one-dimensional, a contradiction.
\end{proof}
{\bf Large scale flatness.}
The critical solution has an extremely rigid structure:
\begin{proposition}[\textbf{Large-scale flatness}]
\label{prop:W}
Let $n\leq 7$. There exists a {\bf dimensional} modulus of continuity $\omega$ such that the following holds: Let $u:\R^n\to (-1,1)$ be a critical solution, and let $\zz\in \mathcal Z(u)$. 

For any $R\geq 1$, there exists $e_{\zz,R}\in \Sp^{n-1}$ such that
\begin{equation}
\label{eq:omega_modintro}
	\{|u|\leq 0.9\}\cap B_R(\zz)\subset\{|e_{\zz,R}\cdot (x-\zz)|\leq \omega(R^{-1})R\}
	\quad \text{ and } \quad
	K_*-\omega(R^{-1})\leq {\bf M}_R(\zz)\leq K_*.
\end{equation}
Moreover, $K_*$ is an integer. We emphasise that, in particular, $\omega$ is {\bf independent of $\zz\in \cZ$}.
\end{proposition}
In other words, for a sufficiently large scale---around {\bf any} bad center---the transition level sets of our solution become close to a hyperplane of multiplicity $K_*$.
Needless to say, this is a vast improvement with respect to a general solution\footnote{Think, for instance, of a catenoidal-type solution with a neck at scale one, looking flat again for several larger scales, and then developing a much larger neck at scale $R\gg 1$.}, and we will make heavy use of this structure in the rest of the article.
\begin{proof}[Sketch of proof (see detailed proof in \Cref{subsec:indlargesc} below)]

The reason behind \eqref{eq:omega_modintro} is as follows:
\begin{itemize}
    \item The densities $K<K_*$ are subcritical. If ${\bf M}_{R}(\zz)\leq K$, \Cref{prop:curvestindeps} then yields curvature estimates (for $R$ large enough, and with constants depending on $K$). Letting $R\nearrow\infty$, since by definition a bad ball is a region with a definite amount of curvature, this forces ${\bf M}_{R}(\zz)\nearrow K_*$, and with a uniform rate!
    \item In other words, we have found a modulus of continuity $\omega$, independent of $\zz$, such that
    $$K_*-\omega(R^{-1})\leq {\bf M}_R(\zz)\leq K_* \quad \mbox{for any}\quad \zz\in\cZ\,.$$
    But in particular, the density is becoming constant at large scales.
    \item By (the equality case of) the monotonicity formula, as $R$ becomes larger our solution (rescaled) becomes quantitatively closer to some stable minimal cone $C=C_{\zz,R}$, with vertex density $K_*$.
    \item After a rotation, a minimal cone can always be written as $C=\widetilde C\times \R^{n-k}$, where $\R^{n-k}$ is its ``spine'' (i.e. directions of translation invariance), and $\widetilde C\subset \R^k$ has density {\bf strictly} less than the vertex density (i.e. $K_*$) at any given point $x$ outside the origin.
    \item But then, we can apply the ``inductive'' curvature estimates (\Cref{prop:curvestindeps}) in a small neighborhood around any such $x$, giving smooth convergence there -- and thus showing that $\widetilde C\setminus \{0\}$ is smooth.
    
    If $k\neq 2$, Simons' classification directly shows that $C$ is a hyperplane. If $k=2$, \Cref{prop:stablejunction} gives that $C$ is a hyperplane as well.
    
    Finally, \Cref{thm:HutchTon} shows that $K_*$ is an integer.
\end{itemize}
\end{proof}
This noticeably general argument, which reduces the general classification problem to ruling out a critical solution satisfying \eqref{eq:omega_modintro}, would likely lead to similar conclusions in many other Bernstein-type problems. \\

{\bf ``A-priori estimate philosophy''.} We believe it is worth highlighting that our proof of \Cref{mainthm} does \emph{not} use the powerful regularity theory in \cite{Wic14}. In fact, the use of the rich Allen--Cahn varifold theory developed in \cite{HT00, TonegawaWickramasekera2012} can interestingly be entirely bypassed in our article, and we elaborate on this point in \Cref{sec:stabselfcont}.

This renders the theory for stable solutions of the Allen--Cahn equation completely self-contained (not relying on the theory of critical points), with \cite{WW19} understood as a (crucial) component.\\

For later use, we record the following byproduct of the proof (see \Cref{subsec:indlargesc}):
\begin{proposition}[{\bf No gaps}]\label{prop:deltamod2} In the conclusions of \Cref{prop:W}, the following additionally holds:
    \begin{equation}\label{eq:deltamod2}
    K_*-\omega(R^{-1})\leq{\bf M}_{\omega(R^{-1})R}(y)\leq K_*  \quad \mbox{for every} \quad y\in \{e_{\zz,R}\cdot (x-\zz)=0\}\cap B_R(\zz)\,.
\end{equation}
\end{proposition}
\subsection{A tangential form of stability}\label{sec:tanstabover}
We consider a tangential version of the stability inequality (introduced in \cite{Ton05}) which relates the behaviour of the bad set to the flatness. We first need some definitions:
\begin{definition}[\textbf{Tangential gradient and tangential Allen--Cahn 2nd fundamental form}]\label{def:A'intro}
    Let $e\in \Sp^{n-1}$ be a unit vector. We set $\nabla^{e'} u:=\nabla u-(e\cdot\nabla u) e$, and
    \begin{align*}
        \mathcal A_e'^2&=\begin{cases}\frac{|D\nabla^{e'}u|^2-|\nabla|\nabla^{e'} u||^2}{|\nabla u|^2}&\mbox{ if } \nabla^{e'} u\neq 0,\\
        0 &\mbox{ otherwise.}
        \end{cases}
    \end{align*}
Here $D\nabla^{e'}u$ denotes the differential matrix of the vector field $\nabla^{e'}u$.
\end{definition}
\begin{remark}\label{rmk:A'expansion}
    In a frame with $e=e_n$ and $\frac{\nabla^{e'} u}{|\nabla^{e'} u|}(x_0)=e_1$, we have $|\nabla u|^2\mathcal A_{e}'^2(x_0)
        =\sum_{j=2}^{n-1}\sum_{i=1}^n u_{ij}^2(x_0)$.
\end{remark}

\begin{definition}\label{def:excessesintro}
    We set the following dimensionless quantities.
    \begin{itemize}
    \item $L^2$-height excess:
    $${\bf H}_r^2(u,e):=\frac{1}{r^{n+1}}\int_{B_r} (x\cdot e)^2\left[\frac{|\nabla u|^2}{2}+W(u)\right], \quad \mbox{and} \quad \quad {\bf H}_r^2(u):=\inf_{e\in\Sp^{n-1}} {\bf H}_r^2(u,e).$$
    \item $L^2$-tilt excess:
    $${\bf T}_r^2(u,e):=\frac{1}{r^{n-1}}\int_{B_r} \Big(1-(e\cdot \frac{\nabla u}{|\nabla u|})^2\Big)|\nabla u|^2,\quad \mbox{and} \quad {\bf T}_r^2(u):=\inf_{e\in\Sp^{n-1}} {\bf T}_r^2(u,e).$$
    \item $L^2$-tangential curvatures:
    $${\bf K}_r^2(u,e):=\frac{1}{r^{n-3}}\int_{B_r(\zz)} \mathcal A_e'^2(u)|\nabla u|^2, \quad \mbox{and}\quad \quad {\bf K}_r^2(u):=\inf_{e\in\Sp^{n-1}}{\bf K}_r^2(u,e).$$
\end{itemize}
Naturally, ${\bf H}_r(u),{\bf T}_r(u),{\bf K}_r(u)$ denote the corresponding square roots. Moreover, we denote
$${\bf H}_r^2(u,e,x_0):={\bf H}_r^2(u(\cdot-x_0),e),\quad {\bf T}_r^2(u,e,x_0):={\bf T}_r^2(u(\cdot-x_0),e),\quad {\bf K}_r^2(u,e,x_0):={\bf K}_r^2(u(\cdot-x_0),e),
$$
and likewise for ${\bf T}_r^2(u,x_0)$, ${\bf K}_r^2(u,x_0), {\bf H}_r^2(u,x_0)$. We will mostly omit $u$ from the notation, as it will be clear from the context.
\end{definition}

\begin{proposition}[\textbf{Tangential stability inequality}]\label{prop:tanstab2}
    Let $u:\R^n\to\R$ be a stable solution to A--C, and let $e\in\Sp^{n-1}$. Then, there is $C=C(n)$ such that
    \begin{align*}
        {\bf K}_R^2(e)\leq C{\bf T}_{2R}^2(e).
    \end{align*}
\end{proposition}
\begin{proof}[Sketch of proof]
    We test \eqref{stability_1stintro} with $\xi=|\nabla^{e'} u|\eta$, where $\eta$ is a standard cutoff. A detailed proof (including the simple computations which give the final inequality) is given in \Cref{subsec:tanpfview}.
\end{proof}
In our situations of interest, a Caccioppoli-type inequality will show\footnote{See \Cref{sec:excesses} for how the notions of excess that we will use are related.} that
\begin{equation}\label{eq:KL2Hintro}
        {\bf K}_{R}^2(e)\leq C{\bf H}_{4R}^2(e)\,,
    \end{equation}
thus we work exclusively with the latter (we can {\bf forget} about the tilt excess). Here are the main takeaways:
\begin{itemize}
    \item The geometric flatness in \eqref{eq:omega_modintro} immediately shows that ${\bf H}_R^2(e_{\zz,R},\zz)= o(1)$ as $R\to \infty$. Therefore, we have improved the right term from \eqref{eq:SZineq1} to a quantity which decays at infinity!
    \item  However, with a heavy drawback: $\mathcal A_e'^2$ does not control the curvatures of the level sets in the direction of $e$ anymore (see \Cref{rmk:A'expansion}).  This appears to be a major---and unresolved for now---difficulty in most problems about stable solutions, in striking contrast with the very special case of minimal hypersurfaces in which one can recover the full 2nd fundamental form on the left, yielding a stronger form of the stability inequality (Schoen's inequality, which is precisely \cite[Lemma 1]{SS81}).
    \item Since we lack a method to replicate \cite{SS81} (or \cite{SSY75}, which uses Simons' identity instead), we adopt a different strategy. In fact, \eqref{eq:KL2Hintro} will be used exclusively (but crucially) as a way to control the {\bf behaviour of the bad set}.
    \item By definition, each bad balls contributed a definite amount to the left term in stability---we show that each bad ball {\bf still contributes} a definite amount to the left term in \eqref{eq:KL2Hintro}, regardless of the choice of $e$.
\end{itemize}
\begin{proposition}[\textbf{$\mathcal A'$ detects bad balls}]\label{prop:A'badcent}
    Let $u$ be a critical\footnote{Criticality is not needed here: Assuming only $u$ stable and $\int_{B_1}\mathcal A_u^2|\nabla u|^2\geq c_0>0$ instead, we would deduce that $\int_{B_{\frac{1}{\delta}}}\mathcal A_e'^2|\nabla u|^2\geq \delta$ for every $e\in\Sp^{n-1}$, where $\delta=\delta(c_0,n)>0$.} solution, and let $\zz\in\cZ(u)$. Then, there is some $\delta_{bad}'>0$ depending only on $n$ such that 
    $$\int_{B_1(\zz)}\mathcal A_e'^2|\nabla u|^2\geq \delta_{bad}' \quad\mbox{for every}\quad e\in\Sp^{n-1}\,.$$
\end{proposition}
The proof is given in \Cref{subsec:tanpfview}. Essentially, the reason behind the result is that $\mathcal A_e'^2\equiv 0$ implies that $u$ is 2D (and thus 1D by \Cref{thm:2dclas}).

Morally, ``curvature accumulates in all directions''. The proof relies only on the classification of 2D stable solutions, the best known result for many semilinear and free boundary problems, opening up the possibility of adapting our strategy to them. 

An important consequence of \Cref{prop:A'badcent}---combined with \eqref{eq:KL2Hintro}---is the following lower bound for the height excess around bad balls:
\begin{equation}\label{lowerboundheightR}
    C{\bf H}_{4R}^2(\zz) \ge \frac{c}{R^{n-3}}.
\end{equation}
We are now ready to delve into the contradiction argument.
\subsection{Reaching a contradiction -- A full overview of the argument}\label{sec:contradictionintro}
To prove \Cref{mainthm}, we may assume for contradiction that $K_*<\infty$. There exists then a critical solution $u:\R^4\to (-1,1)$, with (nonempty, by \Cref{lem:badnonempty}) bad set $\mathcal Z$ as in \Cref{def:badballintro}.\\
We will show that the existence of such $u$ gives a contradiction (equivalently, the set of subcritical densities is also closed). This may be seen as an ``inductive step'': the fact that any solution with density $<K_*$ must be 1D constrains the critical solution to have a special structure \eqref{eq:omega_modintro}. This information will be crucially used to set up our contradiction argument.

\vspace{3pt}
Interestingly, once we have combined the critical solution, the bad set inspired by Wang--Wei and its control via tangential stability, all of which are new and completely independent from \cite{CFFS25}, our overarching strategy will share striking analogies with the classification of 3D stable solutions to the Bernoulli problem in \cite{CFFS25}. As explained in \Cref{rem:analogies} this analogy is not obvious at all a priori. We choose on purpose notations and write our results in a way that such parallelisms---which would otherwise be difficult to grasp---become as apparent as possible.

\vspace{3pt}

{\bf Overarching strategy.} Perhaps optimistically, one might hope to exploit the closeness of Allen--Cahn level sets to minimal surfaces in order to perform a geometric improvement-of-flatness iteration (\`a la De Giorgi, Allard, or Savin). This would allow us to bring the flatness in \eqref{eq:omega_modintro} from scale $R\gg1$ down to scale one around some fixed $\zz \in \cZ$ and reach a contradiction.

Unfortunately, this naive approach fails. The point is that we have access to enhanced mean curvature bounds \emph{only} on good balls; meanwhile, nothing prevents the bad set from appearing dense along a hyperplane (at large scales). In other words, good balls may be of infinitesimal size relative to the scale under consideration. In such situations, the mean curvature estimates available on these good balls are useless for establishing regularity in the much larger ball of interest.

The previous obstruction also makes an intrinsic approach in the spirit of \cite{SSY75} or \cite{Bel25} essentially hopeless, apart from the fact that we do not have access to Schoen's or Simons' inequalities anyways.

To overcome this, our guiding philosophy will instead be to seek a contradiction from the very existence of the bad set as a whole. 
An outline of our argument in this article is the following:
\begin{itemize}
    \item We consider a carefully selected large bad ball $B_{R_k}(\zz_k)$, with $\ep_k^2:={\bf H}_{R_k}^2(\zz_k)\to 0$.
    \item Using the special properties of {\bf this} ball, we find $R_k^\flat$, with $1\ll R_k^\flat\ll R_k$ as $k\to \infty$, and a new center $\boldsymbol y_k$ such that $B_{R_k^\flat}(\boldsymbol y_k)\subset B_{R_k}(\zz_k)$ and
    \begin{equation}\label{eq:wluigasugs}
        {\bf H}_{R_k^\flat}^2(\boldsymbol y_k)\lesssim \left(R_k^\flat/R_k\right)^{\chi}\ep_k^2\quad \mbox{and}\quad K_*-{\bf M}_{R_k^\flat}(\boldsymbol y_k)\lesssim \left(R_k^\flat/R_k\right)^{\chi}\ep_k^2\qquad\mbox{for some tiny } \chi>0.
    \end{equation}
    \item Via a crucial monotonicity-type inequality, which relates the height excess\footnote{The motivation behind this definition of excess in our article was precisely that we found (a \textit{weighted} version of) it to satisfy this challenging monotonicity-type relation.} with the density pinching, we will bring the improved flatness {\bf back to the original scale $R_k$}, up to a logarithmic error:
    $$
    {\bf H}_{4R_k}^2(\boldsymbol y_k)\lesssim {\bf H}_{R_k^\flat}^2(\boldsymbol y_k)+ \log\big(R_k/R_k^\flat\big)\left[K_*-{\bf M}_{R_k^\flat}(\boldsymbol y_k)\right]\lesssim \left(R_k^\flat/R_k\right)^{\chi}\log\left(R_k/R_k^\flat\right)\ep_k^2\,,
    $$
    so that
    $$
    \ep_k^2={\bf H}_{R_k}^2(\zz_k)\lesssim {\bf H}_{4R_k}^2(\boldsymbol y_k)\lesssim \left(R_k^\flat/R_k\right)^{\chi}\log\left(R_k/R_k^\flat\right)\ep_k^2\,.
    $$
    For $k$ sufficiently large, since $R_k^\flat/R_k\to 0$ this yields a contradiction.
\end{itemize}
Very informally, we improve the flatness of $u$ in $B_{R_k}(\zz_k)$ with respect to itself, which is naturally a contradiction. We now explain our strategy in more detail.

{\bf Selection of center and scale.} Consider the following:
\begin{definition}[{\bf Size of the bad set at resolution $\theta R$}]\label{def:Ntheta}
    Given $\theta\in(0,1]$, define
    \begin{align}
        N(\theta,B_R(\zz)):=&(\theta R)^{-n}\Big|B_{\theta R}(\cZ\cap B_R(\zz)) \Big|\\
        =&(\theta R)^{-n}\Big|\cup_{\widetilde\zz\in \cZ\cap B_R(\zz)} B_{\theta R}(\widetilde\zz) \Big|\,,
    \end{align}
    i.e. essentially the number of balls of radius $\theta R$ needed to cover $\cZ\cap B_R(\zz)$ (by standard covering arguments).
\end{definition}
Set $N_\theta:= N(\theta,B_R(\zz))$. Stability easily gives $\Big|B_{1}(\cZ\cap B_R(\zz)) \Big|\lesssim R^{n-3}$, i.e. $N_{\frac{1}{R}}\lesssim R^{n-3}$. We may then hope at best that $\cZ$ behaves like a codimension-three submanifold of $\R^n$, or codimension-two in $\{u=0\}$. Now, if that were the case, we would have $N_\theta\sim \theta^{-(n-3)}$ for {\bf every} $\theta\in(0,1]$, a ``Minkowski-type'' bound; on the other hand, we only know a ``Hausdorff'' bound $N_{\frac{1}{R}}\lesssim R^{n-3}$, perfectly compatible with the bad set looking essentially $(n-1)$-dimensional at large resolutions $\theta R\gg 1$, leaving no big good regions available.

As explained before, there is then no hope to be able to perform a geometric iteration, and we will need to improve the excess via another argument. An instructive observation is the following: Assume that $N_\theta\geq \theta^{-(n-3+\beta)}$; by Vitali, we find $\sim N_\theta$ disjoint bad balls $B_{\theta R}(\zz_i)\subset B_R(\zz)$. Then, we can bound
$$
\sum_{i=1}^{\sim N_\theta} {\bf K}_{\theta R}^{2}(\zz_i)\leq \inf_{e\in\Sp^{n-1}}\sum_{i=1}^{\sim N_\theta} \frac{1}{(\theta R)^{n-3}}\int_{B_{\theta R}(\zz_i)} \mathcal A_e'^2|\nabla u|^2\leq \inf_{e\in\Sp^{n-1}}\frac{1}{(\theta R)^{n-3}}\int_{B_{R}(\zz)} \mathcal A_e'^2|\nabla u|^2= \theta^{-(n-3)}{\bf K}_{R}^{2}(\zz)\,;
$$
in particular, since $N_\theta\geq \theta^{-(n-3+\beta)}$, there is at least one $\zz_j$ in the sum such that ${\bf K}_{\theta R}^{2}(\zz_j)\leq  \theta^{\beta}{\bf K}_{R}^{2}(\zz)$.

In other words, we have improved our tangential curvatures with an algebraic rate! Changing the center may seem nonstandard in this class of problems---the point is that $\zz_j$ is a bad center again, just as good as $\zz$ to continue our contradiction argument thanks to the uniformity in \eqref{eq:omega_modintro}.\\

To transform this into an excess decay, we would need a sort of ``reverse stability relation''. This motivates (see \Cref{sec:seleccentrscale} for the short proofs):

\begin{lemma}[\textbf{Selection 1: Choice of $R_k$ and $\zz_k$}]\label{lem:radcentselec}
Let $\alpha\in(0,\frac{1}{4})$. There exist $R_k\to\infty$ and $\zz_k\in\cZ$ such that
    $$\ep_k^2:={\bf H}_{4R_k}^2(\zz_k)\to 0$$
    and
    \begin{equation}\label{eq:geiweghio}
        {\bf H}_{4R}^2(\zz)\leq 2\left(\frac{{\bf K}_R^{2}(\zz)}{{\bf K}_{R_k}^{2}(\zz_k)}\right)^{1+\alpha}\ep_k^2\quad \mbox{for every }\zz\in\cZ \mbox{ and } R\in[R_0, R_k].
    \end{equation}
\end{lemma}
We obtain this by essentially choosing $z_k$ and $R_k$ which maximise $\frac{{\bf H}_{4R}^2(\zz)}{{\bf K}_R^{2(1+\alpha)}(\zz)}$, which is a slightly penalised (by $\alpha>0$) version of the tangential stability inequality, among all $\zz\in\cZ$ and $R\leq R_k$.\\

From the discussion above, it is then natural to consider, for a fixed $\beta\in (0,1)$:
\begin{equation}\label{eq:oirghlvbjintro}
    \widetilde\theta_k\simeq\inf \left\{\theta\mbox{ : exists } \zz\in \cZ \mbox{ with } B_{\theta R_k}(\zz)\subset B_{R_k}(\zz_k), {\bf K}_{\theta R_k}^2(\zz)\leq\theta^{\beta}{\bf K}_{R_k}^2(\zz_k)\right\},
\end{equation}
so that we find new $\Rk:=\widetilde \theta_k R_k$ and $\zk$ with (by \eqref{eq:geiweghio}) the desired excess decay ${\bf H}_{4\Rk}^2(\zk)\leq \widetilde\theta_k^{\beta(1+\alpha)}\ep_k^2$. On top of that, the bound we wanted for $N_\theta$ will actually now hold\footnote{As otherwise the argument before the lemma would allow the curvature decay to continue.} relative to $B_{\Rk}(\zk)$. The precise statement is:
\begin{lemma}[\textbf{Selection 2: Passing to $\Rk$ and $\zk$}]\label{lem:NbdRknew}
    Let $R_k,\zz_k,\ep_k$ be given by \Cref{lem:radcentselec}, and fix $\beta\in(0,1)$. For $k$ large enough, there exist $\widetilde\theta_k\in(\frac{R_0}{R_k},1]$ and $\zk\in \cZ\cap B_{R_k}(\zz_k)$ such that, putting $\Rk:=\widetilde \theta_k R_k$ and $\epk:=\widetilde\theta_k^{\beta(1+\alpha)}\ep_k$, the following hold:
    \begin{itemize}
        \item $B_{\Rk}(\zk)\subset B_{R_k}(\zz_k)$ and $\Rk\to\infty$.
        \item ${\bf H}_{4\Rk}^2(\zk)\leq 2\epk^2\to 0$.
        \item {\bf Excess bound:} \begin{equation}\label{eq:hjgagfas}
        {\bf H}_{4R}^2(\zz)\leq 2\left(\frac{{\bf K}_R^{2}(\zz)}{{\bf K}_{\Rk}^{2}(\zk)}\right)^{1+\alpha}\epk^2\leq 2\left(\frac{\Rk}{R}\right)^{(n-3)(1+\alpha)}\epk^2\quad \mbox{for any } \zz\in\cZ,\, R\geq R_0\ \mbox{ s.t. }B_R(\zz)\subset B_{\Rk}(\zk).
    \end{equation}
    \item {\bf Bad ball count:}
    \begin{equation}\label{eq:fhiowqrhgoqwg}
        N(\theta,B_{\Rk}(\zk))\leq C\theta^{-(n-3+\beta)}\quad \mbox{for all}\quad  \theta\in(\frac{R_0}{\Rk},1].
    \end{equation}
    \end{itemize}
\end{lemma}
Observe that we carry over \eqref{eq:hjgagfas} too; this will be just as important as \eqref{eq:fhiowqrhgoqwg}. This concludes \Cref{sec:seleccentrscale}.\\

{\bf Linearisation and improvement of flatness.} It is at this second center and scale that our setting will allow us to perform a linearisation procedure. Let us restrict to $n=4$ in what follows. The main steps will be:

\begin{enumerate}
    \item[(i)] By \eqref{eq:hjgagfas}, which guarantees flatness at many scales, we will be able to cover almost all of $\{u = 0\} \cap B_{\Rk/4}(\zk)$—except for a small neighbourhood of the bad set—by a union of $K_*$ graphs $x_n = \widetilde g_i(x')$, with $\widetilde g_i: B_{\Rk/4}'(\zk') \to \R$ (in suitable Euclidean coordinates after rotation). The prime notation $'$ will be used throughout to denote objects in $\R^{n-1}$, the first $n-1$ coordinates of a point, etc.---see more comments on this notation below.    
    
    We will then consider, for each $i=1,\dots, K_*$, the functions
    $$ h_i (x') : = \frac{\widetilde g_i(\zk'+ \Rk x')-(\zk)_n}{\Rk \epk}\, ,   \quad \mbox{which have }L^1\mbox{ norm of size }1\mbox{ in } B'_{1/4}.$$
    
    \item[(ii)] A crucial step will be to show that the mean curvatures of the $h_i$ are much smaller---in fact, bounded (in $L^1$) by a positive power of $\epk$. This will combine all the special properties of the pair center-scale $(\zk,\Rk)$ and of course will strongly rely on \cite{WW19}. Then, this result will imply (via a standard iteration), that at any $\bar\zz\in \cZ\cap B_{\Rk}$ we can find $K_*$ different affine functions $\bar \ell_i$  such that \begin{equation}\label{eq:quiefabfoab}| h_i-\bar \ell_i|
    \quad \mbox{is of size $O(\epk^{3\chi/2})$ on average over } B'_{\epk^\chi} (\bar \zz'/\Rk)\,,
    \end{equation}
    for some tiny exponent $\chi>0$.
    
    Up to scaling back, this amounts to a small improvement of flatness for each layer $\widetilde g_i$. Notice however that, a priori, there would be no reason why the $\bar \ell_i$ should be similar to one another (for example, the $\widetilde g_i$ could have all been linear functions themselves to begin with).
    
    \item[(iii)] Finally, by carefully leveraging the properties obtained from the continuous induction argument---more precisely, using \eqref{eq:omega_modintro}, which forces all layers of $\{u = 0\}$ to lie close to each other when viewed inside a sufficiently large ball centered at {\it any} bad center---we will show that the $\bar \ell_i$ must in fact be very close to one another.
 More precisely, we will find a {\it single} affine function $\ell$ (for example one can take $\ell:= R_k\bar \ell_1(\zk'+  \,\cdot\,)$) such that:
\begin{equation}\label{eq:quiefabfoab2} \frac{|\tilde g_i-\ell|}{\epk^\chi\Rk}\quad\mbox{is of size $O(\epk^{\chi/2})$ on average over}\  B'_{\epk^\chi\Rk}(\bar\zz)\,,\quad\mbox{for every}\quad i\in\{1,...,K_*\}\,.
\end{equation}
    This {\it collective} improvement of flatness of the layers will then be upgraded to an improvement on ${\bf H}^2$ and $K_*-{\bf M}$, leading to \eqref{eq:wluigasugs} and thus the final contradiction.
\end{enumerate}
    
\noindent We now explain in more detail the steps above. Throughout the paper, we use the prime notation~$'$ to denote projections of sets\footnote{For example, if $A \subset \R^n$ we set
$A' := \{x' \in \R^{n-1} \ :\ (x',t) \in A \ \text{for some } t \in \R\}$.}, points, and balls onto $\R^{n-1}= \R^3$.

Given $\lambda>0$, we set
\begin{equation}\label{eqdef:Omega2-gamm}
    \Omega_{\lambda}:=\{{\rm dist}(x',[\cZ\cap B_{\Rk}(\zk)]')\geq \epk^{\lambda}\Rk\}\cap B_{\frac{1}{2}\Rk}'(\zk')\subset\R^{n-1}\,.
\end{equation}
Notice that the set $\Omega_{\lambda}$ clearly depends on $k$, but we omit this dependence in the notation for readability, as there is no risk of confusion.

Throughout, $o_k(1)$ denotes a (positive) quantity that can be made arbitrarily small by taking $k$ sufficiently large (independently of all other variables involved in the statements).

Our starting point towards point (i) described above is:
\begin{proposition}[{\bf Graphical decomposition}]\label{prop:graphdec}
Let $\gamma : =1/4$.
For any given $k$, let us choose a Euclidean coordinate frame\footnote{This frame always exists thanks to the second bullet in \Cref{lem:NbdRknew}.} such that ${\bf H}_{4\Rk}^2(e_n,\zk)\leq 2\epk^2$. Then---for $k$ sufficiently large---there are $K_*$ smooth graphs $g_i:\Omega_{2-\gamma/2}\to \R$, $g_1<...<g_{K_*}$, such that
\begin{equation*}
        \{u=0\}\cap B_{\frac{1}{2}\Rk}(\zk)\cap (\Omega_{2-\gamma/2}\times\R)=\bigcup_{i=1}^{K_*} {\rm graph} \,g_i\,.
\end{equation*}
Furthermore, recalling that $\HH[f]:={\rm div}(\frac{\nabla f}{\sqrt{1+|\nabla f|^2}})$, we have
    \begin{equation}\label{eq:uagfabaf}
         |\nabla g_i|\le o_k(1)\quad \mbox{and}\quad |\HH[g_i]|(x')\leq \frac{C}{{\rm dist}^2(x',[\cZ\cap B_{\Rk}(\zk)]')}\,.
    \end{equation}
\end{proposition}
By \eqref{eq:hjgagfas} it can be seen that
$$\frac{1}{\Rk}\|g_i-(\zk)_n\|_{L^\infty(\Omega_{2-\gamma/2})}\lesssim \epk \quad\ \ (\mbox{i.e. we have initial collective flatness}\ \epk)\,.$$
As explained in (ii) above, we would like to say that these graphs look minimal/harmonic, to perform an iteration that improves this flatness. On the other hand, they are only defined on a ``punctured domain'' $\Omega_{2-\gamma/2}$, and with estimates from \eqref{eq:uagfabaf} degenerating as we approach the (projected) bad set. This forces us to work with:
\begin{itemize}
    \item $L^1$ measures of smallness and minimality/harmonicity, instead of $L^\infty$.
    \item New, delicate Whitney-type extensions $\widetilde g_i$ of our graphs (performed in \Cref{sec:extension}), with $\widetilde g_i\equiv g_i$ in $\Omega_{2-\gamma}$, which capture the {\it integral} information coming from \eqref{eq:hjgagfas}.
\end{itemize}

Let $
h_i(x') := \frac{\widetilde{g}_i(\zk' + \Rk x') - (\zk)_n}{\epk \Rk}.$
In \Cref{sec:linHeq} we will establish precise local $L^1$ bounds for $|\Delta h_i|$.  
In particular, we will prove---together with other, more refined estimates that are also needed but omit here---that
\[
\int_{B_{1/8}'} |\Delta h_i| = O\big(\epk^{1/10}\big).
\] Then,  given $\bar\zz\in B_{\frac{1}{8}\Rk}(\zk)$, a simple iteration then gives affine functions $\ell_i:\R^3\to \R$ so that
\begin{equation}\label{eq:ygagvbogao}
    \frac{1}{(\epk^\chi\Rk)}\fint_{B_{(\epk^\chi\Rk)}'(\bar\zz')} |\widetilde g_i-\ell_i|\lesssim\epk^{1+\chi/2}\,,
\end{equation}
i.e.  \eqref{eq:quiefabfoab} up to rescaling. 

To upgrade \eqref{eq:ygagvbogao} to \eqref{eq:quiefabfoab2}---which is point (iii) above---we will proceed as follows:
\begin{itemize}
    \item {\bf An $L^\infty$ idealisation:} Define $B^1 := B_{\epk^\chi \Rk}(\bar{\zz})$ and $B^2 := B_{\epk^{1+2\chi} \Rk}(\bar{\zz}) \subset B^1$. Set $\bar{\ep}_k := \epk^{1+\chi/2} (\epk^\chi \Rk)$. Let us imagine that we had \eqref{eq:ygagvbogao} in $L^\infty$ form rather than merely on average. For this idealisation, let us also forget about the distinction between $g_i$ are $\widetilde g_i$, thus pretending that they are both defined in the full domain without holes.
  We would then have: \begin{equation}\label{eq:ygagvbogaoinfty}
        |g_i-\ell_i|\lesssim \bar\ep_k\quad \mbox{in}\ {B'}^1,\quad \mbox{thus}\quad |g_i-\ell_i|\lesssim \bar\ep_k\quad \mbox{in}\ {B'}^2\subset {B'}^1\ \mbox{as well}.
    \end{equation}
    Moreover, by \eqref{eq:omega_modintro} all of the $g_i$ pass inside $B^2$; in particular, $|g_j-g_i|\lesssim {\rm diam}(B^2)\lesssim \bar\ep_k$ in ${B'}^2$.
    
    By the above and the triangle inequality, $|\ell_j-\ell_i|\lesssim \bar\ep_k$ in ${B'}^2$ too. Since the $g_i$ can't cross each other, the $\ell_i$ (which are affine) are forced to be always very close, and then $|\ell_j-\ell_i|\lesssim \bar\ep_k$ in ${B'}^1$ as well, as desired.
    
    \item {\bf Our case ($L^1$ with $L^\infty$ scaling):} As we will see, $B^2\setminus \Omega_{2-\gamma}$ is actually tiny, thus the distinction between $g_i$ and $\widetilde g_i$ is only a minor technical issue. On the other hand, compared to \eqref{eq:ygagvbogaoinfty},
    the implication \begin{equation}\label{eq:ygagvbogaoL1}
        \fint_{{B'}^1}|\widetilde g_i-\ell_i|\lesssim \bar\ep_k \quad \Longrightarrow \quad \fint_{{B'}^2}|\widetilde g_i-\ell_i|\lesssim \bar\ep_k
    \end{equation}
    is not automatic and hard to establish.
    
    We will show this implication by running (in \Cref{prop:decaysqrt3}) a second iteration (a geometric improvement of oscillation); however, this will require $|\Delta h_i|$ to be substantially smaller than what is needed to prove \eqref{eq:ygagvbogao}. This extra smallness ---obtained in \Cref{prop:linHeq}--- happens to break a natural criticality of the problem and must be obtained through a delicate  dichotomy argument.
    Once the implication \eqref{eq:ygagvbogaoL1} is established, arguing just as above for the rest we find that $|\ell_j-\ell_i|\lesssim \bar\ep_k$ in ${B'}^1$ for all $1\le i,j\le K_*$.
\end{itemize}
Combining the above, letting $\ell:=\ell_1$ we will find:
\begin{proposition}[{\bf Improvement of flatness at scale $\epk^\chi\Rk$}]\label{prop:impflatsinglin}
Let $n=4$ and fix $\chi\in(0,\frac{1}{20}]$, $\beta\in(0,\frac{1}{40}]$ and $\alpha\in(0,\frac{1}{40}]$. Given any $\bar \zz\in \cZ\cap B_{\Rk/8}(\zk)$, there exists---for $k$ large enough---an affine function $\ell:\R^{n-1}\to\R$ such that
    \begin{equation}
    \label{eq:flatlaystar}
    \frac{1}{(\epk^\chi\Rk)|B_{\epk^\chi\Rk}'(\bar\zz')|} \int_{B_{\epk^\chi\Rk}'(\bar\zz')\cap\Omega_{2-\gamma}} \big| g_i - \ell \big| \le C\epk^{1+\chi/2}  \qquad \mbox{for every}\quad i\in\{1,...,K_*\},
    \end{equation}
where $C$ depends only on $\chi$, $\beta$, and $\alpha$.
\end{proposition}

\Cref{sec:improvlayers} is devoted {\it exclusively} to proving this result, following the strategy we just outlined. In \Cref{sec:conclusion} we use \Cref{prop:impflatsinglin} to conclude our contradiction argument:
\vspace{0.3cm}

{\bf Improvement of Allen--Cahn excess and density.} Choosing $\bar \zz$ in \Cref{prop:impflatsinglin} to be at the ``boundary'' of the bad set, we find a clean (meaning free of bad centers) ball $\widetilde B^1(\mathbf{\bar y})\subset B^1$ of comparable size, where we have uniform elliptic estimates. Then \eqref{eq:flatlaystar} transforms into an $L^\infty$ estimate, i.e. $\frac{1}{(\epk^\chi\Rk)}\|g_i-\ell\|_{L^\infty(\widetilde {B'}^1(\mathbf{\bar y}))}\lesssim \epk^{1+\chi/2}$. Since $\{u=0\}$ is given by the $K_*$ graphs $g_i$ in this ball, this will lead to\footnote{The bound on ${\bf H}$ is easy, since we will have exponential decay away from $\{u=0\}$. To go from the large concentration of $\{u=0\}$ to a lower bound on ${\bf M}$, we will combine a slicing argument and the behaviour of 1D periodic Allen--Cahn solutions (\Cref{sec:1DAC}).}

\[
{\bf H}_{(\epk^\chi\Rk)/4}^2({\bf \bar y}) \lesssim \epk^{2+  2\cttc/3}\,,\quad\mbox{ and moreover}\quad K_*-\epk^{2+2\chi/3}\lesssim {\bf M}_{(\epk^\chi\Rk)/4}({\bf \bar y})\leq K_*\,.
\]
\vspace{0.3cm}
    
{\bf A monotonicity-type relation and the final contradiction.} By the monotonicity formula, the above density pinching should mean $u$ is very conical. In \Cref{prop:moncontrexc} we exhibit a new, delicate monotonicity-type relation which {\bf quantifies} this. Letting ${\bf \widetilde H}$ and ${\bf \widetilde M}$ be the variants\footnote{Interestingly, we obtain the challenging formula for weighted (by a heat-type kernel) versions of ${\bf H}$ and ${\bf M}$, which will be just as good as the original versions (see \Cref{lem:compweighHM}) for our purposes.} of ${\bf H}$ and ${\bf  M}$ defined in \Cref{def:heathexc}, it says that
    \begin{equation}\label{excmonineqintro}
        {\bf \widetilde H}_r\leq {\bf \widetilde H}_{\lambda r}+C|\log(\lambda)|^{1/2}({\bf \widetilde M}_r-{\bf \widetilde M}_{\lambda r})^{1/2}\quad\mbox{for any}\ \ r>0\ \ \mbox{and}\ \ \lambda\in(0,\frac{1}{2})\,.
    \end{equation}
This inequality, which may be viewed as an outward epiperimetric-type relation, allows us to transport the improved height excess {\bf back} to the original $B_{R_k}(\zz_k)$ from Selection 1 (see \Cref{lem:radcentselec}). This improves ${\bf H}_{4R_k}(\zz_k)$ with respect to itself, reaching a contradiction as in the brief overview at the beginning of the section.

\subsection{Obstructions to the proof in higher dimensions}\label{obstructionshd}
Most results have been carefully optimised so that they hold in higher dimensions. The main obstruction to extending our proof comes from the mean curvature estimate in \Cref{thm:sheetimpc2alpha} (which is {\bf optimal} as stated).

Indeed, even in the very idealised case where one knew a priori that there is \textit{only one single bad ball}, say centered at 0, a natural obstruction arises in dimensions $ 5\le n \le 7$.

Here is an overview in this idealised setting: Using that ${\bf H}_{4\Rk}^2(\zk)\leq \epk^2$ and applying \eqref{lowerboundheightR}, we find a lower bound for the flatness in $B_{4\Rk}(\zk)$ for our solution: namely,
\begin{equation}\label{eq:7yhkagabshgbf}
    \epk^2\geq \frac{c}{\Rk^{n-3}}\,,\quad\mbox{or}\quad \epk^{2/(n-3)}\geq \frac{c}{\Rk}\,.
\end{equation}
This bound is moreover {\it sharp} in the case in which $\cZ\cap B_{4\Rk}(\zk)$ consisted exclusively of a single bad center, i.e. $\zk$.

Now, exactly as in (i) above, consider the rescalings
$$ h_i (x') : = \frac{\widetilde g_i(\zk'+ \Rk x')-(\zk)_n}{\Rk \epk}\, ,   \quad \mbox{which have }L^1\mbox{ norm of size }1\mbox{ in } B'_{1/4}\,.$$
To perform an improvement of flatness iteration---even just for a small number of scales---we need the mean curvatures of the $h_i$ to be small, as we described in (ii) above. More precisely, the mean curvatures should be bounded (in $L^1$) by some positive power of $\epk$.

Naturally, the mean curvature bound we have access to is the one in \Cref{thm:sheetimpc2alpha}, or \eqref{eq:uagfabaf} in our setting. Rescaled to the $h_i$, it says {\it at best} that
$$\int_{B_{1/4}'} |H[h_i]|\quad \mbox{is of size } O\left(\frac{1}{\Rk \epk}\right).$$
Combined with the bound \eqref{eq:7yhkagabshgbf}, we find then that
$$\int_{B_{1/4}'}|H[h_i]|\leq C\frac{\epk^{2/(n-3)}}{\epk}=C\epk^{\frac{5-n}{n-3}}\,.$$
Already for $n=5$ we do not obtain a positive power of $\epk$ anymore. This means that the mean curvature estimates on the $h_i$ are simply too large to allow for an improvement of flatness iteration.

It would appear then that, in order to overcome this obstruction, one cannot forget the more precise information of where \eqref{eq:meancurvsheet} comes from: namely, a Toda-type elliptic system governing the interaction between layers (see \cite{WW19}).\\

\section{Key preliminary results}\label{sec:oviewproofs}
\subsection{Regularity results -- Proofs of Theorems \ref{prop:curvestindeps} and \ref{thm:goodimpsheet}}\label{subsec:regpfview}
\begin{proof}[Proof of \Cref{prop:curvestindeps}]
Notice that the statement is scaling invariant. Hence, by considering the rescaled function $u(R\,\cdot\,)$, we can (and do) assume that $R=1$. 
    
\noindent \textbf{Step 1.} We first prove the lower bound for $|\nabla u_\ep|$.\\
    Assume for contradiction there are $\delta_{1,i},\delta_{2,i},\ep_i\to 0$ such that:\\
    We have
    \begin{equation}\label{eq:fiowgopgagop}
        {\bf M}_1^{\ep_i}(u_{\ep_i})\leq K+\delta_{1,i}\,,
    \end{equation}
    yet
    \begin{equation}\label{eq:fiowgopgagop2}
        \mbox{there is some}\quad x_i\in B_{\delta_{2,i}}\cap \{|u_{\ep_i}|\leq 0.9\} \quad\mbox{with}\quad|\nabla u_{\ep_i}|(x_i)<\frac{1}{i\ep_i}\,.
    \end{equation}
    Since $B_{1-\delta_{2,i}}(x_i)\subset B_1$, by \eqref{eq:fiowgopgagop} we can bound
    \begin{equation}\label{eq:figiuewhiu}
        {\bf M}_{1-\delta_{2,i}}^{\ep_i}(u_{\ep_i},x_i)=\frac{\mathcal E^{\ep}(u_{\ep_i},B_{1-\delta_{2,i}}(x_i))}{(1-\delta_{2,i})^{n-1}}\leq \frac{\mathcal E^{\ep}(u_{\ep_i},B_1)}{(1-\delta_{2,i})^{n-1}}\leq\frac{K+\delta_{1,i}}{(1-\delta_{2,i})^{n-1}}\,.
    \end{equation}
    Consider the rescalings $\widetilde u_i(x):=u_{\ep_i}(\ep_i(x-x_i))$. Letting $R_i=\frac{(1-\delta_{2,i})}{\ep_i}$, the $\widetilde u_i$ are now A--C solutions with parameter $1$ defined on $B_{R_i}$. Moreover, by \eqref{eq:figiuewhiu} and monotonicity we find that
    \begin{equation}\label{eq:781halavifa}
        {\bf M}_r(\widetilde u_i)\leq \frac{K+\delta_{1,i}}{(1-\delta_{2,i})^{n-1}} \quad \mbox{ for all } r\leq R_i\,,
    \end{equation}
    yet by \eqref{eq:fiowgopgagop2} we have $|\widetilde u_i|(0)\leq 0.9$ and $|\nabla \widetilde u_i|(0)<\frac{1}{i}$.
    Since $|\widetilde u_i|\leq 1$ and they all satisfy \eqref{eq:aceqintro}, standard $C^3$ estimates and Arzel\`a-Ascoli provide a subsequence converging in $C^2_{\rm loc}(\R^n)$ to $\widetilde u_\infty$, a bounded stable solution to A--C on all of $\R^n$, such that
    $${\bf M}_\infty(\widetilde u_\infty)=\lim_{r\to\infty} {\bf M}_r(\widetilde u_\infty)\leq K\qquad\mbox{yet}\qquad |\widetilde u_\infty|(0)\leq 0.9\quad\mbox{and}\quad|\nabla \widetilde u_\infty|(0)=0\,.$$
    On the other hand, since we are assuming that $K$ is a subcritical density, together with the fact that $|\widetilde u_\infty|(0)\leq 0.9<1$ we deduce that $\widetilde u_\infty=\phi(a\cdot x+b)$ for some appropriate $a,b$, but then $|\nabla \widetilde u_\infty|(0)=\phi'(b)\neq 0$, a contradiction.
    
    \noindent \textbf{Step 2.} We prove the bound for $\mathcal A_{u_{\ep}}$; \eqref{lvl2ffind} follows then directly.\\
    Recall that we are assuming $R=1$; we will find $C$ and $r\in(0,1)$, depending on $K$, such that
    $$
   \mathcal A_{u_{\ep}}\leq \frac{C}{r} \quad\mbox{in}\quad \{|u_{\ep}(x)|\leq 0.9\}\cap B_{r}\,,$$
   which is equivalent to proving the statement.\\
   We show this by contradiction. Assume there were $\delta_{1,i},\delta_{2,i},\ep_i\to 0$ such that ${\bf M}_1^{\ep_i}(u_{\ep_i})\leq K+\delta_{1,i}$\,, yet
    $$
   \delta_{2,i}\mathcal A_{u_{\ep_i}}(x_i)\to\infty \quad \mbox{ for some } x_i\in \{|u_{\ep_i}|\leq 0.9\}\cap B_{\delta_{2,i}}\,,$$
   which in particular gives that
   $$
   \sup_{x\in \{|u_{\ep_i}|\leq 0.9\}\cap B_{2\delta_{2,i}}}{\rm dist}(x,\partial B_{2\delta_{2,i}})|\mathcal A(u_i)|(x)\to\infty\,.
   $$
   Let $y_i\in \{|u_{\ep_i}|\leq 0.9\}\cap B_{2\delta_{2,i}}$ be such that 
   \begin{equation}\label{eq:phpaogpabg}
       2{\rm dist}(y_i,\partial B_{2\delta_{2,i}})|\mathcal A(u_i)|(y_i)\geq \sup_{x\in B_{2\delta_{2,i}}} {\rm dist}(x,\partial B_{2\delta_{2,i}})|\mathcal A(u_i)|(x)\,,
   \end{equation}
    and set
    $$R_i:={\rm dist}(y_i,\partial B_{2\delta_{2,i}}),\quad A_i:=|\mathcal A(u_i)|(y_i)\quad\mbox{and}\quad\widetilde\ep_i=\frac{\ep_i}{\mathcal A_{u_{\ep_i}}(x_i)}\,.$$
    Defining $\widetilde u_i(x):=u_{\ep_i}(y_i+\frac{1}{A_{u_{\ep_i}}(y_i)}x)$, we now have solutions with A--C parameters $\widetilde\ep_i$, defined on domains $B_{\frac{r_i}{2}A_i}$ converging to $\R^n$ (since $r_iA_i\to\infty$), and which (by the same computation as in Step 1, since $y_i\in B_{2\delta_{2,i}}$) satisfy
    \begin{equation}\label{eq:781halavifa2}
    {\bf M}_r^{\widetilde \ep_i}(\widetilde u_i)\leq \frac{K+\delta_{1,i}}{(1-2\delta_{2,i})^{n-1}} \quad \mbox{ for all } r\leq R_i\,.
    \end{equation}
    \textbf{Case 1.} We have $\widetilde\ep_i\to \widetilde\ep_\infty\in(0,\infty]$. The functions $\widetilde u_i$ then satisfy elliptic estimates, which (since $\mathcal A_{\widetilde u_{i}}(0)=1$) shows that necessarily $\widetilde\ep_\infty\in(0,\infty)$. But then, the same argument as in Step 1 combined with the fact that $\mathcal A_{\phi(a\cdot x+b)}\equiv 0$ yields a contradiction.
    
    \noindent\textbf{Case 2.} Otherwise, $\widetilde\ep_i\to 0$. Now, observe that we have $\mathcal A(\widetilde u_i)(0)=1$ and (thanks to \eqref{eq:phpaogpabg}) also $\mathcal A(\widetilde u_i)\leq 4$ in $B_{\frac{r_i}{2}A_i}$.
    
    Together with Step 1, this means precisely that the $\widetilde u_i$ satisfy the sheeting assumptions from \Cref{def:sheetassump} in balls of radius one, and therefore we have the conclusions of \Cref{thm:sheetimpc2alpha} as well. This gives uniform $C^{2,\vartheta}$ estimates for some fixed $\theta>0$, say $\theta=\frac{1}{2}$, which (by Ascoli--Arzel\`a) shows that the level sets $\{\widetilde u_i=\widetilde u_i(0)\}$ converge (up to passing to a subsequence, not relabeled) in $C^{2,\vartheta}_{loc}(\R^n)$, to a complete minimal hypersurface $\Sigma$. By \Cref{thm:TonStable}, it is stable; moreover, as in Step 1 using \eqref{eq:varenconv} we find that $\mathcal H^{n-1}(\Sigma\cap B_R)\leq CC_0R^{n-1}$ for every $R>0$. Then, by the stable Bernstein theorem with Euclidean area growth in $\R^n$ with $n\leq 7$, see \cite{SS81}, we deduce that $\Sigma$ is a union of parallel hyperplanes, and in particular $|{\rm I\negthinspace I}_{\Sigma}|\equiv 0$. By the $C^{2,\vartheta}$ estimates once again, which imply convergence in $C^2_{loc}$, we deduce that actually $\sup_{B_1}|{\rm I\negthinspace I}_{\{\widetilde u_i=\widetilde u_i(0)\}}|\to 0$.
    
    It then follows (see \cite{WW18,WW19} for more details) that
    $$\sup_{x\in\{\widetilde u_i=\widetilde u_i(0)\}\cap B_{1/2}}|\mathcal A(\widetilde u_i)|(x)\to 0.$$
   On the other hand, we had $|\mathcal A(\widetilde u_i)|(0)=1$, a contradiction for $i$ large enough.
\end{proof}

\begin{proof}[Proof of \Cref{thm:goodimpsheet}]
    This is essentially a variant of \cite[Corollary 1.3]{WW19}. Following exactly the same contradiction argument as in the previous proof, in Step 1 we would get solutions with the additional assumption that $\int_{B_1(x_i)}\mathcal A(u_{\ep_i})^2|\nabla u_{\ep_i}|^2\to 0$, so that the limit $\widetilde u_\infty$ would satisfy $\int_{B_1}\mathcal A(\widetilde u_{\infty})^2 |\nabla u_{\infty}|^2=0$. Then $\widetilde u_{\infty}$ would be one-dimensional by \Cref{lem:A0imp1D}, thus either $\pm 1$ or of the form $\phi(a\cdot x+b)$ by \Cref{prop:1dclas}, reaching a contradiction exactly as in the previous proof. In Step 2, the proof of Case 1 would follow this exact same reasoning, and Case 2 did not use the subcritical density assumption anyways.
\end{proof}

\subsection{Large-scale flatness -- Proofs of Propositions \ref{prop:W} and \ref{prop:deltamod2}}\label{subsec:indlargesc}

\begin{proof}[Proofs of Propositions \ref{prop:W} and \ref{prop:deltamod2}]
{\bf Step 1.} Density at large scales.\\
We first show that $K_*-\omega(\frac{1}{R})\leq{\bf M}_{R}(\zz)\leq K_*$. The second inequality holds by monotonicity. For the first one, let $K<K_*$; we want to see that there is some universal $R_0=R_0(K)$ such that ${\bf M}_{R}(\zz)> K$ if $R\geq R_0$. But indeed, by definition of $K_*$ we know that $K$ is subcritical, and in particular we can apply \Cref{prop:curvestindeps} (with $\eps=1$), from which it follows that if ${\bf M}_{R}(\zz)\leq K$ then $\mathcal A_u\leq \frac{C}{R}$ in $B_1(\zz)$ (as long as $R\geq R_0(K)$). 
On the other hand $\int_{B_1(\zz)}\mathcal A_u|\nabla u|^2\geq \delta_{bad}>0$ by assumption, and $|\nabla u|\leq C$ is bounded.
This implies that necessarily ${\bf M}_{R}(\zz)>K$ up to making $R_0(K)$ large enough, as desired.\\

We are ready to prove the rest of Propositions \ref{prop:W} and \ref{prop:deltamod2}. We assume then for contradiction that there were some $\delta_1>0$, and some $R_i\to\infty$, $u_i$ and $\zz_i$ as in the statement, but such that for any given $e\in \mathbb S^{n-1}$ either
\begin{equation}\label{eq:8qtuhabgbao}
	\{|u_i|\leq 0.9\}\cap B_{R_i}(\zz_i)\not\subset\{|e\cdot (x-\zz_i)|\leq \delta_1 R_i\}
\end{equation}
or
\begin{equation}\label{eq:8qtuhabgbao2}
{\bf M}_{\delta_1 R}(u_i,y)<K_*-\delta_1  \quad \mbox{for some} \quad y\in \{e\cdot (x-\zz_i)=0\}\cap B_{R_i}(\zz_i)\,.
\end{equation}

\noindent\textbf{Step 2.} Rescaling and cone analysis.\\
    Rescaling by $\frac{1}{R_i}$, we get solutions $\widetilde u_i(x):=u_i(\zz_i+R_i x)$ of $\widetilde\ep_i$-A--C with parameters $\widetilde \ep_i=\frac{1}{R_i}$. By \Cref{thm:HutchTon} and \Cref{thm:TonStable}, they will converge---up to subsequence---to a limit stationary stable integral varifold $V=(\Sigma,\theta)$. We will write $\Sigma$ instead of $V$ and ${\rm supp}\, V$, by slight abuse of notation; our simple arguments will be of very standard nature. Let $\mathbb M_R(\Sigma,y):=\frac{1}{\omega_{n-1}R^{n-1}}\|\Sigma\|(B_R(y))$, and $\mathbb M_R(\Sigma):=\mathbb M_R(\Sigma,0)$. Using \eqref{eq:varenconv} and Step 1 we see that $\mathbb M_{R}(\Sigma)=\lim_{i\to\infty}{\bf M}_{RR_i}(u_i)\equiv K_*$. Thus, by the monotonicity formula for stationary integral varifolds, $\Sigma$ will be a cone, meaning that $\lambda\Sigma=\Sigma$ for any $\lambda>0$. We want to see that $\Sigma$ is a hyperplane.\\
    
    All points in $\Sigma$ have density at most $K_*$, by upper semicontinuity. Define $\mathcal S$ to be the set of all points of $\Sigma$ with density exactly $K_*$; since ${\mathbb M}_\infty(\Sigma)=K_*$ as well, $\Sigma$ is also a cone around any point of $\mathcal S$. By standard GMT\footnote{Geometric Measure Theory; see for instance \cite{Sim18}.} arguments, $\mathcal S$ is actually a linear subspace, the ``spine'' of $\Sigma$, and $\Sigma+x=\Sigma$ for any $x\in\mathcal S$. In an appropriate Euclidean frame, we can decompose $\Sigma=C\times\R^{n-k}$, $1\leq k\leq n$, so that $C\subset\R^k$ is still a stable minimal hypercone. Moreover, letting $\Theta_{\Sigma}(x):=\lim_{R\to 0} \mathbb M_R(\Sigma, x)$, we have that
    \begin{equation}\label{eq:spinedens}
        \{0\}\times\R^{n-k}=\mathcal S=\{\Theta_{\Sigma}=K_*\}=\{\Theta_{\Sigma}\geq K_*\}\,. 
    \end{equation}
    The takeaway here is that all points in $C\setminus\{0\}$ have density below $K_*$. In fact, we have a strict drop, i.e. there exists\footnote{Indeed, notice that since $C$ is a cone we have $\sup_{x\in (C\setminus\{0\})\times B_1^{n-k}}\Theta_\Sigma(x)=\sup_{(x^k,x^{n-k})\in (C\cap \Sp^{k-1})\times B_1^{n-k}}\Theta_\Sigma(x^k,x^{n-k})$. If the strict drop were false, we would find a limit point $x=(x^k,x^{n-k})\in C\cap \Sp^{k-1}\times B_1^{n-k}$ with (by upper semicontinuity of densities) $\Theta_\Sigma(x^k,x^{n-k})= K_*$, but (by \eqref{eq:spinedens}) then $x^k=0$, a contradiction with $x^k\in\Sp^{n-1}$.} some $K=K(\Sigma)<K_*$ such that $\sup_{x\in (C\setminus\{0\})\times B_1^{n-k}}\Theta_\Sigma(x)\leq K$. \\
    Let $\delta=\delta(K)>0$ be given by \Cref{prop:curvestindeps}. Let $A=B_1^k\setminus B_{1/2}^k\subset\R^k$.
    By the convergence of the A--C energies to the limit and upper semicontinuity, we easily see that up to making $\delta>0$ smaller, for $i$ large enough we have ${\bf M}_{\delta}^{\widetilde \ep_i}(\widetilde u_i,X)\leq K+\delta$ for every $X\in (C\cap A)\times B_1^{n-k}$. Moreover, up to making $i$ larger, we have $\widetilde \ep_i\leq \delta^2$. This means that (by \Cref{prop:curvestindeps}) the sheeting assumptions are satisfied in $B_{\delta^2}(X)$ for some $C=C(K)$, for every $X\in A\cap \Sp^{k-1}\times B_1$, which gives uniform curvature estimates for $\{u=t\}$ for every $|t|\leq 0.9$. Passing to the limit (via Ascoli--Arzel\`a and Hausdorff convergence), we conclude these level sets converge in $C^{1,\alpha}$ to $(C\cap A)\times B_1^{n-k}$, and in particular $C\cap A$ is $C^{1,\alpha}$. Since $C$ is a cone, we deduce that $C\setminus\{0\}$ is $C^{1,\alpha}$ as well, and therefore smooth by standard bootstrapping results for minimal graphs.
    
    \textbf{Case 1.} If $3\leq k\leq n$, then Simons' classification \Cite{Simons68} of hypercones $C\subset\R^k$, $3\leq k \leq 7$, which are smooth and stable outside of the origin gives that $C$ (and thus $\Sigma$) is a hyperplane.
    
    \textbf{Case 2.} If $k=1,2$: If $k=1$, $\Sigma$ is trivially a hyperplane. If $k=2$, i.e. $\Sigma=C\times \R^{n-2}$, applying \Cref{prop:stablejunction} (given that $\Sigma$ is an $\ep$-limit of stable $\ep$-A--C solutions), we conclude that $\Sigma$ needs to be a hyperplane as well.

\noindent \textbf{Step 3.} Conclusion.\\
    Choosing an appropriate Euclidean coordinate frame, by the Constancy Theorem (\cite[p. 243]{Sim18}) we have $\Sigma=K_*[\{x_n=0\}]$, and $K_*\in\N$ by integrality. In particular, by \eqref{eq:varhausconv} applied to the $\widetilde u_i$, $\{|\widetilde u_i|\leq 0.9\}\cap B_1\subset\{|x_n|\leq \delta_1\}$
for $i$ large enough, which scaling back gives \eqref{eq:8qtuhabgbao} for the $u_i$.

We are left with showing that, given $\delta_1>0$, we have
\begin{equation*}
{\bf M}_{\delta_1}^{\widetilde \ep_i}(\widetilde u_i,y)\geq K_*- \delta_1 \quad \mbox{for every} \quad y\in \{x_n=0\}\cap B_1
\end{equation*}
for $i$ large enough, since it gives then \eqref{eq:8qtuhabgbao2} after scaling and thus yields a contradiction.\\

Now, since $\Sigma=K_*[\{x_n=0\}]$, obviously $\mathbb M_r(\Sigma,y)=K_*$ for any such $y$ and $r>0$. Assume there were however $y_i\in \{x_n=0\}\cap B_1$ such that ${\bf M}_{\delta_1}^{\widetilde \ep_i}(\widetilde u_i,y_i)<K_*-\delta_1$ instead, and let $y_\infty=\lim_k y_{i_k}$ be an accumulation point.

Let $\lambda<1$; by \eqref{eq:varenconv}, we see that $\lim_{i\to\infty} {\bf M}_{\lambda \delta_1}^{\widetilde \ep_{i_k}}(\widetilde u_i,y_\infty)=\mathbb M_{\lambda \delta_1}(\Sigma,y)=K_*$. In particular, ${\bf M}_{\lambda \delta_1}^{\widetilde \ep_{i_k}}(\widetilde u_{i_k},y_\infty)\geq K_*-\frac{1}{10}\delta_1$ for all $k$ large enough.

Now, up to making $k$ even larger, we additionally have $B_{\delta_1}(y_{i_k})\supset B_{\lambda \delta_1}(y_\infty)$. Fixing $\lambda$ close enough to $1$ (depending on $K_*$ and $\delta_1$), we can then ensure that ${\bf M}_{\delta_1}^{\widetilde \ep_{i_k}}(\widetilde u_{i_k},y_{i_k})\geq \frac{K_*-\delta_1}{K_*-\frac{1}{10}\delta_1}{\bf M}_{\lambda \delta_1}^{\widetilde \ep_{i_k}}(\widetilde u_{i_k},y_\infty)$, just by direct comparison. Combining the above we reach ${\bf M}_{\delta_1}^{\widetilde \ep_{i_k}}(\widetilde u_{i_k},y_{i_k})\geq K_*-\delta_1$, a contradiction.
\end{proof}
\subsection{Tangential stability and bad balls -- Proofs of Propositions \ref{prop:tanstab2} and \ref{prop:A'badcent}}\label{subsec:tanpfview}
\begin{proof}[Proof of \Cref{prop:tanstab2}]
We will show that, for every $\eta\in C_c^1(\R^n)$,
    \begin{align*}
        \int\mathcal A_e'^2\eta^2\,|\nabla u|^2&\leq\int \left(1-\left(e\cdot \frac{\nabla u}{|\nabla u|}\right)^2\right)|\nabla \eta|^2 |\nabla u|^2=\int |\nabla \eta|^2 |\nabla^{e'} u|^2\,,
    \end{align*}
which plugging in a standard linear cutoff shows \Cref{prop:tanstab2}.

We test stability with $\xi=|\nabla^{e'} u|\eta$, getting
    \begin{align}\label{eq:iqhrgiua}
        -\int W''(u)(|\nabla^{e'} u|\eta)^2&\leq\int |\nabla(|\nabla^{e'} u|\eta) |^2=\int |\nabla^{e'} u|^2|\nabla\eta|^2+|\nabla|\nabla^{e'} u||^2\eta^2+\frac{1}{2}\nabla(|\nabla^{e'} u|^2)\cdot\nabla\eta^2\,.
    \end{align}
    On the other hand, differentiating \eqref{eq:aceqintro} in a direction $\tau\in\Sp^{n-1}$ and multiplying by the directional derivative $u_\tau$ gives that $u_{\tau}\Delta u_{\tau}=W''(u)u_{\tau}^2$, which adding over an orthonormal basis $\{\tau_j\}_{j=1}^{n-1}$ of $\{\tau\cdot e=0\}$ and multiplying by $\eta^2$ shows
    $$
    \sum_j u_{\tau_j}\Delta u_{\tau_j}\eta^2=W''(u)|\nabla^{e'} u|^2\eta^2\,.
    $$
    Integrating by parts, we obtain that
    \begin{align*}
        -\int W''(u)|\nabla^{e'} u|^2\eta^2=\int \sum_j |\nabla u_{\tau_j}|^2\eta^2+\frac{1}{2}\int \sum_j \nabla u_{\tau_j}^2\cdot\nabla \eta^2=\int \sum_j |\nabla u_{\tau_j}|^2\eta^2+\frac{1}{2}\int \nabla (|\nabla^{e'} u|^2)\cdot\nabla \eta^2\,.
    \end{align*}
    Combining this with \eqref{eq:iqhrgiua}, we find that
    \begin{align*}
        \int  \Big(\sum_j|\nabla u_{\tau_j}|^2-|\nabla|\nabla^{e'} u||^2\Big)\eta^2\leq \int |\nabla^{e'} u|^2|\nabla\eta|^2\,,
    \end{align*}
    which since $\sum_j|\nabla u_{\tau_j}|^2=\sum_j\sum_{i=1}^n (\partial_i u_{\tau_j})^2
        =|D\nabla^{e'}u|^2$ shows the desired result.
\end{proof}

\begin{proof}[Proof of \Cref{prop:A'badcent}]
    Assume the proposition is false for contradiction. Then, we find appropriate $e_k,u_k$ such that $\int_{B_1(x_k)} \mathcal A^2(u_k)|\nabla u_k|^2\geq \delta_{bad}$ but $\int_{B_1(x_k)} \mathcal A_{e_k}'^2(u_k)|\nabla u_k|^2\leq \frac{1}{k}$. Letting $\widetilde u_k(x):=u_k(x-x_k)$, up to a subsequence we obtain a stable solution $\widetilde u_\infty=\lim_k \widetilde u_k$ and some $e_\infty\in\Sp^{n-1}$ such that $\int_{B_1} \mathcal A^2(\widetilde u_\infty)|\nabla u_\infty|^2\geq \delta_{bad}$ but $\mathcal A_{e_\infty}'(\widetilde u_\infty)\equiv 0$ in $B_1$. Then, by \Cref{prop:Atanclassif} below, we deduce that $\widetilde u_\infty$ is two-dimensional. By the 2D stable De Giorgi conjecture (\Cref{thm:2dclas}) then $\widetilde u_\infty$ is either $\pm 1$ or of the form $\phi(a\cdot x+b)$, but then $\mathcal A(\widetilde u_\infty)\equiv 0$, and we reach a contradiction.
\end{proof}

\begin{proposition}\label{prop:Atanclassif}
    Let $u:\R^n\to \R$ be a solution to A--C, and assume that $\mathcal A_e'\equiv 0$ for some $e\in\Sp^{n-1}$ in some open set $\Omega$. Then $u$ is two-dimensional in $\R^n$, i.e. $u=v(a_1\cdot x,a_2\cdot x)$ for some orthogonal directions $a_1,a_2\in\Sp^{n-1}$ and $v:\R^2\to\R$.
\end{proposition}
\begin{proof}
Assume withot loss of generality $e=\boldsymbol e_n$. Assuming that $\nabla' u\neq 0$, we can compute
\begin{align*}
    D \frac{\nabla' u}{|\nabla' u|}=\frac{D\nabla' u -\nabla|\nabla'u|\otimes \frac{\nabla' u}{|\nabla' u|}}{|\nabla' u|}\,,\quad\mbox{thus}\quad \mathcal A_e'^2=\left\|D \frac{\nabla' u}{|\nabla' u|}\right\|^2\,.
\end{align*}
The same argument as in \Cref{lem:A0imp1D} then gives the result.
\end{proof}
\subsection{Height and tilt excesses}\label{sec:excesses}
The first result is a Caccioppoli-type inequality for the level sets of $u$ (rather than $u$ itself).
\begin{lemma}[\textbf{Caccioppoli}]\label{lem:Cacc-ineq-AC}
Let $u:\R^n\to\R$ be an A--C solution, and let $\eta\in C_c^1(\R^n)$. Then, there is $C=C(n)$ such that
    \begin{align*}
        \int |\nabla^{e_n'} u|^2\eta^2\leq 4\int |x_n|^2|\nabla u|^2|\nabla \eta|^2+C\int W(u)|x_n\eta\eta_n|\,.
    \end{align*} 
\end{lemma}
\begin{proof}
    Obtained as in \cite[Remark 4.7]{Wang17}; we give the proof in \Cref{app:standres} for convenience of the reader.
\end{proof}

Recall the cylinder notation $\mathcal C_r:=B_{r}'\times[-r,r]\subset\R^{n-1}\times\R$. The following well-known lemma asserts that most of the energy is concentrated around the intermediate transition layers.
\begin{lemma}[{\bf Exponential decay}]\label{lem:expdecay}
    Let $u:\R^n\to\R$ be an A--C solution, and assume that $\{|u|\leq 0.85\}\cap \mathcal C_{4R}\subset \{|x_n|\leq \delta R\}$, where $\delta \in(0,1)$ and $R\ge 1$. Then, there are dimensional constants $C$ and $c_0>0$ such that
    \begin{equation}\label{eq:expdecay}
        \sup_{\mathcal C_{2R}\cap\{|x_n|\geq 2\delta R\}}\left[\frac{|\nabla u|^2}{2}+W(u)\right]\leq Ce^{-c_0\delta R}\,.
    \end{equation}
\end{lemma}
\begin{proof}
    Standard; a proof is given in \Cref{app:standres} for convenience of the reader.
\end{proof}
We now relate $L^\infty$-flatness (``height'') and the $L^2$-height excess.
\begin{lemma}[\textbf{Height controls height excess}]\label{lem:hconthexc}
    Given $\lambda\in(0,1)$, there exist $c_0>0$ and $C$ depending on $\lambda$ such that the following holds. Assume that $\{|u|\leq 0.9\}\cap \mathcal C_{4R}\subset \{|x_n|\leq \delta R\}$ for some $\delta>0$ and $R\ge 1$. Then,
    $${\bf H}_{R}^2(e_n)\leq C\max\{\delta^2,R^{-2(1-\lambda)}\}{\bf M}_R+Ce^{-c_0R^{\lambda}}.$$
\end{lemma}
\begin{proof}
Let $\widetilde \delta=\max\{\delta,R^{-(1-\lambda)}\}$. If $\widetilde \delta\geq 1$, the inequality is trivial. Otherwise, letting $P=\left[\frac{|\nabla u|^2}{2}+W(u)\right]$, using \Cref{lem:expdecay} we can estimate
\begin{align*}
    {\bf H}_R^2(u,e)&=\frac{1}{R^{n+1}}\int_{B_R} |x\cdot e|^2P=\frac{1}{R^{n+1}}\int_{B_R\cap \{|x\cdot e|\leq 2\widetilde\delta R\}} |x\cdot e|^2P + \frac{1}{R^{n+1}}\int_{B_R\cap \{|x\cdot e|> 2\widetilde\delta R\}} |x\cdot e|^2P\\
    &\leq \frac{1}{R^{n+1}}(2\widetilde\delta R)^2\int_{B_R\cap \{|x\cdot e|\leq 2\widetilde\delta R\}} P + \frac{C}{R^{n-1}}e^{-c_0\widetilde \delta R}\leq C\widetilde\delta^2 {\bf M}_R + Ce^{-c_0R^{\lambda}}\,.
\end{align*}
\end{proof}

\begin{lemma}[{\bf Height excess controls height}]\label{lem:Hbdhtez}
    Let $\delta_1\in(0,1)$. Then, there is $\delta_2>0$ depending on $\delta_1,n$ such that if ${\bf H}_R^2(e_n)\leq \delta_2^2$, then, for all $R\ge 1$,
    $$\{|u|\leq 0.9\}\cap B_{R/2}\subset \{|x_n|\leq \max\{\delta_1,\frac{1}{R}\}R\}\,.$$
    More generally, given additionally $c_0>0$, up to making $\delta_2>0$ smaller we have that
    $$\{x:{\bf M}_{\frac{\delta_1}{2} R}(x)\geq c_0\}\cap B_{R/2}\subset \{|x_n|\leq \delta_1 R\}\,.$$
\end{lemma}
\begin{proof}
    \textbf{Step 1.} Density of intermediate layers.\\
    Let $\widetilde \delta_1:= \max\{\delta_1,\frac{1}{R}\}$.
    We first see that $\{|u(x)|\leq 0.9\}\subset \{{\bf M}_{\frac{\widetilde\delta_1}{2}R}(x)\geq c_0\}$, for some $c_0>0$ depending on $n$.\\
    Indeed, since $\|\nabla u\|_{L^\infty}\leq C(n)$ and $W(s)>0$ for $s\neq \pm 1$, there are $c>0$ and $r_0\in(0,\frac{1}{2})$ depending on $n$ such that $W(u)\geq c>0$ in $B_{r_0}(x)$. Therefore, ${\bf M}_{r_0}(x)=\frac{1}{r_0^{n-1}}\int_{B_{r_0}(x)}\frac{1}{\sigma_{n-1}}\left[\frac{|\nabla u|^2}{2}+W(u)\right]\geq c_0$. On the other hand, since $\frac{\widetilde\delta_1}{2}R\geq \frac{1}{2}>r_0$, \Cref{lem:monformula} gives that ${\bf M}_{\frac{\widetilde\delta_1}{2}R}(x)\geq c_0$ as well.\\
    \textbf{Step 2.} Conclusion.\\
    By Step 1, it suffices to prove the second part of the result. If by contradiction there is $x\in B_{R/2}\setminus \{|x_n|\leq \delta_1 R\}$ with $M_{\frac{\delta_1}{2} R}(x)\geq c_0$, then (since $B_{\frac{\delta_1}{2}R}(x)\subset \{|x_n|\geq \frac{\delta_1}{2}R\}\cap B_R$) we have that
    \begin{align*}
        {\bf H}_R^2(e_n)&\geq \frac{1}{R^{n+1}}\int_{B_{\frac{\delta_1}{2}R}(x)}|y_n|^2\left[\frac{|\nabla u|^2}{2}+W(u)\right]\,dy \geq \left(\frac{\delta_1}{2}\right)^2\frac{1}{R^{n-1}}\int_{B_{\frac{\delta_1}{2}R}(x)}\left[\frac{|\nabla u|^2}{2}+W(u)\right]\,dy\\
        &= \left(\frac{\delta_1}{2}\right)^{n+1}{\bf M}_{\frac{\delta_1}{2}R}(x)\geq c_0\left(\frac{\delta_1}{2}\right)^{n+1}\,,
    \end{align*}
    and then if ${\bf H}_R^2(e_n)\leq \delta_2^2$ with $\delta_2$ small enough we reach a contradiction.
\end{proof}
We can finally give:
\begin{proposition}[\textbf{Height excess controls tilt excess}]\label{lem:Cacc-ineq-AC2} Let $u:\R^n\to\R$ be an A--C solution, with ${\bf M}_\infty\leq C_0$. Then, there are $C$, $R_0$ and $c>0$, depending on $C_0$ and $n$, such that if $R\geq R_0$, then
\begin{align*}
        {\bf T}_R^2(e_n)\leq C{\bf H}_{2R}^2(e_n)+e^{-cR}\,.
\end{align*}
    
\end{proposition}
\begin{proof}
We will show that if $r\geq R_0$ then
    \begin{equation}\label{eq:4uiaglavsg}
        {\bf T}_r^2(e_n)\leq C{\bf H}_{16r}^2(e_n)+e^{-cr}\,,
    \end{equation}
which up to a finite covering argument shows the desired result as well (with some different $C,c,R_0$).

    Let $\xi\in C_c^1(\R)$ be a standard cutoff satisfying $\chi_{[-1,1]}\leq \xi\leq \chi_{[-4/3,4/3]}$, so that $\xi\equiv 1$ in $[-1,1]$. Plugging $\eta(x',x_n)=\xi(\frac{|x'|}{r})\xi(\frac{x_n}{r})$ into \Cref{lem:Cacc-ineq-AC}, since $\chi_{B_r}\leq \eta\leq\chi_{B_{2r}}$ and $\eta_n\equiv 0$ in $\{|x_n|\leq r\}$ we find
    \begin{align*}
        \int_{B_r} |\nabla^{e_n'} u|^2\leq \frac{C}{r^2}\int_{B_{2r}} |x_n|^2|\nabla u|^2+C\int_{B_{2r}\cap\{|x_n|\geq r\}} \left[\frac{|\nabla u|^2}{2}+W(u)\right]|x_n\eta\eta_n|\,.
    \end{align*}
    We consider two different cases:\\
    If $\{|u|\leq 0.85\}\cap \mathcal C_{4r}\subset \{|x_n|\leq \frac{r}{2}\}$, then applying \Cref{lem:expdecay} to the last term we obtain \eqref{eq:4uiaglavsg}.\\
    Otherwise, \Cref{lem:Hbdhtez} shows that ${\bf H}_{16r}^2(e_n)\geq c>0$ as long as $R\geq R_0$ is large enough. Since ${\bf T}_r^2(e_n)\leq C{\bf M}_r\leq C{\bf M}_\infty$ is uniformly bounded, we conclude \eqref{eq:4uiaglavsg} as well.
\end{proof}

We end this section with a general comparison result. Morally, it says that if the level sets can be approximated well by two hyperplanes $x+e_x^\perp$ and $y+e_y^\perp$, then they are both close in direction and ``height''.
\begin{lemma}[{\bf Coefficient comparison}]\label{lem:coefcompexc}
Given $C_0\ge1$, there exist $C$ and $R_0$ depending on $C_0$,$n$ such that the following holds.

Let $u:\R^n\to\R$ be an A--C solution. Then, for any $x,y\in\R^n$ satisfying  $\frac{1}{C_0}\leq {\bf M}_R(x)\leq C_0$ and $R\geq \max\{\frac{1}{2}{\rm dist}(x,y), R_0\}$, for all $e_x,e_y\in\Sp^{n-1}$ we have:
    \begin{align}\label{eq:aigpharspga}
        \min_{\sigma\in\{+1,-1\}} |e_x-\sigma e_y|\leq C[{\bf H}_{6R}(e_x,x)+{\bf H}_{6R}(e_y,y)] + Ce^{-cR}
    \end{align}
    and
    \begin{align}\label{eq:aigpharspga2}
        |e_x\cdot(x-y)|\leq C[{\bf H}_{6R}(e_x,x)+{\bf H}_{6R}(e_y,y)] + Ce^{-cR}.
    \end{align}
\end{lemma}
\begin{proof}
{\bf Step 1.} Proof of \eqref{eq:aigpharspga}.\\
Combining \eqref{eq:deltamod3} and the lower density assumption, we see that $\frac{1}{R^{n-1}}\int_{B_{R}(x)} |\nabla u|^2\geq c$ as long as $R\geq R_0$, so that in particular
   \begin{equation}\label{eq:tfuagflkjasg}
        \min_{\sigma\in\{+1,-1\}} |e_x-\sigma e_y|^2\leq \frac{C}{R^{n-1}}\int_{B_R(x)} \min_{\sigma\in\{+1,-1\}} |e_x-\sigma e_y|^2 |\nabla u|^2\,.
    \end{equation}
    By the triangle inequality we can compute \begin{align*}
        \min_{\sigma\in\{+1,-1\}}|e_x-\sigma e_y|^2|\nabla u|^2&\leq 2\min_{\sigma\in\{+1,-1\}} \big|e_x|\nabla u|-\sigma \nabla u\big|^2+2\min_{\sigma\in\{+1,-1\}}\left|e_y|\nabla u|-\sigma\nabla u|\right|^2\\
    &= 2\left[ \big(2|\nabla u|^2-2 |\nabla u\cdot e_x||\nabla u|\big)+\big(2|\nabla u|^2-2|\nabla u\cdot e_y||\nabla u|\big)\right]\,,
    \end{align*}
    thus by Cauchy--Schwarz we have
    \begin{align*}
    \min_{\sigma\in\{+1,-1\}}|e_x-\sigma e_y|^2|\nabla u|^2&\leq 4\big[(|\nabla u|^2 -  |\nabla u\cdot e_x|^2)+(|\nabla u|^2 -  |\nabla u\cdot e_y|^2)\big]=4\big[|\nabla^{e_x'} u|^2+|\nabla^{e_y'} u|^2 \big]\,.
\end{align*}
Combining this with \eqref{eq:tfuagflkjasg} and using that $B_R(x)\subset B_{3R}(y)$, we find that
\begin{align*}
        \min_{\sigma\in\{+1,-1\}} |e_x-\sigma e_y|^2&\leq C\Big[\int_{B_{3R}(x)} |\nabla^{e_x'} u|^2+\int_{B_{3R}(y)} |\nabla^{e_y'} u|^2\Big]=C\left[{\bf T}_{3R}^2(e_x,x)+{\bf T}_{3R}^2(e_y,y)\right]\,.
\end{align*}
Using \Cref{lem:Cacc-ineq-AC2} we obtain \eqref{eq:aigpharspga}.\\

\noindent {\bf Step 2.} Proof of \eqref{eq:aigpharspga2}.\\
Recall that
$${\bf H}_R^2(e_x,x)=\frac{1}{R^{n+1}}\int_{B_R(x)} |e_x\cdot (X-x)|^2\left[\frac{|\nabla u|^2}{2}+W(u)\right]\,dX\,,$$
and likewise for ${\bf H}_R^2(e_y,y)$. On the other hand, we can bound
$$|e_x\cdot (x-y)|\leq |e_x\cdot(X-x)|+|e_x-\sigma e_y||X|+|e_y\cdot (X-y)|\,.$$
Integrating, and using \eqref{eq:aigpharspga} and ${\bf M}_R(x)\leq C_0$ to bound the second term on the right, we readily find \eqref{eq:aigpharspga2}.
\end{proof}

\subsection{Wang--Wei in a flat setting}
We rewrite \Cref{thm:sheetimpc2alpha} in a flat, large-scale setting. This is the version we will use in the rest of the article:
\begin{theorem}[\textbf{\cite{WW19} in a flat setting}]\label{thm:sheetpart2}
     Let $n\leq 10$. Let $u:B_R\to (-1,1)$ be a stable A--C solution, and assume the sheeting assumptions hold in $B_R$ for some $C_1$. Let $\delta_1>0$ and $\theta\in(0,1)$. Then, there exist $C_2$ depending on $C_1,\theta$, as well as $R_0$ depending on $\delta_1,C_1,\theta$ such that the following hold.\\
    Assume additionally that
    $$\varnothing \neq \{u=0\}\cap B_R\subset \{|e_n\cdot x|\leq \delta_2 R\}$$ 
    for some $\delta_2\in(0,1/16)$  and $R\geq R_0$. Then, there is some $N\in\N$ such that:
    \begin{itemize}
        \item There are $C^{\infty}$ graphs $g_i:B_{\frac{3}{5}R}'\to[-R/8,R/8]$, $g_1<...<g_N$, such that $\{u=0\}\cap \mathcal C_{\frac{3}{5}R}=\bigcup_{i=1}^N {\rm graph} \,g_i$.
    \item The estimates \eqref{eq:c2alphasheet}--\eqref{eq:wwsepbnd} hold for the $g_i$.
    \end{itemize}
\end{theorem}
\begin{proof}
    Immediate from \Cref{lem:rjbagbda} and \Cref{thm:sheetimpc2alpha}, by making $\delta_2>0$ sufficiently small.
\end{proof}
Moreover, we collect some additional properties.
\begin{lemma}[{\bf Additional properties}]\label{lem:addproperties}
    Let $\delta_1>0$. In the setting of \Cref{thm:sheetpart2}, up to making $\delta_2>0$ smaller and $R_0$ larger depending also on $\delta_1$, the following additionally hold:
    \begin{itemize}
    \item The graphs satisfy $\|\nabla g_i\|_{L^\infty}\leq \delta_1$.
    \item There holds $\{|u|\leq 0.9\}\subset \{|e_n\cdot x|\leq \delta_2 R +C_2\}$ in $\mathcal C_{R/2}$.
    \item We have $|{\bf M}_{R/2}(u)-N|\leq \delta_1 N$.
    \end{itemize}
\end{lemma}
For the proof, we first need:
\begin{lemma}[{\bf 1D approximation}]\label{lem:1dapprox}
    Let $u:B_R\to (-1,1)$ be a stable A--C solution, and assume the sheeting assumptions hold in $B_R$ for some $C_1$. Given $b_1\in(0,1)$ and $\delta>0$, there exists $R_0=R_0(C_1,b_1,\delta,n)$ such that if $R\geq R_0$, then
    $$
    \inf_{a\in \Sp^{n-1},\,b\in\R} \|u-\phi(a\cdot x+b)\|_{L^\infty(B_\frac{1}{\delta}(x_0))}\leq \delta \quad\mbox{for any}\quad x_0\in \{|u|\leq 1-b_1\}\cap B_{R/2}\,.
    $$
    In other words, making $R$ large enough then $u$ is close to a $1D$ solution around any intermediate point.
\end{lemma}
\begin{proof}[Proof of \Cref{lem:1dapprox}]
    Notice that, by exponential decay (\Cref{lem:expdecay}), there is $C=C(n,b_1)$ such that $\{|u|\leq 1-b_1\}\cap B_{R/2}\subset \{x:{\rm dist}(x,\{|u|\leq 0.85\})\leq C\}$ for $R_0$ large enough. It easily follows that it suffices to show the Lemma just with $1-b_1=0.85$.
    
     Observe now that, for any $x_0\in \{|u|\leq 0.85\}\cap B_{R/2}$, there is $r_0=r_0(n)>0$ such that---up to making $R_0$ larger---we have $B_{r_0}(x_0)\subset \{|u|\leq 0.9\}\cap B_R$. In particular, the bounds in \eqref{eq:sheetassump} are satisfied at all such points. We conclude by arguing exactly as in \cite[Lemma 2.1]{WW19}. 
\end{proof}

\begin{proof}[Proof of \Cref{lem:addproperties}]
    The first bullet follows directly by interpolation, up to making $\delta_2>0$ small.

    For the second one, observe that if $R_0$ and $C_2$ are large enough, \Cref{lem:1dapprox} ensures that $u$ changes sign in $B_{C_2}(x)$ for any $x\in \{|u|\leq 0.9\}\cap \mathcal C_{R/2}$. In particular
    \begin{equation}\label{eq:lqjhghagpa}
        {\rm dist}(\{|u|\leq 0.9\}\cap \mathcal C_{R/2},\cup_{i=1}^N{\rm graph}\, g_i\cap \mathcal C_{\frac{3}{4}R})={\rm dist}(\{|u|\leq 0.9\}\cap \mathcal C_{R/2},\{u=0\}\cap \mathcal C_{\frac{3}{4}R})\leq C_2\,,
    \end{equation}
    which gives the second bullet.

    For the third one, define $v_i(x)=[\sign(u)]\phi({\rm dist}(x,{\rm graph}\, g_i))$. By the first bullet, $|\nabla g_{i}|\leq \delta_1$ as long as $R_0$ is large enough and $\delta_2$ is small enough. Up to updating these values, this easily shows that $|{\bf M}_{R/2}(v_i)-1|\leq \delta_1$, since $|{\bf M}_{\infty}(\phi)|=1$ and $g_i$ is almost flat.\\
    On the other hand, given $\delta>0$, set $\Omega_i:=\{{\rm dist}(x,{\rm graph}\, g_i)\leq \frac{1}{\delta}\}\cap B_{R/2}$. Thanks to \eqref{eq:lqjhghagpa} and \Cref{lem:expdecay}, we have exponential decay away from $\cup_i {\rm graph}\, g_i$, thus up to making $\delta$ small enough we easily obtain that $\mathcal E(u,B_{R/2}\setminus \cup_{i=1}^N \Omega_i)\leq \delta_1 N$. Moreover, using \Cref{lem:1dapprox}, it is easy to see that $|u-v_i|\leq \delta_1$ in $\Omega_i$ as long as $R_0$ is large enough.\\
    Adding up these facts,  the third bullet readily follows.
\end{proof}

\section{Setting up the contradiction}\label{sec:contbigs}
\subsection{General setting}\label{sec:contrsetsub}
We assume that $n\leq 7$ in this section; we will restrict to $n=4$ from \Cref{sec:graphdec} on.\\
From now on, we fix a critical solution $u:\R^n\to(-1,1)$, which has density at infinity ${\bf M}_\infty = K_*<\infty$. In particular, $u$ is not 1D.

We will use $C,C_i,R_0$ (resp. $c,c_i,r_0$) to denote large (resp. small) positive universal constants---unless otherwise stated---and which may change from line to line.

\begin{lemma}[{\bf Excess goes to $0$}]\label{lem:excgoto0}
    Let $\zz\in\cZ$. Then, ${\bf H}_R(\zz)\leq \widetilde\omega(R^{-1})$, where $\widetilde \omega$ is a dimensional modulus of continuity.
\end{lemma}
\begin{proof}
    By \Cref{prop:W}, there is $e_{\zz,R}$ such that
    $$\{|u|\leq 0.9\}\cap B_R(\zz)\subset\{x:|e_{\zz,R}\cdot (x-\zz)|\leq \omega(R^{-1}) R\}.$$
    But then, applying \Cref{lem:hconthexc} (appropriately translated and rotated) with, say, $\lambda =1/2$, we deduce that ${\bf H}_{\frac{1}{4\sqrt{2}}R}(e_{\zz,R},\zz)= \widetilde\omega(R^{-1})$, for some new modulus of continuity $\widetilde\omega$.
    
    Thus ${\bf H}_{\frac{1}{4\sqrt{2}}R}(\zz)= \widetilde\omega(R^{-1})$ as well, and considering $4\sqrt{2}R$ instead of $R$ (and conveniently updating $\widetilde\omega$) gives the result.
\end{proof}
\begin{proposition}[{\bf Tangential stability and bad ball count}]\label{prop:geiohfah}
Let $\zz\in\cZ$. If $R\geq R_0$, then
    \begin{equation}\label{eq:awioufhiu}
        R^{3-n}|B_1(\cZ)\cap B_{\frac{3R}{4}}(\zz)|\leq C{\bf K}_{R}^2(e,\zz)\leq C{\bf H}_{4R}^2(e,\zz) \quad \mbox{for any}\quad e\in\Sp^{n-1},
    \end{equation}
    thus
    \begin{equation}\label{eq:awioufhiunoe}
       R^{3-n}|B_1(\cZ)\cap B_{\frac{3R}{4}}(\zz)|\leq C{\bf K}_{R}^2(\zz)\leq C{\bf H}_{4R}^2(\zz)\,.
    \end{equation}
\end{proposition}
\begin{remark}\label{rmk:awioufhiu4n}
    In particular, for $n=4$ we see that
    \begin{equation}\label{eq:awioufhiu4n}
        {\bf H}_{R}^2(\zz)\geq \frac{c}{R}\,.
    \end{equation}
\end{remark}
\begin{proof}
{\bf Step 1.} Bad ball count.\\
We first show that
\begin{equation}\label{eq:bseta'conv}
        R^{3-n}|B_1(\cZ\cap B_{\frac{3R}{2}})|\leq C{\bf K}_{2R}^2(e).
    \end{equation}
    By \Cref{prop:A'badcent}, given any bad center $\zz$ and any $e\in\Sp^{n-1}$ we know that $\int_{B_1(\zz)}\mathcal A_e'^2|\nabla u|^2\geq \delta_{bad}'$. Consider the collection of balls $\{B_1(\zz)\}_{z\in \cZ\cap B_{\frac{3R}{2}}}$ for such $\zz$, and pass to a disjoint Vitali subcollection $\{B(\zz_i)\}_{i=1}^N$, so that $B_1(\cZ\cap B_{\frac{3R}{2}})\subset \bigcup_i B_5(\zz_i)$. Since the $B_1(\zz_i)$ are disjoint, we can bound
    \begin{align*}
        |B_1(\cZ\cap B_{\frac{3R}{2}})|\leq 5^n\sum_{i=1}^N |B_1(\zz_i)|\leq \frac{5^n}{\delta_{bad}'}\sum_{i=1}^N\int_{B_1(\zz_i)} \mathcal A_e'^2|\nabla u|^2 \leq \frac{5^n}{\delta_{bad}'}\int_{B_{2R}} \mathcal A_e'^2|\nabla u|^2\,,
    \end{align*}
    where in the last inequality we used that $B_1(\zz_i)\subset B_{1+\frac{3R}{2}}\subset B_{2R}$ as long as $R\geq 2$. Multiplying by $R^{3-n}$ we conclude \eqref{eq:bseta'conv}.\\
    
\noindent {\bf Step 2.} Comparison of ${\bf T}$ and ${\bf H}$.\\
    Recentering \eqref{eq:bseta'conv} and combining this with \Cref{prop:tanstab2} and \Cref{lem:Cacc-ineq-AC2}, we see that
    \begin{equation}\label{eq:awioufhiu2}
        R^{3-n}|B_1(\cZ)\cap B_{\frac{3R}{4}}(\zz)|\leq C{\bf K}_{R}^2(e,\zz)\leq C{\bf T}_{2R}^2(e,\zz)\leq C{\bf H}_{4R}^2(e,\zz)+e^{-cR}
    \end{equation}
    for $R\geq R_0$ large enough.
    Since $\zz\in \cZ$ we obviously have $|B_1(\cZ)\cap B_{\frac{3R}{4}}(\zz)|>c$, and then \eqref{eq:awioufhiu2} shows in particular that ${\bf H}_{4R}^2(e,\zz)\geq \frac{c}{R^{n-3}}$. Thus, 
    \begin{equation}\label{eq:84tbbgobfsg}
        {\bf T}_{2R}^2(e,\zz)\leq C{\bf H}_{4R}^2(e,\zz)\,.
    \end{equation}
    We see then that
    \begin{equation}\label{eq:qiogbog}
        R^{3-n}|B_1(\cZ)\cap B_{\frac{3R}{4}}(\zz)|\leq C{\bf K}_{R}^2(e,\zz)\leq {\bf T}_{2R}^2(e,\zz)\leq C{\bf H}_{4R}^2(e,\zz)\,,
    \end{equation}
    which gives \eqref{eq:awioufhiu}; taking infima among $e\in \mathbb S^{n-1}$, we obtain \eqref{eq:awioufhiunoe}.
\end{proof}

\subsection{Selection of center and scale -- Proof of Lemmas \ref{lem:radcentselec} and \ref{lem:NbdRknew}}\label{sec:seleccentrscale}

We first give:

\begin{proof}[Proof of \Cref{lem:radcentselec}]
Given $R\geq R_0$, set 
$$F(R):=\sup_{\\z\in \cZ}\frac{{\bf H}_{4R}^2(\zz)}{{\bf K}_R^{2(1+\alpha)}(\zz)},$$ which corresponds to penalising the stability inequality \eqref{eq:awioufhiu} by introducing an $\alpha>0$ exponent. The lower bound for ${\bf K}$ in \eqref{eq:awioufhiunoe}, together with the fact that ${\bf H}$ is uniformly bounded (since ${\bf M}$ is), shows that $F(R)\leq CR^{(n-3)(1+\alpha)}$. Moreover, tangential stability (i.e. \eqref{eq:awioufhiunoe}) gives that
    $$F(R)=\sup_{z\in Z}\frac{{\bf H}_{4R}^{2(1+\alpha)}(\zz)}{{\bf K}_R^{2(1+\alpha)}(\zz)}\frac{1}{{\bf H}_{4R}^{2\alpha}(\zz)}\geq c \sup_{z\in Z}\frac{1}{{\bf H}_{4R}^{2\alpha}(\zz)}.$$
\Cref{lem:excgoto0} then shows that $\lim_{R\to\infty} F(R)=\infty$; defining $\widetilde F(R)=\sup_{r\in[R_0,R]} F(r)$, which is finite and nondecreasing, naturally $\lim_{R\to\infty} \widetilde F(R)=\infty$ as well.\\
By definition, we can find $r\in[R_0,R]$ with $\frac{2}{3}\widetilde F(R)\leq  F(r)\leq  \widetilde F(R)$, thus trivially $\frac{2}{3}\widetilde F(r)\leq  F(r)\leq  \widetilde F(r)$ as well. It is then clear that we can find a sequence $R_k\to\infty$ and associated $\zz_k\in \cZ$ with
$$
\frac{1}{2}\widetilde F(R_k)\leq\frac{{\bf H}_{4R_k}^2(\zz_k)}{{\bf K}_{R_k}^{2(1+\alpha)}(\zz_k)}\leq \widetilde F(R_k)\,,
$$
and setting $\epsilon_k:={\bf H}_{4R_k}^2(\zz_k)$, which tends to zero by \Cref{lem:excgoto0}, we conclude the proof.
\end{proof}

Recall the definition of $N(\theta,B_R(\zz))$, the size of the bad set at resolution $\theta R$, in \Cref{def:Ntheta}.
\begin{lemma}\label{lem:Ndeflem}
    Let $\zz\in\cZ$, $R>0$ and $\theta\in(0,1]$. Then, we can find disjoint bad balls $\{B_{\theta R}(\widetilde\zz_i)\}_{i=1}^Q$ such that $\{\widetilde\zz_i\}_{i=1}^Q\subset \cZ\cap B_R(\zz)$, and
    \begin{equation}
      cN(\theta,B_R(\zz))\leq Q\leq CN(\theta,B_R(\zz))\quad \mbox{and}\quad   \bigcup_{i=1}^Q B_{\theta R}(\widetilde\zz_i)\subset B_{\theta R}(\cZ\cap B_R(\zz))\subset \bigcup_{i=1}^Q B_{5\theta R}(\widetilde\zz_i)\,.
    \end{equation}
\end{lemma}
\begin{proof}
    The collection of balls $B_{\theta R}(\widetilde\zz)$ with $\widetilde\zz\in \cZ\cap B_R(\zz)$ covers $B_{\theta R}(\cZ\cap B_R(\zz))=\cup_{\widetilde\zz\in \cZ\cap B_R(\zz)} B_{\theta R}(\widetilde\zz)$ by definition. We can then pass to a $5$-Vitali disjoint subcover, i.e. a finite disjoint subcollection $\{B_{\theta R}(\widetilde\zz_i)\}_{i=1}^Q$ with the property that the $\{B_{5\theta R}(\widetilde\zz_i)\}_{i=1}^{Q}$ still cover $B_{\theta R}(\cZ\cap B_R(\zz))$. The result is then immediate.
\end{proof}
We now give:

\begin{proof}[Proof of \Cref{lem:NbdRknew}]
Let
\begin{equation}\label{eq:oirghlvbj}
    \theta_k:=\inf \left\{\theta\geq \frac{R_0}{R_k}\mbox{ : exists } \zz\in \cZ \mbox{ with } B_{\theta R_k}(\zz)\subset B_{R_k}(\zz_k), {\bf K}_{\theta R_k}^2(\zz)\leq\theta^{\beta}{\bf K}_{R_k}^2(\zz_k)\right\}\,.
\end{equation}
Combining \eqref{eq:awioufhiunoe} (with $R=R_0$) and \eqref{eq:geiweghio}, for $k$ large enough we see that $0<\frac{10R_0}{R_k}\leq \theta_k \leq 1$. By definition of infimum, we can choose $\widetilde \theta_k\in[\theta_k,\min\{2\theta_k,1\}]$ and $\zk\in \cZ$ such that, letting $\Rk:=\widetilde \theta_k R_k$,
\begin{equation}\label{eq:guighiugh}
    B_{\Rk}(\zk)\subset B_{R_k}(\zz_k)\quad \mbox{and}\quad {\bf K}_{\Rk}^2(\zk)\leq\widetilde\theta_k^{\beta}{\bf K}_{R_k}^2(\zz_k).
\end{equation}
Combining \eqref{eq:guighiugh} with \eqref{eq:geiweghio}, in particular we see that
\begin{equation}\label{eq:gaghsdlg}
    {\bf H}_{4R}^2(\zz)\leq 2\left(\frac{{\bf K}_R^{2}(\zz)}{{\bf K}_{\Rk}^{2}(\zk)}\right)^{1+\alpha}\epk^2\quad \mbox{for any}\quad B_R(\zz)\subset B_{\Rk}(\zk),
\end{equation}
so that the second inequality in \eqref{eq:hjgagfas} follows by using the trivial comparison ${\bf K}_R^{2}(\zz)\leq (\Rk/R)^{n-3}{\bf K}_{\Rk}^{2}(\zk)$.
Applying \eqref{eq:gaghsdlg} with $\zz=\zk$ and $R=\Rk$, we deduce that ${\bf H}_{4\Rk}^2(\zk)\leq 2\epk^2\leq 2\ep_k^2\to 0$, so that moreover $\Rk\to\infty$ by \eqref{eq:awioufhiunoe}.\\
It remains to see \eqref{eq:fhiowqrhgoqwg}. Consider $\theta = t\widetilde\theta_k$ with $t\in(\frac{R_0}{\Rk},\frac{1}{2})$, and assume that $N_\theta\geq C\theta^{-(n-3+\beta)}$ for some $C$ large for contradiction; by \Cref{lem:Ndeflem}, we find $Q\geq\theta^{-(n-3+\beta)}$ disjoint bad balls with $\bigcup_{i=1}^Q B_{\theta R}(\widetilde\zz_i)\subset B_{\theta R}(\cZ\cap B_R(\zz))$. Then, by the disjoint property, we can bound
\begin{equation}\label{eq:yasgpbasougaao}
\sum_{i=1}^{Q} {\bf K}_{\theta \Rk}^{2}(\widetilde\zz_i)\leq \inf_{e\in\Sp^{n-1}}\sum_{i=1}^{Q} \frac{1}{(\theta \Rk)^{n-3}}\int_{B_{\theta\Rk}(\zz_i)} \mathcal A_e'^2|\nabla u|^2\leq \inf_{e\in\Sp^{n-1}}\frac{1}{(\theta \Rk)^{n-3}}\int_{B_{\Rk}(\zk)} \mathcal A_e'^2|\nabla u|^2= \theta^{-(n-3)}{\bf K}_{\Rk}^{2}(\zk)\,;
\end{equation}
in particular, since $Q\geq\theta^{-(n-3+\beta)}$ there is at least one $j\in\{1,...,Q\}$ such that ${\bf K}_{\theta \Rk}^{2}(\widetilde\zz_j)\leq  \theta^{\beta}{\bf K}_{\Rk}^{2}(\zk)$,
which contradicts the definition of $\widetilde \theta_k$ and $\theta_k$.
\end{proof}

\section{Graphical decomposition}\label{sec:graphdec}
We restrict to $n=4$ in the remainder of the article.
\subsection{Graphical direction}
We first show that, thanks to \eqref{eq:hjgagfas}, a single direction dominates at all scales of interest. Recall that $\gamma:=\frac{1}{4}$. 
\begin{lemma}[\bf{Graphical direction}]\label{lem:dxzggiuhg2}
    Let $\alpha\in (0,\frac{\gamma}{8}]=(0,\frac{1}{32}]$. For any given $k$, let us choose a Euclidean coordinate frame\footnote{This frame always exists thanks to the second bullet in \Cref{lem:NbdRknew}.} such that ${\bf H}_{4\Rk}^2(e_n,\zk)\leq 2\epk^2$. Then,
    \begin{equation}\label{eq:keiaisbgfliab}
        \epk^{2-\gamma/2}\Rk\to\infty\quad \mbox{and}\quad\sup_{\zz\in\cZ\cap B_{\frac{1}{2}\Rk}(\zk),\,R\in[\epk^{2-\gamma/2}\Rk,\Rk]} {\bf H}_{R}^2(e_n,\zz) \to 0\qquad \quad \mbox{as}\quad k\to\infty\,.
    \end{equation}
\end{lemma}
\begin{proof}
Let $\Rh:=\epk^{2-\gamma/2}\Rk$.\\
{\bf Step 1.} Growth of the smallest scale $\Rh$.\\
    Combining \Cref{lem:NbdRknew} and \eqref{eq:awioufhiu4n} we have that $2\epk^2\geq {\bf H}_{4\Rk}^2(\zk)\geq\frac{c}{\Rk}$, or $\Rk\geq c\epk^{-2}$. Thus, $\Rh\geq c\epk^{-\gamma/2}\xrightarrow[]{\epk\to 0}\infty$.
    
\noindent {\bf Step 2.} Coefficient comparison.\\
Let $l\in\N$ be such that $R_l:=2^{-l}\Rk\geq \frac{\Rh}{2}$. By Step 1, we have that $R_l \gg R_0$ for $k$ large enough, so that we can apply \eqref{eq:hjgagfas}. Hence,  given $\zz\in\cZ\cap B_{\frac{1}{2}\Rk}(\zk)$ for every $i=1,...,l$ there exists some vector $v_i$ such that
\begin{equation}
\label{eq:Azbound23}
{\bf H}_{4R_i}^2(v_i,\zz) \le C (\Rk/R_i)^{1+\alpha} \epk^2 = C 2^{(1+\alpha) i}\epk^2.
\end{equation}
Observe that we can take $v_1=e_n$ in \eqref{eq:Azbound23}, since ${\bf H}_{4\Rk}^2(e_n,\zk)\leq 2\epk^2$.

\noindent {\bf Claim.} We have the bound ${\bf H}_{4R_i}^2(e_n,\zz) \le C (\Rk/R_i)^{1+\alpha} \epk^2$, for every $i=1,...,l$.
\vspace{0.2cm}

Assuming the claim the second property in \eqref{eq:keiaisbgfliab} follows, thanks to the bound
$$(\Rk/R_l)^{1+\alpha}\epk^2\leq C\epk^{-(2-\gamma/2)(1+\alpha)} \epk^2 = C\epk^{\gamma/2-2\alpha+\gamma\alpha/2}\leq C\epk^{\gamma/4}\xrightarrow[]{\epk\to 0} 0\,.$$

To show the claim, observe that \eqref{eq:Azbound23} gives that ${\bf H}_{4R_i}^2(v_i,\zz)+{\bf H}_{4R_i}^2(v_{i+1},\zz) \le  C 2^{(1+\alpha) i}\epk^2$.
Applying \Cref{lem:coefcompexc}, and adding over $i$ the corresponding geometric sum, we find that
$$
|v_j-v_i|^2\leq C 2^{(1+\alpha)i}\epk^2=C(\Rk/R_i)^{1+\alpha} \epk^2\quad \mbox{for any}\quad  1\leq j<i\leq l.
$$
Taking $j=1$ (recall that $v_1=e_n$) we find that $|e_n-v_i|^2\leq C(\Rk/R_i)^{1+\alpha} \epk^2$, which easily shows that
\begin{equation*}
{\bf H}_{4R_i}^2(e_n,\zz)\le C{\bf H}_{4R_i}^2(v_i,\zz)+C|e_k-v_i|^2\leq C (\Rk/R_i)^{1+\alpha} \epk^2
\end{equation*}
as desired.
\end{proof}

We have only used the second inequality of \eqref{eq:hjgagfas}, which is a ``worst case bound'' that holds for every $\zz$. For later use, we note the following ``stronger'' result :
\begin{lemma}[{\bf Sum of height excesses}] \label{lem:sum_of_height_ex}
Assume that $R\in[\epk^{2-\gamma/2}\Rk,\Rk]$ and $k$ is large enough. Then, given $\{\zz_j\}_{j\in J}\subset \cZ\cap B_{\frac{1}{2}\Rk}(\zk)$ with $\{B_{R}(\zz_j)\}_{j\in J}$ pairwise disjoint, there exist $v_j\in\Sp^{n-1}$ such that
\begin{equation}
\label{eq:Azboundsum}
|v_j-e_n|\le o_k(1) \quad \mbox{and}\quad \sum_{j\in J} {\bf H}_{4R}^2(v_j,\zz_j) \le C \left(\frac{\Rk}{R}\right)^{1+\alpha} \epk^2. 
\end{equation}
\end{lemma}
\begin{proof}
We will use\footnote{To show this: Dividing by $(\sum_{j\in J} a_j)^{1+\alpha}$, we can assume $\sum_{j\in J} a_j=1$, thus in particular $a_j\leq 1$ for every $j\in J$. But then $a_j^{1+\alpha}\leq a_j$, so that $\sum_{j\in J} a_j^{1+\alpha}\leq \sum_{j\in J} a_j=1=(\sum_{j\in J} a_j)^{1+\alpha}$
as desired.}: Given positive numbers $\{a_j\}_{j\in J}$ and some $\alpha>0$, we have $\sum_{j\in J} a_j^{1+\alpha}\leq (\sum_{j\in J} a_j)^{1+\alpha}$.

By this fact and \eqref{eq:hjgagfas},  we can estimate
\begin{equation}\label{eq:hgahgpgnog}
        \sum_{j\in J} {\bf H}_{4R}^2(\zz_j)\leq  2\sum_{j\in J}\left(\frac{{\bf K}_R^{2}(\zz_j)}{{\bf K}_{\Rk}^{2}(\zk)}\right)^{1+\alpha}\epk^2\leq  2\left(\sum_{j\in J}\frac{{\bf K}_R^{2}(\zz_j)}{{\bf K}_{\Rk}^{2}(\zk)}\right)^{1+\alpha}\epk^2.
    \end{equation}
Since the $\{B_{R}(\zz_j)\}_{j\in J}$ are pairwise disjoint, arguing as in \eqref{eq:yasgpbasougaao} we can bound
$$\sum_{j\in J}{\bf K}_{R}^{2}(\zz_j)\leq (\Rk/R_l)^{n-3}{\bf K}_{\Rk}^{2}(\zk)=(\Rk/R_l){\bf K}_{\Rk}^{2}(\zk)\,,$$
so that together with \eqref{eq:hgahgpgnog} and the definition of ${\bf H}$ we find $v_j\in\Sp^{n-1}$ with
\begin{equation}\label{eq:hgahgpgnog3}
        \sum_{j\in J} {\bf H}_{4R}^2(v_j,\zz_j)\leq C(\Rk/R)^{1+\alpha}\epk^2\,.
    \end{equation}
By \Cref{lem:dxzggiuhg2}, since $R\geq \epk^{2-\gamma/2}\Rk$ we know that ${\bf H}_{4R}^2(e_n,\zz_j) = o_k(1)$; likewise, \eqref{eq:hgahgpgnog3} shows that ${\bf H}_{4R}^2(v_j,\zz_j)= o_k(1)$. Applying \Cref{lem:coefcompexc}, up to possibly changing the sign of the $v_j$ then $|v_j-e_n|= o_k(1)$ as well.
\end{proof}

\subsection{Decomposing $\{u=0\}$ into $K_*$ layers -- Proof of \Cref{prop:graphdec}}
We can finally show that the level sets of our solution decompose into exactly $K_*$ graphs away from the bad set.

\begin{proof}[Proof of \Cref{prop:graphdec}]
Let $\Rh:=\epk^{2-\gamma/2}\Rk$.\\
Observe first that, combining \Cref{prop:W} with \eqref{eq:keiaisbgfliab} and \Cref{lem:coefcompexc}, we immediately deduce:\\
Let $\delta>0$. Then, for $k$ large enough, for any $\zz\in\cZ\cap B_{\frac{1}{2}\Rk}(\zk)$ and $R\in[\Rh,\Rk]$ we have
\begin{equation}\label{eq:gwiongown1}
    \{u=0\}\cap B_R(\zz)\subset\{|e_n\cdot (x-\zz)|\leq o_k(1) R\}
\end{equation}
and
\begin{equation}\label{eq:gwiongown2}
    K_*-{\bf M}_{\delta R}(y)\leq o_k(1) \quad \mbox{for every} \quad y\in \{e_n\cdot (x-\zz)=0\}\cap B_R(\zz)\,.
\end{equation}
Given $\bar x'\in\Omega_{2-\gamma/2}$, let $\zz(\bar x')\in\cZ\cap B_{\Rk/2}(\zk)$ be such that $\rho:=\dist(\bar x',\cZ')=|\bar x' - \zz'(\bar x)|$, and put $\bar x=(\bar x',(\zz(\bar x'))_n)$. $B_\rho(\bar x)$ is obviously a good ball, i.e. $B_\rho(\bar x)\cap \cZ=\emptyset$. For $k$ large enough, the sheeting assumptions are then satisfied in $B_{3\rho/4}(\bar x)$ thanks to \Cref{thm:goodimpsheet}, and moreover \eqref{eq:gwiongown1}--\eqref{eq:gwiongown2} give that
\begin{equation}\label{eq:gwiongown12}
    \{u=0\}\cap B_\rho(\bar x)\subset\{|e_n\cdot (x-\bar x)|\leq o_k(1) \rho\} \quad\mbox{and}\quad
    K_*-{\bf M}_{\delta\rho}(\bar x)\leq o_k(1)\,.
\end{equation}
Set $\mathcal C_{\rho/2}(\bar x)=B_{\rho/2}'(\bar x')\times [\bar x_n-\rho/2,\bar x_n+\rho/2]$. For $k$ large enough, by \Cref{thm:sheetpart2} and \Cref{lem:addproperties} we see that there are exactly $K_*$ smooth graphs $g_i:B_{\rho/2}'(\bar x')\to[-\rho/8,\rho/8]$, $g_1<...<g_{K_*}$, such that $\{u=0\}\cap \mathcal C_{\rho/2}(\bar x)=\bigcup_{j=1}^{K_*} {\rm graph} \,g_i$ and $|\nabla g_i|\le o_k(1)$.\\

Vertically, these graphs actually cover all of the zero level set: Indeed, given $y\in\{u=0\}\cap B_{\frac{1}{2}\Rk}(\zk)$, with $y'\in B_{\rho/2}'(\bar x')$, assume for contradiction that $y\notin \bigcup_{j=1}^{K_*} {\rm graph} \,g_i$, so that $y\notin \mathcal C_{\rho/2}(\bar x)$. This means that $|y_n-\bar x_n|\geq \rho/2$, thus letting $R=2|y_n-\bar x_n|$ we see that
\begin{equation}
    y\in \{u=0\}\cap B_R(\bar x)\quad\mbox{but}\quad |e_n\cdot(y-\bar x)|=|y_n-\bar x_n|\geq \frac{R}{2}\,,
\end{equation}
which contradicts \eqref{eq:gwiongown12} for $k$ large enough.\\

Summing up our argument up to now, given any point in $\Omega_{2-\gamma/2}$, we have obtained a local (horizontally) decomposition of the intermediate level sets as exactly (vertically) $K_*$ graphs. In particular, given any $x', y'\in \Omega_{2-\gamma/2}$, in case their associated decompositions have an overlap then they obviously need to match. Taking the union over $\bar x'\in \Omega_{2-\gamma/2}$ of these graphs, we have shown:\\
For $k$ large enough there are $K_*$ smooth graphs $g_i:\Omega_{2-\gamma/2}\to \R$, $g_1<...<g_{K_*}$, such that
\begin{equation*}
        \{u=0\}\cap B_{\frac{1}{2}\Rk}(\zk)\cap (\Omega_{2-\gamma/2}\times\R)=\bigcup_{i=1}^{K_*} {\rm graph} \,g_i\,.
\end{equation*}
Finally, we already saw that $|\nabla g_i|\le o_k(1)$, which together with \Cref{thm:sheetpart2} (i.e. \eqref{eq:meancurvsheet}) gives \eqref{eq:uagfabaf}.
\end{proof}
We record moreover the following byproduct of the proof (which follows directly from \eqref{eq:gwiongown1}):
\begin{lemma}[{\bf Large scale flatness}]\label{lem:graphlargeflat} We have
$|g_i(x')-\zz_n|\leq o_k(1)|x'-\zz'|$ in $\Omega_{2-\gamma/2}$,
for every $\zz\in \cZ\cap B_{\frac{1}{2}\Rk}(\zk)$ and $i\in\{1,...,K_*\}$.
\end{lemma}

\section{Improvement of excess for the graphs -- Proof of \Cref{prop:impflatsinglin}}\label{sec:improvlayers}
The goal of this section is to prove \Cref{prop:impflatsinglin}. It will be a consequence of the following:
\begin{proposition}\label{prop:decaysqrt0}
    Fix $\chi\in(0,\frac{1}{20}]$, $\beta\in(0,\frac{1}{40}]$ and $\alpha\in(0,\frac{1}{40}]$.  There exist a constant $C$ and smooth extensions  $\widetilde g_i : B_{\frac 1 4 \Rk}'(\zk') \to \R$, with $\widetilde g_i\equiv g_i$ in $B_{\frac 1 4 \Rk}'(\zk')\cap\Omega_{2-\gamma}$  such that the following holds for every $k$ large enough: 
    For any given $\bar\zz\in \cZ\cap B_{\frac{1}{4}\Rk}(\zk)$, there is an affine function $\ell:\R^{n-1}\to\R$ such that
    \begin{equation}
    \label{eq:flatlaystar22}
     \frac{1}{(\epk^\chi\Rk)}\fint_{B_{\epk^\chi\Rk}'(\bar\zz')} \big| \widetilde g_i(x') - \ell \big| \, dx \le C\epk^{1+\chi/3} \qquad \mbox{for every}\quad i\in\{1,...,K_*\}\,,
    \end{equation}
    In particular, \eqref{eq:flatlaystar} follows.
\end{proposition}

\subsection{A Whitney-type extension}\label{sec:extension}
We construct new Whitney-type extensions which capture \eqref{eq:hjgagfas}. We need:
\begin{definition}[{\bf Projected bad set and distance}]
    In what follows, we set $N_\theta:= N\big(\theta,B_{\Rk}(\zk)\big)$. Moreover, for $x'\in\R^n$ we set $d(x'):={\rm dist}\big(x',[\cZ\cap B_{\Rk}(\zk)]'\big)$ and
    \begin{align*}
        N_\theta':=(\theta \Rk)^{-3}\Big|B_{\theta \Rk}'([\cZ\cap B_{\Rk}(\zk)]') \Big|=(\theta \Rk)^{-3}\Big|{\textstyle \bigcup}_{\widetilde\zz'\in [\cZ\cap B_{\Rk}(\zk)]'} B_{\theta \Rk}'(\widetilde\zz') \Big|\,.
    \end{align*}
\end{definition}
\begin{lemma}
    Let $\theta\in(0,1]$. We have $N_\theta'\leq CN_\theta$, so that
    \begin{equation}\label{eq:uibobdvadvb}
        |\{d\leq \theta \Rk\}\cap B_{\frac{1}{2}\Rk}'(\zk)|\leq C(\theta\Rk)^3N_\theta\,.
    \end{equation}
\end{lemma}
\begin{proof}
    Since $|\{x':d(x')\leq \theta \Rk\}\cap B_{\frac{1}{2}\Rk}'(\zk)|\leq (\theta\Rk)^3N_\theta'$ by definition, it suffices to show $N_\theta'\leq CN_\theta$. But indeed, projecting the simple result in \Cref{lem:Ndeflem}, we find that we can cover $B_{\theta \Rk}'([\cZ\cap B_{\Rk}(\zk)]')$ with at most $CN(\theta,B_{\Rk}(\zk))=CN_\theta$ projected bad balls of radius $5\theta\Rk$, which gives $N_\theta'\leq CN_\theta$.
\end{proof}
    
\noindent The main result of this section is:

\begin{proposition}[{\bf Extension}]\label{prop:extension}There exist smooth functions $\widetilde g_i:B_{\frac{1}{4}\Rk}'(\zk')\to \R$ such that the following hold:

\begin{enumerate}
    \item[(i)] $\widetilde g_i\equiv g_i$ in $B_{\frac{1}{4}\Rk}'(\zk') \cap \Omega_{2-\gamma}= B_{\frac{1}{4}\Rk}'(\zk')\setminus B_{\epk^{2-\gamma}\Rk}'([\cZ\cap B_{\Rk}(\zk)]')$.
    \item[(ii)] $\fint_{B_{\frac{1}{4}\Rk}'(\zk')}|\widetilde g_i-(\zk)_n|\leq C\epk\Rk$.
    \item[(iii)] We have $|\nabla \widetilde g_i|\le  o_k(1)$  in $B'_{\frac 1 4 \Rk}(\zk')$.

    \item[(iv)] The mean curvature of $\widetilde g_i$ has the following bound in $L^1$, in the region $B'_{\frac 1 4 \Rk}(\zk')\setminus \Omega_{2-\gamma}$ where $\widetilde g_i$ and $g_i$ differ:
    \begin{equation}\label{eq:euaglfgqd}
          \int_{B'_{\frac 1 4 \Rk}(\zk')\setminus \Omega_{2-\gamma}} |\HH[\widetilde g_i]|\leq C N_{\epk^{2-\gamma}}\,\epk^{2-\gamma}\Rk + C \sqrt{N_{\epk^{2-\gamma}}}\, \epk^{1+(2-\gamma)(3-\alpha)/2 } \Rk^2\,.
    \end{equation}   
   \end{enumerate}
\end{proposition}

\noindent {\bf Preliminaries.} Before giving the proof we introduce some definitions and a preparatory lemma. Set
\begin{equation}\label{C0-and-lstar}
    C_0 : = \ceils{\tfrac 12 \log_2 (n-1)} +3 =4\qquad  (n=4) \qquad \mbox{and} \qquad  l_\star : = \ceils{ -\log_2 \epk^{2-\gamma} }+3. 
\end{equation}
{\bf Dyadic cubes.} For $l\ge C_0$, consider the grid of {\em dyadic cubes} of the form
\begin{equation}\label{form_of_Q}
    Q =  2^{-l} \Rk([0,1)^{n-1} + y'),\quad\mbox{where } y'\in \mathbb Z^{n-1}\,.
\end{equation}
We say that $Q$ has {\em vertex} at $2^{-l} \Rk y'$. We define the {\em triple cube} $TQ$ by $TQ = 2^{-l} \Rk([-1,2)^{n-1} + y')$.\\
We denote $\boldsymbol Q^l:=\{Q\mbox{ of the form \eqref{form_of_Q}}: TQ\cap B_{\frac{1}{4}\Rk}'(\zk')\neq\emptyset\}$. We define the {\em predecessor} of $Q$, denoted $PQ$, as the cube in $\boldsymbol Q^{l-1}$ which contains $Q$.\\
{\it Note:} By the choice of $C_0$, since $l\ge C_0$, if $Q\in\boldsymbol Q^l$ then $TQ\subset B_{\frac{1}{2}\Rk}'(\zk')$. Furthermore, the Euclidean ball of radius equal to the diameter of $TQ$ centered at any point of $TQ$ is still contained in  $B_{\frac{1}{2}\Rk}'(\zk')$.\\
{\bf Cutoffs and partition of unity.} Fix a smooth cutoff $\eta_\circ\in C^\infty_c(  (-1/2,3/2)^{n-1})$ such that $\eta_\circ \ge 1$ in $[0,1]^{n-1}$. We define 
\[
\widetilde \eta_\circ(x') = 
\begin{cases}
\displaystyle\frac{\eta_\circ(x')}{\sum_{\zeta' \in \{-1,0,1\}^{n-1}}  \eta_\circ(x'+\zeta')} \quad &\mbox{ if }x'\in [-1,2]^{n-1}\,,
\vspace{3pt}
\\
0 & \mbox{ otherwise}\,,
\end{cases}
\quad\quad\quad\mbox{so that}\quad \sum_{\zeta' \in \Z^{n-1}} \widetilde\eta_\circ(\,\cdot\,+ \zeta') \equiv 1 \,.
\]
We can get a rescaled dyadic version in the obvious way: For $Q \in \boldsymbol Q^l$ with vertex $2^{-l} \Rk y'$, we put
\[
\eta_Q(x')  : = \widetilde \eta_\circ\left(\frac{x'}{2^{-l} \Rk } -y'\right)\,,\quad\quad\mbox{so that}\qquad \sum_{\widetilde Q\subset TQ} \eta_{\widetilde Q} \equiv 1 \mbox{ in }Q\quad \mbox{and}\quad   \sum_{Q\in \boldsymbol Q^l}\eta_Q \equiv 1 \mbox{ in } B'_{\frac 1 4 \Rk}(\zk)\,.
\]
We emphasise that $\eta_Q$ is supported in the triple cube $TQ$. Moreover, we obtain\footnote{Standard by scaling. The constant depends on $\widetilde\eta_\circ$, which is fixed depending on the dimension.}
\begin{equation}\label{dimboundsetas}
    |\eta_Q| +2^{-l}\Rk |\nabla \eta_Q| + (2^{-l}\Rk)^2 |D^2 \eta_Q| \le C(n)\quad \mbox{ for all }Q \in \boldsymbol Q^l.
\end{equation}
{\bf Bad dyadic cubes.} Recall that $2^{-4}\epk^{2-\gamma}\leq 2^{-l_\star}\leq 2^{-3}\epk^{2-\gamma}$. Let us define
\[
X :=  \bigcup \big\{ Q\in \boldsymbol Q^{l_\star} \  : \ TQ\cap [\cZ\cap B_{\Rk}(\zk)]' \neq \varnothing \big\} \quad \mbox{and}\quad   TX : = \bigcup \big\{ TQ\  :\ Q \in \boldsymbol Q^{l_\star}, \ Q\subset X \big\}.
\]
In particular, we have the inclusions\footnote{For the last inclusion, notice that any point in $TX$ must be at most at distance 
$ 3 \sqrt{n-1} \, 2^{-l_*}\Rk\leq 3 \sqrt{n-1} \, 2^{-3}\epk^{2-\gamma}\Rk< \epk^{2-\gamma}\Rk$ from some point in $\big[\cZ\cap B_{\Rk}(\zk)\big]'$. We have used that $3 \sqrt{n-1}\, 2^{-3}<1$ for $n=4$.}
\begin{equation}\label{inclusions_in_extpf}
B_{\frac{1}{4}\Rk}'(\zk') \cap B_{2^{-4}\epk^{2-\gamma}\Rk}\left(\big[\cZ\cap B_{\Rk}(\zk)\big]'\right) \subset T X \subset B_{\epk^{2-\gamma}\Rk}\left(\big[\cZ\cap B_{\Rk}(\zk)\big]'\right).
\end{equation}
We also consider a larger resolution version: For $C_0\le l \le l_*$ we define $\cX^l := \left\{Q\in \boldsymbol Q^l \, : \,   TQ\cap TX \neq \varnothing
\right\}$.

\begin{lemma}\label{lem:prep_extension}
    Let $C_0\le l \le l_\star$. To each $Q\in \cX^l$, we can assign an affine map $\ell_Q: \R^{n-1}\to \R$ with
    \begin{equation}\label{eq:8qgblgzudgfa}
         R_k^{-1}|\ell_Q(\zk') - (\zk)_n| + |\nabla \ell_Q| \le o_k(1)
    \end{equation}
    such that the following hold. Put 
    \begin{equation}\label{condintion_ell_Q}
        h_Q^2 : =  \frac{1}{(2^{-l}\Rk)^{n-1}}\int_{TQ\times\big((\zk)_n-\tfrac 1 4\Rk, (\zk)_n+\tfrac 1 4 \Rk\big)} |x_n-\ell_Q(x')|^2  \left[\frac{|\nabla u|^2}{2}+W(u)\right] dx' dx_n\,;
    \end{equation}
    assuming $TQ\cap \Omega_{2-\gamma/2}\neq\emptyset$ then
      \begin{equation}\label{condintion_ell_Q_g_i}
         \frac{1}{(2^{-l}\Rk)^{n-1}}\int_{TQ\cap \Omega_{2-\gamma/2}}|g_i -\ell_Q(x')|^2  dx'  \le  Ch_Q^2\quad \quad\mbox{for all }1\le i \le K_*\,.
    \end{equation}
    We then have the bound
    \begin{equation}\label{trappinghQ}
        \sum_{Q\in \cX^l} \left(\frac{h_Q}{2^{-l}\Rk}\right)^2 \le  C \left(\frac{\Rk}{2^{-l}\Rk}\right)^{1+\alpha} \epk^2.
    \end{equation}
    
    Moreover, we have the following comparisons:
    For $l\ge C_0+1$, we have
    \begin{equation}\label{comparison_previous}
        \| \ell_Q-\ell_{PQ}\|_{L^\infty(Q_1)} \le C(h_Q+ h_{PQ}) ;
     \end{equation}
     whenever $Q_1,Q_2 \in \cX^l$ and $TQ_1\cap TQ_2 \neq \varnothing$, then 
    \begin{equation}\label{comparison_contiguous}
        \| \ell_{Q_1}-\ell_{Q_2}\|_{L^\infty(Q_1)} \le C(h_{Q_1}+ h_{Q_2}) .
     \end{equation}
    Finally, for $Q\in \cX^{C_0}$ we can take $\ell_Q(x')= (\zk)_n$.
\end{lemma}
\begin{proof}
It is convenient to assume---after a translation---that $\zk=0$. Write $P=\left[\frac{|\nabla u|^2}{2}+W(u)\right]$. We divide the proof into three steps. 

\vspace{3pt}

\noindent {\bf Step 1.} We first establish the following claim: Let
 $R\in[\epk^{2-\gamma}\Rk,\Rk]$. For $k$ large enough, suppose that $\{\zz_j\}_{j\in J}\subset \cZ\cap B_{\frac{1}{2}\Rk}$ is a collection with $\{B_{R}'(\zz_j')\}_{j\in J}$  pairwise disjoint; in particular, the unprojected balls $\{B_{R}(\zz_j)\}_{j\in J}$ are pairwise disjoint too. We claim then that 
 \begin{equation}\label{cylinders1}
\sum_{j\in J}  \frac{1}{R^{n+1}} \int_{B'_{R}(\zz_j') \times \big(\tfrac \Rk 4, \tfrac \Rk 4\big)} |x_n-\ell_j(x')|^2
	P \le C(\Rk/R)^{1+\alpha}\epk^2 + C(\Rk/R)^{n} e^{-cR},
\end{equation}
where $\ell_j: \R^{n-1}\to \R$ are affine functions satisfying 
\begin{equation}\label{controlell}
    R_k^{-1}|\ell_j(0)| + |\nabla \ell_j| \le o_k(1).
\end{equation}

Indeed, \Cref{lem:sum_of_height_ex} gives unit vectors $v_j\in\Sp^{n-1}$ such that
\begin{equation}
\label{eq:uafaigaj}
|v_j-e_n|\le o_k(1) \quad \mbox{and}\quad \sum_{j\in J} {\bf H}_{4R}^2(v_j,\zz_j) \le C (\Rk/R)^{1+\alpha} \epk^2,
\end{equation}
where
$$
{\bf H}_{4R}^2(v_j,\zz_j)=\frac{1}{(4R)^{n+1}}
\int_{B_{4R}(\zz_j)}|v_j\cdot(x-\zz_j)|^2
	P
\,dx.
$$
Using \Cref{lem:Hbdhtez,lem:dxzggiuhg2}, we have 
\begin{equation}\label{cylinder_control}
   |(\zz_j)_n|=|\zz_j\cdot e_n| \le o_k(1) \Rk \quad\mbox{and}\quad\{|u|\le 0.9\} \subset  \{x \ :  \ |x_n-(\zz_j)_n|\le o_k(1) R\} \quad \mbox{in }B'_{4R}(\zz_j') \times \big(\tfrac \Rk 2 ,\tfrac \Rk 2  \big)\,.
\end{equation}   

\noindent Set $a_j=\frac{v_j-(v_j\cdot e_n)e_n}{v_j\cdot e_n}\in \R^{n-1}$ and $b_j=\frac{v_j\cdot\zz_j}{v_j\cdot e_n}\in\R$. Define $\ell_j(x')=a_j\cdot x'+b_j$, so that ${\rm graph}\, \ell_j=\{v_j\cdot (x-\zz_j)=0\}$.

Since $|\zz_j\cdot e_n| \le o_k(1) \Rk$ and $|v_j-e_n|\le o_k(1)$, we get \eqref{controlell}.

Moreover, notice that $|x_n-\ell_j(x')|\leq 2|v_j\cdot(x-\zz_j)|$  for all  $x\in\R^n$ (for $k$ large), thus we have:
$$
\frac{1 }{R^{n+1}}
\int_{B_{4R}(\zz_j)}|x_n-\ell_j(x')|^2
	P\,dx \le C{\bf H}_{4R}^2(v_j,\zz_j)\,.
$$
We estimate the rest by exponential decay: by \eqref{cylinder_control} and  \Cref{lem:expdecay}, we have
\[
P \le  Ce^{-cR} \quad \mbox{in }B'_{R}(\zz_j')\times \big(\tfrac \Rk 4, \tfrac \Rk 4\big) \setminus \{|x_n|  \le R  \}\,,
\]
and thus 
\[
\frac{1 }{R^{n+1}}
\int_{B'_{R}(\zz_j) \times \big(\tfrac \Rk 4, \tfrac \Rk 4\big)}|x_n-\ell_j(x')|^2
	P \le C {\bf H}_{4R}^2(v_j,\zz_j) + Ce^{-cR}.
\]

Then \eqref{cylinders1} follows recalling \eqref{eq:uafaigaj} and using that at most  $C(\Rk/R)^{n}$ disjoint balls of radius $R$ fit inside the cylinder $B'_{\R_k}\times(-\Rk, \Rk)$---a brutal bound that is sufficient here because it is multiplied by the exponential.

\vspace{3pt}

\noindent {\bf Step 2.}  We now combine the claim in Step 1 and a covering argument.

To each $Q \in \mathcal{X}^l$, assign a point $\zz_Q \in \mathcal{Z}$ such that $Q \subset B'_{60\, 2^{-l} \Rk}(\zz_Q')$. Applying the Besicovitch covering lemma to the family of balls $\{ B'_{60 \, 2^{-l} \Rk}(\zz_Q') \}_{Q \in \cX^l}$, we obtain a dimensional constant $b_n$ (with $n = 4$) and a decomposition into $b_n$ subfamilies, each consisting of pairwise disjoint balls, such that their union still covers the union of the cubes in $\cX^l$. By Step 1 (with $R= 2^{-l}\Rk$) applied to each of these subfamilies, \eqref{condintion_ell_Q} and \eqref{trappinghQ} follow\footnote{After noticing that for $l\le l_\star$ we have $(2^l)^n e^{-c2^{-l}\Rk } \le C (\ep_k)^{(\gamma-2)n} e^{-c\ep_k^{-\gamma}} \le C \epk^2$ for $k$ large enough (since $\epk^2 \Rk\ge 1$).}. 

Finally, \eqref{comparison_previous}-\eqref{comparison_contiguous} can be obtained as straightforward consequences of  \eqref{eq:aigpharspga}-\eqref{eq:aigpharspga2} in \Cref{lem:coefcompexc}.

\vspace{3pt}

\noindent {\bf Step 3.} Let us finally prove \eqref{condintion_ell_Q_g_i}.

Let $\lambda=\lambda(n)>0$ be such that, given any $x\in \{u=0\}$, $W(u)\geq \lambda$ in $B_\lambda(x)$. Put $R= 2^{-l}\Rk$. By Fubini,
\begin{align*}
h_Q^2 &\geq\frac{c}{R^{n+1}}
\int_{Q\cap\Omega_{2-\gamma}}\,dx'\int_{g_i(x')-\lambda}^{g_i(x')+\lambda}\,dx_n|x_n-\ell_Q(x')|^2
	P\\
&\geq\frac{c\lambda}{R^{n+1}}
\int_{Q\cap\Omega_{2-\gamma}}\,dx'\int_{g_i(x')-\lambda}^{g_i(x')+\lambda}\,dx_n|x_n-\ell_Q(x')|^2\geq\frac{c}{R^{n+1}}
\int_{Q\cap\Omega_{2-\gamma}}\,dx'|g_i(x')-\ell_j(x')|^2\,.
\end{align*} 
\end{proof}

We can now give the:
\begin{proof} [Proof of \Cref{prop:extension}]

For $C_0\le l \le l_\star$, for each cube $Q\in \cX^l$ let $\ell_Q:\R^{n-1}\to \R$  be the affine function provided by Lemma \ref{lem:prep_extension}.
In particular,
\begin{equation}\label{eq:keycubes}
 \sum_{Q\in \cX^l} (2^{-l} \Rk)^{-2-(n-1)} \int_{\Omega_{2-\gamma/2}\cap TQ} |g_i - \ell_Q|^2 \le C 2^{l(1+\alpha)}\epk^2   \quad\mbox{for all }1\le i \le K_*.
\end{equation}

We extend the definition of $\ell_Q$ to the rest of cubes $Q\in \boldsymbol Q^l\setminus \cX^l$:
For $Q\in \boldsymbol Q^{C_0}$ we set $\ell_Q(x') = (\zk)_n$, and for $Q\in \boldsymbol Q^l\setminus \cX^l$ with $l\ge C_0+1$ we define recursively $
 \ell_Q := \ell_{PQ} \quad \mbox{for } Q\in \boldsymbol Q^l\setminus \cX^l$.

\vspace{4pt}

\noindent {\bf Approximating functions:} For $l\ge C_0$, define\footnote{Notice that, in particular, the domain of $\psi_l$ contains the ball $B_{\frac{1}{4}\Rk}'(\zk')$.}
$\psi_l :  {\textstyle \bigcup \boldsymbol Q^l}   \to \R$ by $
\psi_{l} := \sum_{Q\in \boldsymbol Q^{l}}  \eta_{Q} \ell_Q$.
\vspace{2pt}

\noindent {\bf Definition of $\widetilde g_i$.} Consider the cutoff function $\eta_X = \sum_{Q\in \boldsymbol Q^{l_\star} , \ Q\subset X} \eta_Q$,
which satisfies 
\[
\eta_X\equiv  1 \quad \mbox{in  }X,  \qquad \eta_X\equiv 0 \quad \mbox{in }\R^{n-1}\setminus TX, \qquad   \epk^{2-\gamma}\Rk|\nabla \eta_X| + (\epk^{2-\gamma}\Rk)^2 |D^2\eta_X| \le C.
\]

We define $\widetilde g_i$ by gluing $\psi_{l_\star}$ and $g_i$ via a cutoff:
\[
\widetilde g_i  := \psi_{l_\star} \eta_X +  g_i(1-\eta_X )  = \psi_{l_\star} + \sum_{Q\in \boldsymbol Q^{l_\star}}   \eta_{Q} (1-\eta_X) (g_i -\ell_Q). 
\]
Notice that (i) holds automatically by \eqref{inclusions_in_extpf}. To see (ii), observe first that it holds for $g_i$ and over $B_{\frac{1}{4}\Rk}'(\zk')\cap \Omega_{2-\gamma}$ instead (directly by \eqref{condintion_ell_Q_g_i}, taking $l_Q(x')=(\zk)_n$ for $Q\in \boldsymbol Q^{C_0}$). Now, by \eqref{eq:8qgblgzudgfa} we get the brutal bound $|\widetilde g_i(x')-(\zk)_n|\leq o_k(1)\Rk$ in $B_{\frac{1}{4}\Rk}'(\zk')\setminus \Omega_{2-\gamma}$; using that this set is small easily gives (ii). By \eqref{eq:uagfabaf} and \eqref{eq:8qgblgzudgfa} we also get (iii).\\

\noindent {\bf Mean curvature estimate.} The main work is then in obtaining (iv).

\noindent We first focus on bounding $D^2 \psi_{l_\star}$. By construction, for all $l\ge C_0+1$ and $Q\in \boldsymbol Q^l$, we have 
\begin{equation}\label{support_2345uygv}
    (\psi_{l}-\psi_{l-1})|_{Q} \equiv   0\quad \mbox{whenever }TQ\cap {\textstyle \bigcup } (\cX^l) = \varnothing. 
\end{equation}

Then, combining \eqref{trappinghQ} (applied at scales $l$ and $l-1$) with \eqref{comparison_previous}-\eqref{comparison_contiguous} and using \eqref{dimboundsetas}, we obtain\footnote{To see this, a useful observation is the following: if $TQ_1 \cap (\bigcup \cX^l) \neq \varnothing$ for some $Q_1 \in \cX^l$, and if $Q_2 \in \boldsymbol{Q}^l$ is a contiguous cube such that $Q_2 \subset TQ_1$, then the triples of their respective predecessors, $TPQ_1$ and $TPQ_2$, must both belong to $\cX^{l-1}$ and necessarily intersect.}
\begin{equation}\label{eq:keycubes2}
\sum_{Q \in \boldsymbol{Q}^l} (2^{-l} \Rk)^{-2} \sup_{Q} |\psi_{l} - \psi_{l-1}|^2 \le C\, 2^{l(1+\alpha)} \epk^2.
\end{equation}

Similarly:
\begin{equation}\label{eq:keycubes4}
 \sum_{Q\in \boldsymbol Q^l}  (2^{-l} \Rk )^2\sup_{Q} |D^2(\psi_{l} - \psi_{l-1})|^2 \le C  2^{l(1 + \alpha)}  \epk^2\,.
\end{equation}

Then, Chebyshev's inequality gives
\begin{equation}\label{Chebyshev}
    \#\big\{Q \in \boldsymbol{Q}^l \ : \   \Rk \sup_Q |D^2(\psi_{l} - \psi_{l-1})| \ge t\big\} \le\, \frac{1}{t^2}  {2^{l(3 + \alpha)}}   \epk^2, \quad \mbox{for all }t>0\,.
\end{equation}

Next we use the above to show that for given $l$  with $C_0\le l\le l_*$, setting $N_{l_*} : = \# \cX^{l_*}$ we have
\begin{equation}\label{new_geom}
    \sum_{Q\in \cX^{l_*}}  \Rk \sup_Q|D^2(\psi_{l} - \psi_{l-1})|  \le 
    CN_{l_*}^{1/2} 2^{3l_*/2} 2^{l\alpha/2} \epk.
\end{equation}

We distinguish two cases. Notice that $2^{3(l_*-l)}$ is the number of cubes of $Q\in \cX^{l_*}$ that fit in a single cube of $\boldsymbol{Q}^l$. 

\vspace{4pt}
- {\em Case 1.} Suppose that $N_{l_*} \le 2^{3(l_*-l)}$. Then, we bound the maximum by the sum in \eqref{eq:keycubes4}, obtaining $\Rk\sup_{Q} |D^2(\psi_{l} - \psi_{l-1})| \le C  2^{l(3 + \alpha)/2}  \epk$ for every $Q\in \boldsymbol{Q}^l$. Using this bound in the $Q\in \cX^{l_*}$ and summing,
\[
\sum_{Q\in \cX^{l_*}}  \Rk \sup_Q |D^2(\psi_{l} - \psi_{l-1})|  \le C N_{l_*} 2^{l(3 + \alpha)/2} \epk\le CN_{l_*}^{1/2} 2^{3l_*/2} 2^{l\alpha/2} \epk\,.
\]

- {\em Case 2.} Suppose that $N_{l_*} \ge 2^{3(l_*-l)}$. The worse case scenario is then that the $N_{l_*}$ ``small'' cubes in $\cX^{l_*}$ keep filling the large cubes in $\boldsymbol{Q}^l$, starting from the ones contributing the most to the sum \eqref{eq:keycubes4}, and then continuing in decreasing order according to their contribution.  Let $p^*\in \N$ be such that  $2^{p_*-1}\le N_{l_*}/ 2^{3(l_*-l)} \le 2^{p_*}$. We then use \eqref{Chebyshev} with 
$\frac{1}{t^2}  {2^{l(3 + \alpha)}}  \, \epk^2 = 2^{p-1}$ to estimate the maximal possible contribution of the $2^{p-1}2^{3(l_*-l)}$ cubes that occupy positions $2^{p-1}2^{3(l_*-l)}$ to $2^p 2^{3(l_*-l)}$ in \eqref{new_geom}. Doing so we obtain:
\[
\sum_{Q\in \cX^{l_*}}  \Rk \sup_Q|D^2(\psi_{l} - \psi_{l-1})|   \le C\sum^{p_*}_{p=0} 2^{3(l_*-l)}2^p \frac{{2^{l(3 + \alpha)/2}}  \, \epk}{2^{p/2}} \le   C  2^{3(l_*-l)}  2^{p_*/2}2^{l(3+\alpha)/2}  \, \epk \le C N_{l_*}^{1/2}  2^{3l_*/2}2^{l\alpha/2} \epk.
\]
\vspace{4pt}

\noindent Combining the two cases gives \eqref{new_geom}. Summing \eqref{new_geom} for $C_0\le l \le l_\star$, and using that $\psi_{C_0}\equiv 0$, we reach: 
\[
 \sum_{Q\in \cX^{l_\star}}   \Rk \sup_{Q} |D^2\psi_{l_\star} | \le C N_{l_*}^{1/2} 2^{l_*(3+\alpha)/2} \epk\,.
\]
We now transform this into an integral bound.
Note that cubes $Q\in \boldsymbol Q^{l_\star}$ have volume $(2^{-l_\star}\Rk)^3 \le (2\epk^{2-\gamma}\Rk)^3$ and cover $B_{\frac{1}{4}\Rk}'(\zk')$. Moreover,  $2^{l_\star} \le \epk^{-(2-\gamma)}$ and $N_{l_*}\le CN_{\epk^{2-\gamma}}$. We then obtain
\begin{equation}\label{eq:keycubes5}
\begin{split}
\int_{\bigcup \cX^l}  \Rk|D^2\psi_{l_\star} |  &\le  \sum_{Q\in \cX^{l_\star}} 
(2\epk^{2-\gamma}\Rk)^3\sup_{Q}  \Rk|D^2\psi_{l_\star} | \le C N_{l_*}^{1/2}\epk^{-(2-\gamma)(3+\alpha)/2}\epk   (\epk^{2-\gamma}\Rk)^3 
\\ 
&\le C N_{\epk^{2-\gamma}}^{1/2} \epk^{1+(2-\gamma)(3-\alpha)/2 } \Rk^3\,.
\end{split}
\end{equation}
Notice that  $\HH[\widetilde{g}_i] = L_i \widetilde {g}_i:=  {\rm tr}( A_i(x') D^2 \widetilde{g}_i)$,
for some $|A_i(x')-{\rm Id}| \le  o_k(1)$ by (iii). We split our desired estimate into regions.\\
{\bf Region 1:} Since $\eta_X\equiv 1$ in $X$, by \eqref{eq:keycubes5} we reach
\[
\int_{X\cap B_{\frac{1}{4} \Rk'}(\zk') \setminus \Omega_{2 - \gamma}}  |\HH[\widetilde g_i] |= \int_{B'_{\frac{1}{4} \Rk}(\zk') \setminus \Omega_{2 - \gamma}} |L_i \psi_{l_\star} | \le C \int_{B'_{\frac{1}{4} \Rk'}(\zk') \setminus \Omega_{2 - \gamma}} |D^2 \psi_{l_\star}| \le  C N_{\epk^{2-\gamma}}^{1/2} \epk^{1+(2-\gamma)(3-\alpha)/2 } \Rk^2.
\]

\vspace{5pt}
The remaining estimates will involve $g_i$ as well.\\
\noindent {\bf Region 2:} Since $\eta_X\equiv 0$ in $\R^{n-1}\setminus TX$ and all points in this domain are at distance $\ge 2^{-4}\epk^{2-\gamma}\Rk$ from $[\cZ\cap B_{\Rk}(\zk)]'$, using \eqref{eq:meancurvsheet} and then \eqref{eq:uibobdvadvb} we obtain 
\[
\int_{(B'_{\frac{1}{4} \Rk'}(\zk') \setminus \Omega_{2 - \gamma}) \setminus TX} |\HH[\widetilde g_i] |
= 
\int_{(B'_{\frac{1}{4} \Rk'}(\zk') \setminus \Omega_{2 - \gamma}) \setminus TX} \big|\HH[g_i]\big|
\le C \frac{|B'_{\frac{1}{4} \Rk'}(\zk') \setminus \Omega_{2 - \gamma}|}{ (\epk^{2-\gamma}\Rk)^{2}}\leq CN_{\epk^{2-\gamma}}(\epk^{2-\gamma}\Rk)\,.
\]
\noindent {\bf Region 3:} We conclude by bounding $
\int_{TX\setminus X} |L_i (\widetilde g_i-\psi_{l_\star})|$. Observe that \eqref{eq:keycubes} implies that\footnote{To go from the first inequality to the second one, we use that the number of cubes is bounded by $CN_{\epk^{2-\gamma}}$, the inequality between the arithmetic and quadratic means, and the Cauchy--Schwarz inequality.}
\[
\sum_{\substack{Q\in \boldsymbol Q^{l_*}\\ Q\subset TX\setminus X}}   \frac{1}{(2^{-l_\star}\Rk)^2}\fint_{TQ}| g_i-\ell_Q|^2 \le C2^{l_\star(1+\alpha)}\epk^2\,,\quad\mbox{thus}\quad \sum_{\substack{Q\in \boldsymbol Q^{l_*}\\ Q\subset TX\setminus X}}   \frac{1}{(\epk^{2-\gamma}\Rk)}\fint_{TQ}| g_i-\ell_Q| \le CN_{\epk^{2-\gamma}}^{1/2} 2^{l_\star(1 +\alpha)/2}\epk\,.
\]
{\it Note:} Similar estimates hold replacing $\ell_Q$ by $\ell_{Q_1}$, where $Q_1$ is any adjacent cube with $Q_1\subset TQ$, by \eqref{comparison_contiguous}.

Since $g_i$ is almost a minimal graph, verifying 
$\|\HH[g_i]\|_{C^\theta} \le C (\epk^{2-\gamma}\Rk)^{-2}$ by \eqref{eq:meancurvsheet}, standard elliptic estimates give
\[
\sup_{Q} \bigg( \frac{|g_i -\ell_Q|}{\epk^{2-\gamma}\Rk} + |\nabla(g_i-\ell_Q)| + \epk^{2-\gamma}\Rk|D^2(g_i-\ell_Q)| \bigg)
\le  C \fint_{TQ}\frac{|g_i-\ell_Q|}{\epk^{2-\gamma}\Rk} +  \frac{C}{\epk^{2-\gamma}\Rk}\,,
\]
for every  cube $Q$ in  $\boldsymbol Q^{l_\star}$, with $Q\subset TX\setminus X$.

Using these estimates, noticing that in $TX\setminus X$ we have $\widetilde g_i- \psi_{l_\star} =  \sum_{Q\in \boldsymbol Q^{l_\star}}
\eta_Q (1-\eta_X)  (g_i-\ell_Q)$,
and using the standard product rules to compute the second derivatives of products we obtain: 
\[
\sum_{Q\subset TX\setminus X} \sup_{Q}\epk^{2-\gamma}\Rk |L_i (\widetilde g_i-\psi_{l_\star})| \le C N_{\epk^{2-\gamma}}^{1/2} 2^{l_\star(1+\alpha)/2}\epk   + \frac{C}{\epk^{2-\gamma}\Rk} \#\{Q\subset TX\setminus X\}.
\]

Finally, since the volume of each cube $Q\in \boldsymbol Q^{l_\star}$ is  $(2^{-l_\star}\Rk)^{n-1}$, with  $n=4$, and $2^{-\ell_*}$ is comparable to $\epk^{2-\gamma}$ we obtain:
\[
\begin{split}
\int_{TX\setminus X} |L_i (\widetilde g_i-\psi_{l_\star})|  & \le \frac{(C\epk^{2-\gamma}\Rk)^{n-1}}{\epk^{2-\gamma}\Rk} \bigg(N_{\epk^{2-\gamma}}^{1/2} 2^{l_\star(1 +\alpha)/2}\epk   + \frac{ \#\{Q\subset TX\setminus X\}}{(\epk^{2-\gamma}\Rk)} \bigg)
\\& \le C\Rk^2 N_{\epk^{2-\gamma}}^{1/2}\epk^{1+ (2-\gamma)(3-\alpha)/2} + CN_{\epk^{2-\gamma}}\epk^{2-\gamma}\Rk.
\end{split}
\]
Adding all the previous estimates completes the proof.
\end{proof}

\subsection{Linearised equation}\label{sec:linHeq}
Recall that $\gamma=\frac{1}{4}$. We show that our extended graphs are very harmonic in an $L^1$ sense.
\begin{proposition}[{\bf Linearised equation}]\label{prop:linHeq} Fix $\chi\in(0,\frac{1}{20}]$, $\beta\in(0,\frac{1}{40}]$ and $\alpha\in(0,\frac{1}{40}]$. For every $\bar \zz'\in (\cZ\cap B_{\frac{1}{8}\Rk}(\zk))'$ we have that
\begin{equation}\label{eq:ikwhlogbwo}
    \frac{1}{(\epk\Rk)}\rho^2\fint_{B_{\rho}'( \bar\zz')} |\HH[\widetilde g_i]|\leq C\epk^{\frac{1}{10}} \quad\mbox{for every}\quad \rho\in[\epk^{1+2\chi}\Rk,\frac{1}{8}\Rk]\,,
\end{equation}
provided that $k$ is large enough.
\end{proposition}
\begin{remark}
    This should be interpreted as follows: up to scale $\epk^{1 + 2\chi} \Rk$, the deviation of the $\widetilde g_i$ from an approximating minimal graph is small\footnote{In fact, it is of size $\epk^{1+\frac{1}{10}}\Rk$, thus smaller than the size of the $\widetilde g_i$ by rate of a small power $\epk^{\frac{1} {10}}$.} relative to the size of the $\widetilde g_i$ themselves, which is $\epk \Rk$. A key point is that our analysis penetrates the ``critical scale'' $\epk \Rk$, as explained below in the proof.
\end{remark}
\begin{proof}
We will show that
\begin{equation}\label{eq:vhwofiugwi}
    \int_{B_{\frac{1}{2}\Rk}'(\zk)} |\HH[\widetilde g_i]|\leq C \epk^{2-2\beta}\Rk^2
\end{equation}
and moreover
\begin{equation}\label{eq:vhwofiugwi2}
    \int_{B_{\epk^{1/2}\Rk}'( \bar\zz)} |\HH[\widetilde g_i]|\leq C\epk^{2+\frac{1}{5}}\Rk^2\quad\mbox{for every}\quad  \bar\zz'\in (\cZ\cap B_{\frac{1}{8}\Rk}(\zk))'\,.
\end{equation}
Combining them immediately implies \eqref{eq:ikwhlogbwo}: Indeed, considering \eqref{eq:vhwofiugwi} and the fact that $\beta\leq\frac{1}{8}$, we immediately see that
$$ \frac{1}{(\epk\Rk)}\rho^2\fint_{B_{\rho}'( \bar\zz)} |\HH[\widetilde g_i]|\leq C\epk^{1-\frac{1}{4}}\Rk\rho^{-1}\leq C\epk^{\frac{1}{4}}  \quad\mbox{for every}\quad \rho\in[\epk^{1/2}\Rk,\Rk]\,.$$
Likewise,multiplying \eqref{eq:vhwofiugwi2} by $\rho^{-1}$ and using $\chi\leq \frac{1}{20}$ we find that
$$\frac{1}{(\epk\Rk)}\rho^2\fint_{B_{\rho}'( \bar\zz)} |\HH[\widetilde g_i]|\leq C\epk^{1+\frac{1}{5}} \Rk\rho^{-1}\leq C\epk^{\frac{1}{10}} \Rk^2\quad\mbox{for every}\quad \rho\in[\epk^{1+2\chi}\Rk,\epk^{1/2}\Rk]\,.$$

It remains to establish  \eqref{eq:vhwofiugwi}-\eqref{eq:vhwofiugwi2}. We divide the proof into three steps.\\
\noindent {\bf Step 1.} Estimating $\HH[\widetilde g_i]$.\\
The second term in \eqref{eq:euaglfgqd} is small: By \eqref{eq:fhiowqrhgoqwg}, $\gamma=\frac{1}{4}$, and the bounds for $\alpha$ and $\beta$, we have
\begin{equation}
    \sqrt{N_{\epk^{2-\gamma}}}\, \epk^{1+(2-\gamma)(3-\alpha)/2 } \Rk^2\leq \epk^{-\frac{7}{8}(1+\frac{1}{40})}\epk^{1+\frac{7}{8}(3-\frac{1}{40}) } \Rk^2=\epk^{1+\frac{7}{8}(2-\frac{2}{40})}\Rk^2\leq \epk^{2+\frac{1}{5}}\Rk^2\,.
\end{equation}
Given $V\subset B_{\frac{1}{4}\Rk}(\zk)$, \eqref{eq:euaglfgqd} then gives
\begin{align*} 
\int_{V} |\HH[\widetilde g_i]|&\leq \int_{V\cap \Omega_{2-\gamma}} |\HH[g_i]|+C\epk^{2-\gamma}\Rk N_{\epk^{2-\gamma}}+ C \epk^{2+\frac{1}{5}}\Rk^2\,.
\end{align*}
By \Cref{thm:sheetimpc2alpha} we know that $|\HH[g_i]|\leq Cd^{-2}$, so that given $\theta\in(0,1/2]$, together with \eqref{eq:uibobdvadvb} we can bound
\begin{align*}  
\int_{\{\theta\Rk\leq d\leq 2\theta\Rk\}\cap B_{\frac{1}{2}\Rk}'(\zk')} |\HH[g_i]|\leq C\frac{(\theta\Rk)^3N_\theta}{(\theta \Rk)^2}= C\theta\Rk N_\theta\,.
\end{align*}
Denoting $\theta_j=2^{-j}$, and letting $1\leq j_1<j_2$ with $\theta_{j_2}\leq \epk^{2-\gamma}\leq \theta_{j_2-1}$, by summing we find that
\begin{equation*}
\int_{\{\epk^{2-\gamma}\Rk\leq d\leq \theta_{j_1}\Rk\}\cap B_{\frac{1}{2}\Rk}'(\zk')} |\HH[g_i]|\leq C\Rk\sum_{j_1\leq j\leq j_2}\theta_j N_{\theta_j}\,,
\end{equation*}
or (letting $V=\{d\leq \theta_{j_1}\Rk\}\cap B_{\frac{1}{2}\Rk}'(\zk)$, so that $V\cap \Omega_{2-\gamma}=\{\epk^{2-\gamma}\Rk\leq d\leq \theta_{j_1}\Rk\}\cap B_{\frac{1}{2}\Rk}'(\zk)$)
\begin{equation}\label{eq:qfiou2} 
\int_{\{d\leq \theta_{j_1}\Rk\}\cap B_{\frac{1}{2}\Rk}'(\zk')} |\HH[\widetilde g_i]|
\leq C\Rk\sum_{j_1\leq j\leq j_2}\theta_j N_{\theta_j}+ C \epk^{2+\frac{1}{5}}\Rk^2\,.
\end{equation}

\noindent {\bf Step 2.} Proof of \eqref{eq:vhwofiugwi}.\\
We will use the following bounds:
\begin{itemize}
    \item $N_\theta\leq \theta^{-1-\beta}$. This follows from \eqref{eq:fhiowqrhgoqwg}.
    \item $c\leq \Rk\epk^2$. This follows from \eqref{eq:awioufhiu4n} with $R=4\Rk$ combined with ${\bf H}_{4\Rk}^2(\zk)\leq 2\epk^2$ (see \Cref{lem:NbdRknew}).
\end{itemize}
Choosing $j_1=1$ in \eqref{eq:qfiou2},
\begin{equation}\label{eq:90wahgoiah}
    \int_{B_{\frac{1}{2}\Rk}'(\zk)} |\HH[\widetilde g_i]|
\leq C\Rk\theta_{j_2}^{-\beta}+C \epk^{2+\frac{1}{5}}\Rk^2\leq C(\Rk\epk^2)\Rk\epk^{-\beta(2-\gamma)}+C \epk^{2+\frac{1}{5}}\Rk^2\leq C\epk^{2-2\beta}\Rk^2\,,
\end{equation}
as long as $k$ is large enough.\\

We notice that at scale $\Rk$ the estimate \eqref{eq:90wahgoiah} is sharp (as, even in the absence of bad cubes, the Wang-Wei estimates do not allow to improve this).
One can see (similarly as above) that this estimate can imply \eqref{eq:ikwhlogbwo} at most for $\rho \gg \epk \Rk$. To go below this ``critical threshold''---as required later in our proofs---we need to run a dichotomy argument.

\noindent {\bf Step 3.} Proof of \eqref{eq:vhwofiugwi2}.\\

\noindent {\bf Case 1.} Assume that $|B_1(\cZ)\cap B_{3\Rk/4}(\zk)|\leq\epk^{-\frac{1}{4}}$; we claim that $N_{\theta}\leq \theta^{-1}\epk^{1/4}$ for $\theta\in(\frac{1}{\Rk},\epk^{1/2}]$.

Indeed, by assumption $N_{\frac{1}{\Rk}}\leq C\epk^{-\frac{1}{4}}$; since the property $ N_{\theta_2}\leq C N_{\theta_1}$ for $0<\theta_1<\theta_2\leq 1$ is always true\footnote{This can be seen using first \Cref{lem:Ndeflem}, which gives $\{B_{5\theta_1 \Rk}(\widetilde\zz_i)\}_{i=1}^{Q_1}$ with $Q_1\leq CN_{\theta_1}$ such that $B_{\theta_1 \Rk}(\cZ\cap B_{\Rk}(\zk))\subset \bigcup_{i=1}^{Q_1} B_{5\theta_1 \Rk}(\widetilde\zz_i)$, and then enlarging these balls to cover $B_{\theta_2 \Rk}(\cZ\cap B_{\Rk}(\zk))$.}, we find that $N_{\theta}\leq C\epk^{-\frac{1}{4}}$ for $\theta\in [\frac{1}{\Rk},\epk^{1/2}]$ as well, which using $\theta\leq \epk^{1/2}$ gives $N_{\theta}\leq \theta^{-1}\epk^{1/4}$.

Then, choosing $j_1$ with $\theta_{j_1}\leq \epk^{1/2}\leq \theta_{j_1-1}$ in \eqref{eq:qfiou2}, using $N_{\theta}\leq \theta^{-1}\epk^{1/4}$ instead of $N_\theta\leq \theta^{-1-\beta}$ we find
\begin{equation*}
    \int_{\{d\leq \epk^{1/2}\Rk\}\cap B_{\frac{1}{2}\Rk}'(\zk')} |\HH[\widetilde g_i]|
\leq C\Rk\epk^{1/4}\log{\theta_{j_2}} +C\epk^{2+\frac{1}{5}}\Rk^2\leq C\Rk\epk^{\frac{1}{5}}\,.
\end{equation*}
Since obviously $B_{\epk^{1/2}\Rk}(\bar\zz)\subset \{d\leq \epk^{1/2}\Rk\}\cap B_{\frac{1}{2}\Rk}'(\zk')$, we conclude \eqref{eq:vhwofiugwi2} in this case.\\

\noindent {\bf Case 2.} Assume we are not in Case 1, so that $|B_1(\cZ)\cap B_{3\Rk/4}(\zk)|>\epk^{-\frac{1}{4}}$ instead; we claim that $c\leq \Rk\epk^{2+\frac{1}{4}}$.

Indeed, by \eqref{eq:awioufhiunoe} with $R=\Rk$ and $n=4$, we deduce that $\frac{\epk^{-1/4}}{\Rk}\leq C{\bf H}_{4\Rk}^2(\zk)$, which combined with ${\bf H}_{4\Rk}^2(\zk)\leq 2\epk^2$ from \Cref{lem:NbdRknew} shows $c\leq \Rk\epk^{2+\frac{1}{4}}$.

Then, computing as in \eqref{eq:90wahgoiah} we directly obtain an extra factor $\epk^{1/4}$ in the second inequality, so that in particular
\begin{equation*}
     \int_{\{d\leq \epk^{1/2}\Rk\}\cap B_{\frac{1}{2}\Rk}'(\zk')} |\HH[\widetilde g_i]|
\leq C(\Rk\epk^{2+1/4})\Rk\epk^{-\beta(2-\gamma)}+C \epk^{2+\frac{1}{5}}\Rk^2\leq C\epk^{2+1/4-2\beta}\Rk^2\,.
\end{equation*}
This gives \eqref{eq:vhwofiugwi2} as before as long as $1/4-2\beta\geq 1/5$, which is precisely the condition $\beta\in(0,\frac{1}{40}]$.
\end{proof}
\subsection{Improvement for each layer}

\begin{proposition}
\label{prop:decaysqrt}
Fix $\chi\in(0,\frac{1}{20}]$, $\beta\in(0,\frac{1}{40}]$ and $\alpha\in(0,\frac{1}{40}]$. Given any $\bar \zz\in B_{\Rk/8}(\zk)\cap \cZ$, there are coefficients $a_i\in \R^{n-1},b_i\in\R$ for every $i\in\{1,...,K_*\}$ such that, setting $\ell_i:=a_i\cdot(x'-\bar\zz')+b_i$, we have that
\begin{equation}
    \label{eq:flatlaystardiffi}
    \frac{1}{(\epk^\chi \Rk)}\fint_{ B_{\epk^\chi\Rk}'(\bar\zz')} \big| \widetilde g_i(x') - \ell_i \big| \, dx \le C\epk^{1+\chi/2}, \qquad \mbox{for every}\quad i\in\{1,...,K_*\},
\end{equation}
provided $k$ is large enough.
\end{proposition}
We first need a standard lemma, which we prove in \Cref{app:standres}.

\begin{lemma}[{\bf Harmonic approximation}]
\label{lem:linapprox}
Let $p > 1$ and $d \ge 0$. Given $\lambda > 0$, there exists $\delta = \delta(\lambda, n, p, d)>0$ such that the following holds. 

Let $v\in C^2(B_{1/\delta}')$, $B_{1/\delta}'\subset\R^{n-1}$, be such that
\[
\begin{split}
\rho^2\fint_{B_\rho'} |{\rm div}(A\nabla v)| <\delta\quad\text{and}\quad \fint_{B_\rho'} |v| & \le \rho^{d+1/2}\quad \text{for}\quad 1\le \rho \le \tfrac{1}{\delta}\,,
\end{split} 
\]
where $A\in C^1(B_{1/\delta}')$ and $|A(x')-1|\leq \delta$.\\
Then, 
\[
\int_{B_1'} |v-p_d|\,dx \le \lambda,
\]
where $p_d$ is a harmonic polynomial of degree $\le d$  and such that $\|p_d\|_{L^1(B_1)}\le |B_1'|$.
\end{lemma}

\begin{proof}[Proof of \Cref{prop:decaysqrt}]
We will show the existence of some $\bar C\geq 1$ such that the following holds:
Let $\bar\zz\in B_{\frac{1}{4}\Rk}(\zk)$. Fix $i\in\{1,...,K_*\}$. Then, we have the algebraic decay
\begin{equation}
\label{eq:prop62p}
\min_{a_i \in B_1', \, b_i \in \mathbb{R}}  \frac{1}{R}\fint_{B_R'(\bar \zz')} \big| \widetilde g_i(x') - a_i \cdot (x' - \bar\zz') - b_i \big| \, dx'\le \bar C \epk\left(\frac{R}{\Rk}\right)^{1/2}\,,
\end{equation}
for all $R\in  \left[\epk^\chi\Rk, \frac{1}{8}\Rk\right]$.

The desired statement immediately follows.\\

Let us denote $R_l := 2^{-l}\Rk$. We assume that \eqref{eq:prop62p} holds for $R = R_1, R_2, R_3, \dots, R_l$ and we will show it holds for $R\in [R_{l+1}, R_l]$ as well, as long as $l \ge l_0$ for some $l_0$ universal to be chosen, and $2^{-l}\ge \epk^\cttc$. Notice that, up to making $\bar C$ larger depending only on $l_0$, we can always assume the statement holds up to $R_{l_0}$ indeed, since (ii) in \Cref{prop:extension} shows that
\begin{equation}
    \label{eq:powfjopwv}
     \fint_{B_{\frac{1}{4}\Rk}'(\zk')} \big| \widetilde g_i(x') - (\zk)_n \big| \, dx \le C\epk\Rk\,.
\end{equation}

Denote $h:=\widetilde g_i(\bar\zz'+x')$ in what follows. By assumption, for every $1\le m \le l$ there are $a_m,b_m$ such that
\begin{equation}\label{eq:iubfgsdbg}
    \frac{1}{R_m}\fint_{B_{R_m}'} |h-a_m\cdot x'-b_m|\leq \bar C 2^{-m/2}\epk\,.
\end{equation}

By the triangle inequality,
\[
|a_m-a_{m+1}| +\frac{1}{R_m}|b_m-b_{m+1}| \le C_1 \bar C 2^{-m/2}\epk,\quad\mbox{therefore}\quad |a_m-a_l|+ \frac{1}{R_m}|b_m-b_l| \le C_2 \bar C 2^{-m/2}\epk.
\]
Then, for $l-l_0\le m \le l$, \eqref{eq:iubfgsdbg} transforms into
\[
\frac{1}{R_m}\fint_{B_{R_m}'} |h-a_l\cdot x'-b_l|\le    C_3 \bar C2^{-m/2}\epk, 
\]
for some universal $C_3$. If we define 
\[
\tilde h_l(x'):= \frac{h (R_{l} x')-R_{l}\,a_l\cdot x'-b_l}{C_3 \bar C \epk2^{-l/2} R_l},
\]
we get  
\begin{equation}
    \label{eq:twocond}
  \fint_{ B_{2^{l-m}}'} |\tilde h_l|\le  \frac{2^{-m/2}R_m}{2^{-l/2}R_l}=2^{3/2(l-m)},\quad\mbox{or}\quad \fint_{ B_{\rho}'} |\tilde h_l|\le C_4\rho^{3/2} \qquad\text{for $\rho \in  (1, 2^{l_0})$.}
\end{equation}
\vspace{0.4cm}

On the other hand, letting $A=\frac{1}{\sqrt{1+|\nabla h|^2}}$, \Cref{prop:linHeq} and (iii) in \Cref{prop:extension} give that
\begin{align*}
    \frac{1}{(\epk\Rk)}R_m^2\fint_{B_{R_m}'} |{\rm div}(A\nabla h)|\leq C\epk^{\frac{1}{10}}\,,\quad \mbox{where}\quad |A(x')-1|\leq o_k(1)\,,
\end{align*}
which letting $\widetilde A_l(x'):=A(R_l x')$ and rescaling gives that
\begin{align*}
    \rho^2\fint_{B_{\rho}'} |{\rm div}(\widetilde A_l\nabla \tilde h_l)|\leq C\epk^{\frac{1}{10}}2^{3l/2}\leq C\epk^{\frac{1}{40}}\,,\quad \mbox{where}\quad |\widetilde A_l(x')-1|\leq o_k(1)\,.
\end{align*}
We have used that $2^{3/2l}\leq(\epk^{-\chi})^{3/2}\leq \epk^{-\frac{3}{40}}$ for $\chi\in (0,\frac{1}{20}]$ and $k$ large enough in the last inequality.

By \Cref{lem:linapprox} with $d=1$, given $\lambda>0$, for $k$ and $l_0$ large enough we have that $\int_{B_1'} |\tilde h_l - a\cdot x' - b|\le \lambda$, for some $a\in \R^{n-1}$ and $b\in \R$. Scaling back, we have 
\[
\frac{1}{R_l} \fint_{B_{R_l}'} |h(x) -a_l\cdot x' - b_l - C_3\bar C  2^{-l/2}\epk a\cdot x' - C_3\bar C  R_l 2^{-l/2}\epk b|\, dx'\le   C_4\bar C  2^{-l/2}\lambda \epk\,.
\]
Denoting $b_{l+1} := b_l + C_3\bar C  R_l 2^{-l/2}\epk\, b$ and $a_{l+1} = a_l + C_3\bar C  2^{-l/2}\epk\, a$, we obtain 
\[
\frac{1}{R_l}\fint_{B_{R_l}'} |h(x) -a_{l+1}\cdot x' - b_{l+1}|\, dx'\le  C_4 \bar C 2^{-l/2}\lambda\epk.
\]
We now choose $\lambda$ small so that $C_4\lambda \le \frac{1}{32}$, which in turn fixes $l_0$, $\bar C$, and a lower bound for $k$. We then get \eqref{eq:prop62p} for $R\in[R_{l+1},R_l]$ as desired.
\end{proof}

\subsection{Preservation of mean oscillation}
We can interpret \Cref{prop:decaysqrt} as saying that $\widetilde g_{i}-l_i$ has average size $O(\epk^{1+3\chi/2}\Rk)$ in $B_{\epk^\chi\Rk}(\bar \zz')$. We now show that $\widetilde g_{i}-l_i$ still has average size $O(\epk^{1+3\chi/2}\Rk)$ in the much smaller ball $B_{\epk^{1+2\chi}\Rk}(\bar \zz')$.
\begin{proposition}[{\bf Preservation of layer mean oscillation up to $\epk^{1+2\chi}\Rk$}] \label{prop:decaysqrt3}
Fix $\chi\in(0,\frac{1}{20}]$, $\beta\in(0,\frac{1}{40}]$ and $\alpha\in(0,\frac{1}{40}]$.

    Then, we have
    \begin{equation}\label{eq:weakestdecay}
     \fint_{B_{\epk^{1+2\chi} \Rk}'(\bar\zz')} |\widetilde g_i-\ell_i| \le C \epk^{1+\cttc/2} (\epk^\chi\Rk)\quad\mbox{for all}\quad i\in \{1,...,K_*\}\,,
    \end{equation}
provided $k$ is large enough.
\end{proposition}

\begin{proof}
We argue similarly to the proof of \Cref{prop:decaysqrt}. We fix $i\in\{1,..,K_*\}$ in all of the proof. Let $a_i\in \R^{n-1},b_i\in\R$ be given by \Cref{prop:decaysqrt}, so that setting $\ell_i:=a_i\cdot(x'-\bar\zz')+b_i$ we have that
    \begin{equation}
    \label{eq:ihohbogbzbg}
     \fint_{B_{\epk^\chi\Rk}'(\bar\zz')} \big| \widetilde g_i - \ell_i \big| \, dx \le C\epk^{1+\chi/2}(\epk^\chi\Rk).
    \end{equation}
    Set $w :=  (\widetilde g_i-\ell_i)(\bar\zz' +\, \cdot\,)$. To obtain \eqref{eq:weakestdecay}, we will actually show that for some (tiny) ${\bar \beta}>0$ and some $\bar C$, we have 
\begin{equation}\label{weakdecay11}
   \inf_{c\in \R} \fint_{B_r'} |w-c|  \le \bar C \epk^{1+\cttc/2} (\epk^\chi\Rk)  \left(r/(\epk^\chi\Rk)\right)^{\bar \beta} \quad \mbox{for all}\quad r \in [\epk^{1+2\cttc} \Rk, \epk^\chi\Rk]\,.
\end{equation}
Then, a direct application of the triangle inequality and the geometric decay in \eqref{weakdecay11} give that
\begin{equation}\label{weakdecay12}
   \fint_{B_r'} |w|  \le C\bar C \epk^{1+\cttc/2} (\epk^\chi\Rk)  \quad \mbox{for all}\quad r \in [\epk^{1+2\cttc} \Rk, \epk^\chi\Rk]
\end{equation}
as well, as desired.\\
Set $r_l := 2^{-l}(\epk^\chi\Rk)$, and assume that \eqref{weakdecay11} holds for $r = r_1, r_2, r_3, \dots, r_l$. We will show it holds for $r\in [r_{l+1}, r_l]$ as well as long as $l \ge l_0$ and $2^{-l}\ge \epk^{1+\chi}$.\\
Now, under these assumptions\footnote{Up to making $\bar C$ larger depending on $l_0$, by \eqref{eq:ihohbogbzbg} this will be the case.}, we have $c_m\in \R$ such that, for any $1\le m \le l$,
\begin{equation}\label{industep2}
w_{m}(x) := w(x) -c_m\quad\mbox{satisfies}\quad\fint_{B_{r_m}'} |w_{m}| \le \bar C 2^{-{\bar \beta} m} \epk^{1+\cttc/2} (\epk^\chi\Rk).
\end{equation}
The triangle inequality shows then that $|c_{l}-c_m|\le C_1 \bar C 2^{-{\bar \beta} m}\epk^{1+\cttc/2}  (\epk^\chi\Rk)$ as well, where $C_1$ depends on $\bar\beta$. This gives
\begin{equation}\label{industep22}
\fint_{B_{r_m}'} |w_l| \le C_2\bar C 2^{-{\bar \beta} m} \epk^{1+\cttc/2} (\epk^\chi\Rk)\,,
\end{equation}
which defining $\tilde w_l(x):= \frac{w_l (r_{l} x)}{ C_2\widetilde C_0\epk^{1+\cttc/2} (\epk^\chi\Rk) 2^{-{\bar \beta} l}}$ transforms (taking $\bar\beta<\frac{1}{2}$) into
\begin{equation}
    \label{eq:twocondbis}
\fint_{B_{2^{l-m}}'} |\tilde w_l|\le 2^{\bar\beta (l-m)}\,,\quad\mbox{or}\quad   \fint_{B_{\rho/2}'} |\tilde w_l|\le \rho^{1/2}\quad \text{for}\quad \rho \in  (1, 2^{l_0}).
\end{equation}
Likewise, \Cref{prop:linHeq} and (iii) in \Cref{prop:extension} give that
\begin{align*}
    \rho^2\fint_{B_{\rho}'} |{\rm div}(\tilde A_l\nabla \tilde h_l)|\leq C\epk^{\frac{1}{10}}\epk^{-3\chi/2}2^{\bar\beta l}\leq C\epk^{\frac{1}{80}}\,,\quad \mbox{where}\quad |\widetilde A_l(x')-1|\leq o_k(1)\,.
\end{align*}
We have used that $\epk^{-3\chi/2}2^{\bar\beta l}\leq \epk^{-3/40}(\epk^{-2})^{\bar\beta}\leq \epk^{-\frac{7}{80}}$ up to choosing $\bar\beta\in (0,\frac{1}{160}]$ and $k$ large enough in the last inequality.

We can then apply \Cref{lem:linapprox} with $d=0$ and obtain that for any $\lambda>0$, as long as $l_0$ and $k$ are large enough we have that $\int_{B_{1}'} |\tilde w_l - c|\, dx\le \lambda$, for some $c\in \R$.  Choosing $\lambda$ small enough, after rescaling we conclude that \eqref{weakdecay11} holds for $r\in [r_{l+1}, r_l]$ as well, concluding the proof.
\end{proof}

\subsection{Approximation by a single linear function -- Proof of \Cref{prop:impflatsinglin}}
Combining the previous results and the proximity of our graphs at small scales \eqref{eq:omega_modintro}, we can now give:
\begin{proof}[Proof of \Cref{prop:impflatsinglin} and \Cref{prop:decaysqrt0}]
We will show \Cref{prop:decaysqrt0}, which is a strictly stronger result. To simplify the notation, after a translation we can (and do) assume that $\bar \zz=0$. We divide the proof into three steps. 
    \medskip 

\noindent\textbf{Step 1.} Improvement of the $|b_i|$.\\
Here is where we will use \eqref{eq:weakestdecay}. Let $\rho=\epk^{1+2\cttc} \Rk$. By \eqref{eq:weakestdecay} and the fact that $\widetilde g_i\equiv g_i$ in $\Omega_{2-\gamma}$, we deduce that
\begin{equation}\label{eq:flatlayclean22}
    \frac{1}{|B_{\rho}'|}\int_{B_{\rho}'\cap\Omega_{2-\gamma}} \big| g_i(x') - a_i \cdot x' - b_i \big| \, dx' \le C\epk^{1+\chi/2}(\epk^\chi\Rk) , \qquad i\in\{1,...,K_*\}.
    \end{equation}
Now, by \Cref{lem:graphlargeflat} we can bound (for $k$ large enough)
$$
\frac{1}{|B_{\rho}'|}\int_{B_{\rho}'\cap \Omega_{2-\gamma}} \big| g_i(x')\big| \, dx' \le o_k(1)\rho \leq \epk^{1+\chi/2}(\epk^\chi\Rk)\quad \mbox{for every} \qquad i\in\{1,...,K_*\};$$
intuitively, all of the graphs pass inside $B_\rho$ thanks to \eqref{eq:omega_modintro}. Together with \eqref{eq:flatlayclean22}, then
\begin{equation}
    \frac{1}{|B_{\rho}'|}\int_{B_{\rho}'\cap\Omega_{2-\gamma}} \big|a_i \cdot x' + b_i \big| \, dx' \le C\epk^{1+\chi/2}(\epk^\chi\Rk)\qquad \mbox{for every} \quad i\in\{1,...,K_*\}\quad\mbox{as well}.
    \end{equation}
Restricting to either $\{a_i \cdot x'>0\}$ or $\{a_i \cdot x'<0\}$ in the integral immediately gives that 
\begin{equation}
    \big|b_i\big| \le C\epk^{1+\chi/2}(\epk^\chi\Rk)\,,
\end{equation}
as long as $|B_{\rho}'\cap\Omega_{2-\gamma}|\geq 0.9|B_{\rho}'|$. But this is indeed true: Using \eqref{eq:uibobdvadvb} (and the notation there), by \eqref{eq:fhiowqrhgoqwg} we can bound
\begin{equation}\label{eq:iagsaogbaovg}
    |\{d\leq \epk^{2-\gamma} \Rk\}\cap B_{\frac{1}{2}\Rk}'(\zk)|\leq(\epk^{2-\gamma}\Rk)^3(N_{\epk^{2-\gamma}})^{-1-\beta}\leq C(\epk^{2-\gamma}\Rk)^3(\epk^{2-\gamma})^{-1-\beta}=C\epk^{(2-\gamma)(2-\beta)}\Rk^3\,,
\end{equation}
which for $k$ large enough and $\gamma\in(0,\frac{1}{4}]$, $\beta\in(0,\frac{1}{16}]$, $\chi \in(0,\frac{1}{20}]$ is strictly smaller than $|B_{\rho}'|=c\rho^3=c\epk^{3(1+2\chi)}\Rk^3$.
\medskip

\noindent\textbf{Step 2.} Improvement of the $|a_j-a_i|$.\\
Let $\rho=\epk^\chi\Rk$ now; all of our next arguments will be at scale $\epk^\chi\Rk$. By \eqref{eq:flatlaystar},
\begin{equation}
    \label{eq:lduydjg}
    \frac{1}{|B_{\rho}'|}\int_{B_\rho'\cap \Omega_{2-\gamma}} \big| g_i(x') - a_i \cdot x' - b_i \big| \, dx \le C\epk^{1+\chi/2}\rho , \qquad i\in\{1,...,K_*\}.
    \end{equation}
Given $i<j$ we know that $g_j-g_i>0$. Applying the triangle inequality and \eqref{eq:lduydjg}, and then Step 1,
\begin{equation*}
    \frac{1}{|B_\rho'|}\int_{B_\rho'\cap \Omega_{2-\gamma}} (a_i-a_j) \cdot x' + (b_i-b_j) \, dx' \le C\epk^{1+\chi/2}\rho\,,\quad\mbox{thus}\quad \frac{1}{|B_\rho'|}\int_{B_\rho'\cap \Omega_{2-\gamma}} (a_i-a_j) \cdot x' \, dx' \le C\epk^{1+\chi/2}\rho \,.
\end{equation*}
Assuming that $a_i\neq a_j$, setting $A:=B_{\rho}'\cap \{\frac{a_i-a_j}{|a_i-a_j|}\cdot x'\geq \frac{\rho}{10}\}$ we find that
\begin{equation*}
    \frac{|A\cap \Omega_{2-\gamma}|}{|B_{\rho}'|}|a_i-a_j|\frac{\rho}{10}\leq\frac{1}{|B_{\rho}'|}\int_{A\cap \Omega_{2-\gamma}} (a_i-a_j) \cdot x' \, dx \le C\epk^{1+\chi/2}\rho\,,\quad\mbox{or}\quad \frac{|A\cap \Omega_{2-\gamma}|}{|B_{\rho}'|}|a_i-a_j|\leq C\epk^{1+\chi/2}\,.
\end{equation*}
Arguing exactly as in the end of Step 1, we see that $|A\cap\Omega_{2-\gamma}|\geq 0.9|A| = c|B_{\rho}'|$ for $k$ large enough, thus we conclude that
\begin{equation}
    \label{eq:lduydjg6}
    |a_i-a_j|\leq C\epk^{1+\chi/2}\,.
\end{equation}

\medskip

\noindent\textbf{Step 3.} Conclusion.\\
Steps 1 and 2 show in particular that $\frac{1}{\epk^\chi\Rk}\fint_{B_{\epk^\chi \Rk}'(\bar y')}|\ell_i-\ell_1|\leq C\epk^{1+\chi/2}$ for every $i\in\{1,...,K_*\}$, so that simply setting $\ell:=\ell_1$ we see that \eqref{eq:flatlaystardiffi} transforms into \eqref{eq:flatlaystar22}, showing \Cref{prop:decaysqrt0}. Since (by \Cref{prop:extension}) we have that $g_i\equiv \widetilde g_i$ in $\Omega_{2-\gamma}$, then \eqref{eq:flatlaystar} immediately follows, proving \Cref{prop:impflatsinglin} as desired.
\end{proof}

\section{Conclusion of argument}\label{sec:conclusion}
For simplicity, let us fix $\chi=\frac{1}{20}$, $\beta=\frac{1}{40}$ and $\alpha=\frac{1}{40}$ in what follows.
\subsection{Improvement of Allen--Cahn excess and density deficit}\label{sec:impexcdens}
We want to transform the $L^1$ bounds from \Cref{prop:impflatsinglin} into $L^\infty$ ones, for which we first need:
\begin{lemma}[{\bf Clean cylinder}]\label{lem:cleancyl}
    There exists $\bar\zz\in B_{\Rk/8}(\zk)\cap \cZ$ with the following property: There is some ${\bf \bar y}'$ such that
    $${\rm dist}({\bf \bar y'},[\cZ\cap B_{\Rk}(\zk)]')=|{\bf \bar y'}-\bar\zz'|=\frac{1}{2}\epk^\chi\Rk\,.$$
\end{lemma}
\begin{proof}
    Using \eqref{eq:uibobdvadvb} and \eqref{eq:fhiowqrhgoqwg}, we can bound
    $$|\{d\leq \frac{1}{2}\epk^\chi \Rk\}\cap B_{\frac{1}{2}\Rk}'(\zk)|\leq(\frac{1}{2}\epk^\chi\Rk)^3(N_{\frac{1}{2}\epk^\chi})^{-1-\beta}\leq C(\frac{1}{2}\epk^\chi\Rk)^3(\frac{1}{2}\epk^\chi)^{-1-\beta}=C\epk^{(2-\beta)\chi}\Rk^3\,.$$
    For $k$ large enough this is strictly smaller than $|B_{\frac{1}{16}\Rk}'(\zk)|=c\Rk^3$, thus $\{d> \frac{1}{2}\epk^\chi \Rk\}\cap B_{\frac{1}{16}\Rk}'(\zk)\neq\emptyset$. This clearly means that there is ${\bf \bar y'}\in \{d= \frac{1}{2}\epk^\chi \Rk\}\cap B_{\frac{1}{16}\Rk}'(\zk)$ as well, and then we can choose some $\bar \zz$ such that $|{\bf \bar y'}-\bar\zz'|=d({\bf \bar y'})=\frac{1}{2}\epk^\chi \Rk$.
\end{proof}
Our main result here is:
\begin{proposition}[{\bf Improvement of total excess and density deficit}]
\label{prop:ACexcdens}
There exist $\bar\zz$ and ${\bf \bar y'}$ (given by \Cref{lem:cleancyl}) such that the following holds. Let $\ell$ be given by \Cref{prop:impflatsinglin} (applied with this $\bar\zz$). Setting ${\bf \bar y}:=({\bf \bar y}',\ell({\bf \bar y}'))$, then---for $k$ large enough--- ${\bf \bar y}\in B_{\epk^\chi\Rk}(\bar \zz)$. Moreover, we have:
\begin{itemize}
    \item Height excess bound:\qquad \ \ ${\bf H}_{(\epk^\chi\Rk)/64}^2({\bf \bar y}) \le C \epk^{2+ \cttc}$.
    \item Density deficit bound:\qquad $K_* - C \left[\frac{\ep_k^{1+4\chi/3}}{\theta}\right]^2 \leq {\bf M}_{\theta\Rk}({\bf \bar y}) \leq K_*\quad\mbox{for every}\quad \theta\in[\epk,\epk^\chi/8]$.
\end{itemize}
\end{proposition}
 
\begin{proof}
The fact that ${\bf \bar y}\in B_{\epk^\chi\Rk}(\bar \zz)$ for $k$ large enough follows easily from \Cref{prop:impflatsinglin} and \Cref{lem:graphlargeflat}.

\noindent {\bf Step 1.} Improvement of ${\bf H}$.\\
From \Cref{prop:impflatsinglin} and $|{\bf \bar y}'-\bar\zz'|=\frac{1}{2}\epk^{\chi}\Rk$, for each $i\in\{1,...,K_*\}$ we have $\frac{1}{(\epk^\chi\Rk)}\fint_{B_{\frac{1}{2}\epk^\chi\Rk}'({\bf \bar y'})} |g_i-\ell| \le C \epk^{1+\cttc/2}$.\\
Since $B_{\frac{1}{2}\epk^\chi\Rk}'({\bf \bar y'})$ is a (projected) good ball, by elliptic estimates (using \eqref{eq:meancurvsheet} with $R= \frac 12 \epk^\chi \Rk$) we can upgrade this into
\[
\frac{1}{(\epk^\chi\Rk)}|g_i-\ell| 
\le C \epk^{1+\cttc/2} \quad\text{in}\quad B_{\frac{1}{4}\epk^\chi\Rk}'({\bf \bar y'}).
\]
Set ${\bf \bar y}_n:=\ell({\bf \bar y'})$; choosing an appropriate Euclidean frame, we can assume that ${\bf \bar y}=0$ and $l\equiv 0$ (up to restricting to a smaller cylinder). Then
\begin{align*}
    \{u=0\}=\bigcup_{i=1}^{K_*} {\rm graph}\,g_i\subset \{|x_n|\leq \epk^{1+\cttc/2} (\epk^\chi\Rk)\} \quad\mbox{in the cylinder}\quad \mathcal C_{(\epk^\chi\Rk)/6}\,.
\end{align*}
Using the second bullet in \Cref{lem:addproperties} (since $c\leq \epk^2\Rk\ll \epk^{1+\cttc/3} (\epk^\chi\Rk)$) this can be upgraded to:
$$\{|u|\leq 0.9\}\subset \{|x_n|\leq C\epk^{1+\cttc/2} (\epk^\chi\Rk)\} \quad\mbox{in}\quad \mathcal C_{(\epk^\chi\Rk)/8}.$$
We can then apply \Cref{lem:hconthexc} with $\delta=\epk^{1+\chi/2}$ and $\lambda=1/4$. Choosing $\chi>0$ small enough so that---since $cR_k^{-1}\le \epk^2$ by \eqref{eq:awioufhiu4n}---we have $\delta^2\gg \epk^{(2-\chi)(3/2)}\geq c(\epk^\chi\Rk)^{-2(1-\lambda)}$, for $k$ large, we find that ${\bf H}_{\frac{1}{64}(\epk^\chi\Rk)}^2(\bar{\bf y})\leq C\epk^{2+\cttc}$.

\noindent {\bf Step 2.} Largest cylinder.

Fix $\theta\in[\epk,\epk^\chi/8]$, and set $\bar\ep_k:=\epk^{1+\chi/3}\gg \epk^{1+\chi/2}$. We want to consider a cylinder in $B_{\theta \Rk}$ containing as much of the area of our graphs as possible: Set $\mathcal C^\theta:=B_{a_\theta\Rk}'\times [-\bar\ep_k(\epk^\chi\Rk),\bar\ep_k(\epk^\chi\Rk)]$ (recall ), where
$$
a_\theta=\sqrt{\theta^2-[\bar\ep_k\epk^\chi]^2}=\theta\sqrt{1-[\frac{\bar\ep_k\epk^\chi}{\theta}]^2}\,,\quad\mbox{thus}\quad [a_\theta\Rk]^2+[\bar\ep_k(\epk^\chi\Rk)]^2=\theta^2\Rk^2\,.
$$
Then, we see that
$$\bigcup_i {\rm graph}\, g_i|_{B_{a_\theta\Rk}'}\subset \{|x_n|\leq C\epk^{1+\cttc/2} (\epk^\chi\Rk)\} \cap \mathcal C^\theta \subset B_{\theta\Rk/8}\,.$$
Moreover, observe that
\begin{equation}\label{eq:0hipgbabggar}
    |B_{a_\theta \Rk}'|=|B_{\theta\sqrt{1-\Big[\frac{\bar\ep_k\epk^\chi}{\theta}\Big]^2}\Rk}'|=|B_{\theta\Rk}'|\left(1-\Big[\frac{\bar\ep_k\epk^\chi}{\theta}\Big]^2\right)^{\frac{n-1}{2}}\geq |B_{\theta\Rk}'|\left(1-C\Big[\frac{\bar\ep_k\epk^\chi}{\theta}\Big]^2\right)\,.
\end{equation}

\noindent {\bf Step 3.} Slicing argument---we show that each sheet of $\{u=0\}$, $x_n = g_i (x')$, forces a contribution of almost $
\frac{|B_{a_\theta \Rk}'|}{\omega_{n-1}}$ (with a quantified error) to the total A--C energy of $u$ in $B_{\theta \Rk}$.\\
Let $\Lambda:=\frac{1+3/4}{4}\sqrt{2}\log{\frac{(\epk^\chi\Rk)}{8}}$, and define for every $i\in\{1,...,K_*\}$ the ``transition regions'' 
$$L_i^{-}:=\left\{(x',x_n):x'\in B_{a_\theta\Rk}'\,,\,x_n\in\left(g_i(x')-\Lambda, g_i(x')\right)\right\}$$
and
$$L_i^{+}:=\left\{(x',x_n):x'\in B_{a_\theta\Rk}'\,,\,x_n\in\left(g_i(x'),g_i(x')+\Lambda\right)\right\}\,.$$
They are all contained in $\mathcal C^\theta$. The point is that, by \eqref{eq:wwsepbnd}, we know that $g_{i+1}-g_i\geq \Lambda$, thus $L_i^+\cap L_{i+1}^-=\emptyset$ for every $1\leq i<K_*$ (eventually in $k$). This shows that the transition regions are all disjoint (for $k$ large enough).\\
Fix $i\in\{1,...,K_*\}$; applying Fubini and $|\nabla u|^2\geq u_{x_n}^2$, we see that
\begin{align*}
    \sigma_{n-1}\mathcal E(u,L_i^+)=\int_{B_{a_\theta\Rk}'}dx'\int_{g_i(x')}^{g_i(x')+\Lambda}dx_n\, \left[\frac{|\nabla u|^2}{2} +W(u)\right]\geq \int_{B_{a_\theta\Rk}'}dx'\int_{g_i(x')}^{g_i(x')+\Lambda}dx_n\, \left[\frac{u_{x_n}^2}{2} +W(u)\right]\,.
\end{align*}
Define the 1D function $f_{x'}(t):=u\left(x',g_i(x')+t\right)$, $t\in [0,\Lambda]$. Then, letting $t:=x_n-g_i(x')$ in the integral,
\begin{align*}
    \sigma_{n-1}\mathcal E(u,L_i^+)&\geq \int_{B_{a_\theta\Rk}'}dx'\int_{0}^{\Lambda}dt\,  \left[\frac{(\frac{d}{dt}f_{x'})^2(t)}{2} +W(f_{x'}(t))\right]\,.
\end{align*}
Fixed $x'$, the inner integral is (up to a factor $\frac{1}{\sigma_0}$) precisely the A--C energy in $[0,\Lambda]$ of the 1D function $f_{x'}(t)$. Moreover, $f_{x'}(0)=u(x',g_i(x'))=0$. Now, $f_{x'}(t)$ is not necessarily a solution to A--C in 1D; on the other hand, letting $\kappa=3/4$ and $R=\frac{(\epk^\chi\Rk)}{8}$, its A--C energy is at least
$$P(\kappa,R):= \inf\left\{\mathcal E_{1D}\left (v,\frac{1+\kappa}{4}\sqrt{2}\log{R}\right ):v\in C_c^1([0,\frac{1+\kappa}{4}\sqrt{2}\log{R}])\text{ and } v(0)=0\right\},
$$
where we denote \[\mathcal E_{1D}(v,r)=\frac{1}{\sigma_0}\int_0^r \frac{1}{2} (v'(t))^2 + W(v(t))\,dt\,,\] 
and the infimum is attained by a true 1D Allen--Cahn solution. In \Cref{prop:1Dasympt}, it is shown that $P(\kappa,R)\geq \frac{1}{2}-CR^{-(1+\kappa)}$, which since $\Rk\geq c\epk^{-2}$ (by \eqref{eq:awioufhiu4n}) shows that $P(\kappa,R)\geq \frac{1}{2}-C\epk^{(2-\chi)(1+\kappa)} \geq \frac{1}{2}-C\epk^{3}$.\\
Combined with the above and \eqref{eq:0hipgbabggar}, we find that
\begin{align*}
    \omega_{n-1}\mathcal E(u,L_i^+)&\geq \int_{x'\in B_{a_\theta\Rk}'}P\left(\kappa,R\right)\,dx'\geq \left|B_{a_\theta\Rk}'\right|\left[\frac{1}{2}-C\epk^{3}\right]\geq |B_{\theta\Rk}'|\left[\frac{1}{2}-C[\frac{\bar\ep_k\epk^\chi}{\theta}]^2\right]\,,
\end{align*}
so that
\begin{align*}
    {\bf M}_{\theta\Rk}=\frac{\omega_{n-1}}{|B_{\theta\Rk}'|}\mathcal E(u, B_{\frac{(\epk^\chi\Rk)}{8}})\geq \frac{\omega_{n-1}}{|B_{\theta\Rk}'|}\sum_{i=1}^{K_*}\left[\mathcal E(u,L_i^-)+\mathcal E(u,L_i^+)\right]\geq K_*-C[\frac{\bar\ep_k\epk^\chi}{\theta}]^2
\end{align*}
as desired.
\end{proof}
\subsection{A monotonicity-type formula: Density deficit controls growth of excess}
\begin{definition}[{\bf Weighted excess and density}]\label{def:heathexc} Let $f_{r,d}:\R\to\R$ and $G_{r,d}:\R^n\to\R$ be defined by
$$
    f_{r,d}(t):=\frac{1}{c_d r^d}e^{\frac{-t^2}{r^2}}\quad\mbox{and}\quad G_{r,d}(x):=f_{r,d}(|x|)\,,\qquad\mbox{where}\quad c_d=d\int_0^\infty t^{d-1}e^{-t^2}\,.
$$
We set
    $${\bf \widetilde H}_r^2(e):=\int_{\R^n} (x\cdot e)^2\left[\frac{|\nabla u|^2}{2}+W(u)\right]G_{r,n+1},\quad {\bf \widetilde M}_r^2:=\frac{1}{\sigma_{n-1}}\int_{\R^n} \left[\frac{|\nabla u|^2}{2}+W(u)\right]G_{r,n-1}\,.$$
    Moreover, we denote ${\bf \widetilde H}_r^2:=\inf_{e\in\Sp^{n-1}}{\bf \widetilde H}_r^2(e)$, ${\bf \widetilde H}_r^2(e,x_0):={\bf \widetilde H}_r^2(u(\cdot-x_0),e)$, and ${\bf \widetilde M}_r^2(x_0):={\bf \widetilde M}_r^2(u(\cdot-x_0))$.
\end{definition}
The weighted versions are just as good for our purposes, up to a going to a slightly smaller scale:
\begin{proposition}[{\bf Improvement of weighted excess and density deficit}]\label{lem:compweighHM}
    There exists ${\bf \bar y}$ (given by \Cref{prop:ACexcdens}) such that the following holds. Let $\chi_\sharp=\frac{31}{30}\chi$. Then, we have
\begin{equation}
        {\bf \widetilde H}_{\epk^{\chi_\sharp}\Rk}^2({\bf \bar y}) \leq C\epk^{2+\chi_\sharp/2}\quad\mbox{and}\quad K_*- \epk^{2+\chi_\sharp/2}\leq {\bf \widetilde M}_{\epk^{\chi_\sharp}\Rk}({\bf \bar y})\leq  {\bf \widetilde M}_{\infty}= K_*
\end{equation}
for $k$ large enough.
\end{proposition}
We will show this later. We first state the crucial monotonicity-type formula that will allow to propagate the smallness of ${\bf \widetilde H}$ from scale $\epk^{\chi_\sharp} \Rk$ to the ``original scale'' $R_k$.
\begin{theorem}[\textbf{Density deficit controls excess growth}]\label{prop:moncontrexc}
    For any $r>0$ and  $\lambda\in(0,\frac{1}{2})$, we have\footnote{This result works for any A--C solution $u:\R^n\to[-1,1]$, and not just for the critical solution and $n=4$.}
    \begin{equation}\label{excmonineq}
        {\bf \widetilde H}_r\leq {\bf \widetilde H}_{\lambda r}+C|\log(\lambda)|^{1/2}({\bf \widetilde M}_r-{\bf \widetilde M}_{\lambda r})^{1/2}\,.
    \end{equation} 
\end{theorem}
\begin{proof}
It suffices to show that for every $e\in\Sp^{n-1}$ and $\lambda\in(0,1)$, we have
\begin{equation}\label{excmonineq2}
        {\bf \widetilde H}_r(e)\leq {\bf \widetilde H}_{\lambda r}(e)+C[1+|\log(\lambda)|^{1/2}]({\bf \widetilde M}_r-{\bf \widetilde M}_{\lambda r})^{1/2}\,.
    \end{equation}
Up to choosing the right frame we can assume that $e=e_n$. Define $P:=\left[\frac{|\nabla u|^2}{2}+W(u)\right]$ and $Q:=\left[W(u)-\frac{|\nabla u|^2}{2}\right]$.

We will use the following:
\begin{itemize}
    \item Kernel properties: \qquad $
        \nabla G_{r,d}=\frac{-2}{r^2}G_{r,d} x\quad\mbox{and}\quad r\frac{d}{dr} G_{r,d}=-x\cdot \nabla G_{r,d}-dG_{r,d}\,.$
    \item Pohozaev identity\footnote{Proved via a standard computation using $\Delta u=W'(u)$.}: \qquad \ \ ${\rm div}(Px)={\rm div}((x\cdot \nabla u)\nabla u)+nP-|\nabla u|^2$.
\item $\nabla x_n^2 =2x_n e_n$\,.
\end{itemize}
\textbf{Step 1.} Pohozaev--type computations.\\ 
Observe that
\begin{align*}
     r\frac{d}{dr} {\bf\widetilde H}_r^2(e_n)=r\frac{d}{dr} \int x_n^2PG_{r,n+1}&=\int x_n^2P\left(-x\cdot \nabla G_r-(n+1)G_r\right)\,.
\end{align*}
Integrating by parts,
\begin{align*}
     r\frac{d}{dr} {\bf\widetilde H}_r^2(e_n)&=\int\nabla x_n^2\cdot xPG_{r,n+1}+\int x_n^2{\rm div}(Px)G_{r,n+1}-(n+1)\int x_n^2 P G_r\\
    &=2\int x_n^2PG_{r,n+1}+\int x_n^2\left[{\rm div}((x\cdot \nabla u)\nabla u)-P-|\nabla u|^2\right]G_{r,n+1}\\
    &=\int x_n^2QG_{r,n+1}+\int x_n^2{\rm div}((x\cdot \nabla u)\nabla u)G_{r,n+1}\,.
\end{align*}

Integrating the new divergence term by parts, relating $G_{r,n+1}$ and $G_{r,n-1}$, and using Cauchy--Schwarz we arrive at
\begin{align*}
      r\frac{d}{dr} {\bf\widetilde H}_r^2(e_n)&=\int x_n^2[Q+\frac{2}{r^2}(x\cdot \nabla u)^2]G_{r,n+1}-2\int x_nG_{r,n+1}(x\cdot \nabla u)u_n\\
    &=\frac{c_1}{r^2}\int x_n^2[Q+\frac{2}{r^2}(x\cdot \nabla u)^2]G_{r,n-1}-2c_2\int (x_nu_n G_{r,n+1}^{1/2})(\frac{x\cdot \nabla u}{r}G_{r,n-1}^{1/2})\\
    &\leq C\int [Q+\frac{2}{r^2}(x\cdot \nabla u)^2]G_{r,n-1}+C[\int x_n^2P G_{r,n+1}]^{1/2}[\int \frac{1}{r^2}(x\cdot \nabla u)^2G_{r,n-1}]^{1/2}\,.
\end{align*}
Now, analogous computations show that
\begin{equation}\label{eq:weighmonM}
      r\frac{d}{dr} {\bf\widetilde M}_r= \frac{1}{\sigma_{n-1}}\int [Q+\frac{2}{r^2}(x\cdot \nabla u)^2]G_{r,n-1}\,.
\end{equation}
In particular, ${\bf \widetilde M}_r$ is monotone nondecreasing by \Cref{lem:modicaineq}. Moreover, we see that
\begin{align*}
    r\frac{d}{dr} {\bf\widetilde H}_r^2(e_n)\leq Cr\frac{d}{dr} {\bf\widetilde M}_r+C[{\bf\widetilde H}_r^2(e_n)]^{1/2}[r\frac{d}{dr} {\bf\widetilde M}_r]^{1/2}\,.
\end{align*}

\textbf{Step 2.} Gr\"onwall-type argument.\\
Given $0<a<r$, integrating we find that
\begin{align*}
    {\bf\widetilde H}_r^2(e_n)\leq {\bf\widetilde H}_a^2(e_n)+C[{\bf\widetilde M}_r-{\bf\widetilde M}_a]+C\int_a^r[{\bf\widetilde H}_r^2(e_n)]^{1/2}[\frac{1}{r}\frac{d}{dr} {\bf\widetilde M}_r]^{1/2}=:{\bf A}_r\,.
\end{align*}

We then see that
$$\frac{d}{dr}{\bf A}_r= C\frac{d}{dr} {\bf\widetilde M}_r+C[{\bf\widetilde H}_r^2(e_n)]^{1/2}[r\frac{d}{dr} {\bf\widetilde M}_r]^{1/2}\leq  C\frac{d}{dr} {\bf\widetilde M}_r+C[{\bf\widetilde A}_r]^{1/2}[\frac{1}{r}\frac{d}{dr} {\bf\widetilde M}_r]^{1/2}\,,$$
thus dividing by ${\bf A}_r^{1/2}$ and using ${\bf A}_r\geq C[{\bf\widetilde M}_r-{\bf\widetilde M}_a]$ we reach\footnote{We can assume that $[{\bf\widetilde M}_r-{\bf\widetilde M}_a]>0$, as otherwise \eqref{eq:weighmonM} forces $u$ to be constant along rays, thus $u\equiv u(0)$.}
\begin{align*}
    2\frac{d}{dr}{\bf A}_r^{1/2}\leq C\frac{\frac{d}{dr} {\bf\widetilde M}_r}{{\bf A}_r^{1/2}}+C[\frac{1}{r}\frac{d}{dr} {\bf\widetilde M}_r]^{1/2}&\leq C\frac{\frac{d}{dr} [{\bf\widetilde M}_r-{\bf\widetilde M}_a]}{[{\bf\widetilde M}_r-{\bf\widetilde M}_a]^{1/2}}+C[\frac{1}{r}\frac{d}{dr} {\bf\widetilde M}_r]^{1/2}=2C\frac{d}{dr} [{\bf\widetilde M}_r-{\bf\widetilde M}_a]^{1/2}+C[\frac{1}{r}\frac{d}{dr} {\bf\widetilde M}_r]^
{1/2}\,.
\end{align*}

Integrating and using Cauchy--Schwarz, we find that 
\begin{align*}
    {\bf A}_r^{1/2}-{\bf A}_a^{1/2}&\leq C[{\bf\widetilde M}_r-{\bf\widetilde M}_a]^{1/2}+C[\int_a^r\frac{1}{s}ds]^{1/2}[\int_a^r\frac{d}{ds} {\bf\widetilde M}_s ds]^{1/2} = C[{\bf\widetilde M}_r-{\bf\widetilde M}_a]^{1/2}+C\log^{1/2}{(r/a)}[{\bf\widetilde M}_r-{\bf\widetilde M}_a]^{1/2}\,.
\end{align*}
Using that ${\bf H}_r(e_n)\leq {\bf A}_r^{1/2}$ and ${\bf H}_a(e_n)={\bf A}_a^{1/2}$, we conclude that
\begin{equation*}
    {\bf \widetilde H}_r(e_n)\leq {\bf \widetilde H}_a(e_n)+C[1+\log^{1/2}{(r/a)}]({\bf \widetilde M}_r-{\bf \widetilde M}_a)^{1/2}\,,
\end{equation*}
which putting $a=\lambda r$ gives \eqref{excmonineq2}.
\end{proof}

We will now show \Cref{lem:compweighHM}. We first need a couple of lemmas.
\begin{lemma}[{\bf Density properties}]\label{lem:weighdensprop}
    ${\bf \widetilde M}_r$ is monotone nondecreasing. Moreover, we have the layer cake-type formula
\begin{equation}\label{eq:laycakedensweight}
    {\bf \widetilde M}_{r}=-\int_0^\infty dt\, t^{n-1}f_{r,n-1}'(t){\bf  M}_{t}\,,
\end{equation}
where $-\int_0^\infty dt\, t^{n-1}f_{r,n-1}'(t)=1$.
\end{lemma}
\begin{proof}
    We already discussed the monotonicity of ${\bf \widetilde M}_r$ (recall \eqref{eq:weighmonM}). Moreover, we can compute
    $$-\int_0^\infty dt\, t^{n-1}f_{r,n-1}'(t)=(n-1)\int_0^\infty dt\, t^{n-2}f_{r,n-1}(t)=\frac{n-1}{c_nr^{n-1}}\int_0^\infty dt\, t^{n-2}e^{\frac{-t^2}{r^2}}=\frac{n-1}{c_n} \int_0^\infty dt\, t^{n-2}e^{-t^2}=1\,.$$
    Finally, letting $P:=\left[\frac{|\nabla u|^2}{2}+W(u)\right]$, using polar coordinates and $G_{r,d}(x)=f_{r,d}(|x|)$ we can compute
\begin{align*}
    \int P G_{r,n-1}=\int_0^\infty dt\, f_{r,n-1}(t)\int_{\partial B_t} P=\int_0^\infty dt\, f_{r,n-1}(t) \frac{d}{dt}\int_{B_t} P=-\int_0^\infty dt\, t^{n-1}f_{r,n-1}'(t)\frac{1}{t^{n-1}}\int_{B_t} P\,,
\end{align*}
which gives precisely \eqref{eq:laycakedensweight}.
\end{proof}

\begin{lemma}[{\bf Height excess properties}]\label{lem:compH}
We have
\begin{equation}\label{eq:8gsjgbss}
        {\bf H}_{r}^2\leq C{\bf \widetilde H}_{r}^2\,.
    \end{equation}
Moreover, given $\lambda\in(0,1)$ we have
    \begin{equation}\label{eq:8gsjgbss2}
        {\bf \widetilde H}_{\lambda r}^2 \leq \lambda^{-(n+1)} C{\bf H}_{r}^2+Ce^{-\frac{1}{2\lambda}}{\bf M}_{\infty}\,.
    \end{equation}
\end{lemma}
\begin{proof}
    We clearly have the pointwise inequalities
    \begin{equation*}
        \frac{1}{r^{n+1}}\chi_{B_{r}}\leq CG_{r,n+1}\quad\mbox{and}\quad G_{\lambda r}\leq C\frac{1}{(\lambda r)^{n+1}}\chi_{B_r}+C\frac{1}{(\lambda r)^{n+1}}e^{-\frac{|x|^2}{\lambda^2 r^2}}\chi_{B_r^c}\,.
    \end{equation*}
    Integrating the first one we get \eqref{eq:8gsjgbss}. Integrating the second one we get
    \begin{equation*}
        {\bf \widetilde H}_{\lambda r}^2 \leq \lambda^{-(n+1)} C{\bf H}_{r}^2+C\lambda^{-(n+1)}\frac{1}{r^{n+1}}\int_{B_r^c} |x|^2\left[\frac{|\nabla u|^2}{2}+W(u)\right]e^{-\frac{|x|^2}{\lambda^2 r^2}}\,,
    \end{equation*}
    which splitting the integral into dyadic scales and using that ${\bf M}_R\leq {\bf M}_\infty$ easily gives \eqref{eq:8gsjgbss2}.
\end{proof}

\begin{proof}[Proof of \Cref{lem:compweighHM}]
Up to a translation we can assume that ${\bf \bar y}=0$. Taking $\lambda=\epk^{\frac{1}{30}\chi}$ in \eqref{eq:8gsjgbss2}, by \Cref{prop:ACexcdens} we find that
    \begin{equation}
        {\bf \widetilde H}_{\epk^{\chi_\sharp}\Rk}^2 \leq \epk^{-\chi/6} C{\bf H}_{\epk^\chi \Rk}^2+Ce^{-\epk^{\chi/60}}\leq C\epk^{2+\chi_\sharp/2}\,.
\end{equation}
It remains to estimate ${\bf \widetilde M}_{\epk^{\chi_\sharp}\Rk}$.
Since $-\int_0^\infty dt\, t^{n-1}f_{r,n-1}'(t)=1$ and ${\bf M}_t\xrightarrow[]{t\to\infty}{\bf M}_\infty$, it is clear that
$${\bf \widetilde M}_\infty=-\lim_{r\to\infty}\int_0^\infty dt\, t^{n-1}f_{r,n-1}'(t){\bf M}_t={\bf M}_\infty=K_*\,.$$\\
For the lower bound, using $-\int_0^\infty dt\, t^{n-1}f_{r,n-1}'(t)=1$ again and the monotonicity of ${\bf M}_t$, given $0<a<r$ we can estimate
\begin{align*}
    K_*-{\bf \widetilde M}_r&=-\int_0^\infty t^{n-1}f_{r,n-1}'(t) [K_*-{\bf M}_t]dt\\
    &\leq  -\int_{r}^\infty t^{n-1}f_{r,n-1}'(t) [K_*-{\bf M}_{r}]dt-\int_0^{r} t^{n-1}f_{r,n-1}'(t) [K_*-{\bf M}_t]dt\\
    &\leq [K_*-{\bf M}_r]-\left[\int_0^{a}+\int_a^{r}\right] t^{n-1}f_{r,n-1}'(t) [K_*-{\bf M}_t]dt\,.
\end{align*}
Bounding $f_{r,n-1}'(t)\leq C\frac{t}{r^{n+1}}$ in the integrals, taking $r=\epk^{\chi_\sharp} \Rk$ and $a=\epk\Rk$, and using the density lower bound from \Cref{prop:ACexcdens}, we reach
\begin{align*}
    K_*-{\bf \widetilde M}_r&\leq [K_*-{\bf M}_{r}] +C(a/r)^{n+1}+\frac{C}{r^{n+1}}\int_{a}^{r} t^n [K_*-{\bf M}_t]dt\\
    &\leq (\ep_k^{1+4\chi/3}\Rk/r)^2 +C(a/r)^{n+1}+\frac{C}{r^{n+1}}\int_{a}^{r} t^n (\ep_k^{1+4\chi/3}\Rk/t)^2dt\leq (\ep_k^{1+4\chi/3}\Rk/r)^2 +C(a/r)^{n+1}\,,
\end{align*}
thus $K_*-{\bf \widetilde M}_r\leq \epk^{2+\chi_\sharp/2}$.
\end{proof}

\subsection{Back to the original scale -- Proof of \Cref{mainthm}}

\begin{proof}[Proof of \Cref{mainthm}]
Proving \Cref{mainthm} is equivalent to showing that every $K\in\R_+$ is a subcritical density in $\R^4$ (recall \Cref{def:induction}). We assume then for contradiction that
$K_* = \sup \{K>0 \mbox{ subcritical density in } \R^4\}<\infty$.

By \Cref{prop:critexist}, there is then a stable solution $u:\R^4\to(-1,1)$ with ${\bf M}_\infty=K_*$ which is not 1D. Let $\cZ(u)$ be its bad set, defined in \Cref{def:badballintro}; it is non-empty by \Cref{lem:badnonempty}.\\

Fix $\gamma=\frac{1}{4}$. Fix some $\chi\in(0,\frac{1}{20}]$, $\beta\in(0,\frac{1}{40}]$ and $\alpha\in(0,\frac{1}{40}]$.\\
 
\noindent Let $R_k\to \infty$ and $\zz_k\in \cZ$ be given by \Cref{lem:radcentselec}, with associated $\eps_k^2 = {\bf H}_{4R_k}^2({\zz_k}) \to 0$ as $k \to \infty$.\\
\noindent Let $\Rk=\widetilde \theta_kR_k$ and $\tilde\zz_k\in \cZ\cap B_{R_k}(\zz_k)$ be given by \Cref{lem:NbdRknew}, where $\tilde \theta_k \in (0, 1]$, and let $\epk = \widetilde \theta_k^{\cttb(1+\alpha)} \ep_k$. 
Recall that:
\begin{itemize}
    \item we have $\Rk \to \infty$ and ${\bf H}_{4\Rk}^2({\zk})\leq 2\epk^2 \to 0$ as $k \to \infty$.
    \item moreover \eqref{eq:hjgagfas} and \eqref{eq:fhiowqrhgoqwg} hold.
\end{itemize}

\noindent Combining \Cref{lem:compweighHM} and the monotonicity (by \Cref{lem:weighdensprop}) of ${\bf \widetilde M}_{r}$, there is some ${\bf \bar y}\in B_{\frac{1}{2}\Rk}(\zk)$ such that
\begin{equation*}
        K_*- C\epk^{2+\chi_\sharp/2}\leq {\bf \widetilde M}_r({\bf \bar y})\leq K_*\quad\mbox{for every}\quad r\geq \epk^{\chi_\sharp}\Rk\,,
\end{equation*}
with $\chi_\sharp=\frac{31}{30}\chi$ and for $k$ large enough.
Applying \Cref{prop:moncontrexc} with $r=4R_k$ and $\lambda r = \epk^{\chi_\sharp}\Rk$, and using \Cref{lem:compweighHM} once again,
\begin{align*}
    {\bf \widetilde H}_{4R_k}({\bf \bar y})&\leq {\bf \widetilde H}_{\epk^{\chi_\sharp}\Rk}({\bf \bar y})+C|\log(\lambda)|^{1/2}({\bf \widetilde M}_{4R_k}({\bf \bar y})-{\bf \widetilde M}_{\epk^{\chi_\sharp}\Rk}({\bf \bar y}))^{1/2}\leq C\epk^{1+\chi_\sharp/4}+C|\log(\lambda)|^{1/2}\epk^{1+\chi_\sharp/4}\,.
\end{align*}
Now, we can easily beat the log factor, by using (for the first and only time) that $\beta_0>0$: Since $\lambda=\frac{\epk^{\chi_\sharp}\Rk}{4R_k}=\frac{1}{4}\widetilde \theta_k\ep_k^{\chi_\sharp}$ and $\epk=\widetilde \theta_k^{\cttb(1+\alpha)} \ep_k$, we find that $|\log(\lambda)|^{1/2}\epk^{\chi_\sharp/8}\to 0$ as $\ep_k\to 0$. In particular, we deduce that
\begin{align*}
    {\bf \widetilde H}_{4R_k}({\bf \bar y})&\leq  C\ep_k^{1+\chi_\sharp/8}\qquad\mbox{for all } k \mbox{ large enough}.
\end{align*}
But then, since $B_{R_k}(\zz_k)\subset B_{2 R_k}(\zk)\subset B_{4 R_k}({\bf \bar y})$, together with \eqref{eq:8gsjgbss} we find that
\begin{align*}
    \ep_k={\bf H}_{R_k}(\zz_k)\leq C {\bf \widetilde H}_{R_k}(\zk)\leq C{\bf \widetilde H}_{4R_k}({\bf \bar y})&\leq  C\ep_k^{1+\chi_\sharp/8}\qquad\mbox{for all } k \mbox{ large enough}.
\end{align*}
\noindent Recalling once again that $\ep_k\to 0$ as $k\to\infty$, for $k$ large enough we get a contradiction.
\end{proof}

\subsection{Proof of \Cref{thm:curvestseps}}\label{sec:pfmainthm}
\begin{proof}[Proof of \Cref{thm:curvestseps}]
Let $u_\ep: B_1\subset\R^4\to (-1,1)$ and $\Lambda$ be as in the statement, so that ${\bf M}_1(u_\ep)\leq \Lambda$. If $u_\ep$ were a solution on all of $\R^4$ instead of just $B_1$, since we now known that every density is subcritical in $\R^4$ we would easily conclude by \Cref{prop:curvestindeps}. The only difference is that we cannot directly use \Cref{lem:monformula}; nevertheless, the local monotonicity formula with errors in \Cite[Proposition 3.4]{HT00} gives universal $C_1,C_2$ such that
\begin{equation}\label{eq:894galgbalbg}
    {\bf M}_r(u_\ep,x)\leq C_1\Lambda+C_2\quad\mbox{for every}\quad x\in B_{3/4}\,.
\end{equation}
Moreover, by \Cref{mainthm} we know that $K:=C_1\Lambda+C_2$ is a subcritical density. Then, up to using \eqref{eq:894galgbalbg} in place of \eqref{eq:781halavifa} and \eqref{eq:781halavifa2}, the rest of the proof of \Cref{prop:curvestindeps} follows exactly as written and gives the result.
\end{proof}

\appendix
\part*{Appendix}
\addcontentsline{toc}{part}{Appendix}
\section{The varifold theory for Allen--Cahn}

A central insight of the classical theory is that stable solutions to $\varepsilon$-A--C  converge, as $\varepsilon \downarrow 0$, to stable integral varifolds. We provide here in a restricted setting the most basic definitions and results underlying this convergence (in particular, we do not deal with general varifolds, and we state simpler versions of known results whenever they suffice).

\begin{definition}[{\bf Integral varifold}]\label{def:varifold}
    \begin{itemize}
        \item Let $\Omega\subset\R^n$ open. An {\bf integral} (or integer rectifiable) {\bf varifold} $V$ is a pair $V=(\Sigma,\theta)$, where $\Sigma\subset \Omega$ is an $(n-1)$-rectifiable set, and $\theta: \Sigma \to \N$ is an $\cH_{n-1}$--integrable function. We can naturally associate a {\bf weight measure} $\|V\|$ to it by $\|V\|(A)=\frac{1}{\omega_{n-1}}\int_{A\cap \Sigma} \theta \,d\mathcal H^{n-1}$.
        \item We define its support via ${\rm\supp} V:=\supp \|V\|$. Its regular part ${\rm reg}\, V$ is defined as the set of all points $x\in \supp V$ such that $\Sigma$ is actually a smooth hypersurface in some open neighbourhood around $x$. In particular, ${\rm reg}\, V$ is open. We define ${\rm sing}\, V=\supp V \setminus {\rm reg}\, V$.
        \item We say that $V$ is \textbf{stationary} in $\Omega$ if $\frac{d}{dt}\big|_{t=0}\|(\Phi_X^t)_\sharp V\|(\Omega)=0$, where $X$ is a smooth vector field with ${\rm supp}\, X\Subset \Omega$, $\Phi_X^t$ is its flow at time $t$, and $(\Phi_X^t)_\sharp V=(\Phi_X^t(\Sigma), \theta\circ \Phi_X^{-t})$. Moreover, it is called {\bf stable} if its regular part is a stable smooth minimal hypersurface, meaning that $\frac{d^2}{dt^2}\big|_{t=0}\|(\Phi_X^t)_\sharp V\|(\Omega)\geq 0$ for every $X$ with ${\rm supp}\, X\Subset{\rm reg}\, V$.
    \end{itemize}
\end{definition}
\begin{definition}[\textbf{Allen--Cahn weight measure, \cite{HT00}}]\label{def:varAC}
    Given a solution $u_\ep:\Omega\to[-1,1]$ to $\ep$-A--C, we define the Allen--Cahn weight measure $\|V_\ep\|$ on $\Omega$ via, given $A\subset\Omega$,
\begin{equation*}
    \|V_\ep\|(A):=\frac{1}{\sigma_{n-1}}\int_{-1}^{1} \sqrt{2W(t)}\,\mathcal H^{n-1}(\{u=t\}\cap A)\,dt\,.
\end{equation*}
We are using that, by Sard's theorem, $\{u=t\}$ is a smooth submanifold for almost every $t$.
\end{definition}
As a motivation for the definition above, note that by the coarea theorem we can write
\begin{align*}
    \|V_\ep\|(A)=\frac{1}{\sigma_{n-1}}\int_A \sqrt{2W(u_\ep)}|\nabla u_\ep|\,dx\,.
\end{align*}
Let $a:=\sqrt{\frac{2W(u_\ep)}{\ep}}$ and $b:=\sqrt{\ep}|\nabla u_\ep|$. Using that $2ab=a^2+b^2-(b-a)^2\geq a^2+b^2-|b^2-a^2|$ we readily deduce that
\begin{equation*}
    \mathcal E^\ep(u_\ep,A) - D(u_\ep,A)\leq \|V_\ep\|(A)\leq \mathcal E^\ep(u_\ep,A)\,,
\end{equation*}
where $D(u_\ep,A):=\frac{1}{\sigma_{n-1}}\int_A \left|\frac{W(u_\ep)}{\ep}-\frac{\ep|\nabla u_\ep|^2}{2}\right|\,dx$ is a discrepancy term.

Huchinson and Tonegawa established a deep connection between stationary points of $\ep$-A--C and integral stationary varifolds. The following is a partial statement of their results.
\begin{theorem}[\cite{HT00}]\label{thm:HutchTon}
Let $u_\ep:\Omega\to [-1,1]$ be a sequence of solutions to the $\ep$-A--C equation, with $\ep\to 0$, and assume that $\mathcal E^\ep(u_\ep,\Omega)\leq \Lambda$. Define the A--C weight measures $\|V_\ep\|$ as in \Cref{def:varAC}.\\
Then, there is a subsequence $u_{\ep_k}$ with $\ep_k\to 0$, and a stationary integer rectifiable varifold $V$, such that the following hold:
\begin{itemize}
    \item $\|V_{\ep_k}\|\to \|V\|$ as measures.
    \item $D(u_{\ep_k},A)\to 0$ for any $A\Subset\Omega$. In particular,
    \begin{equation}\label{eq:varenconv}
        \|V\|(\Omega')=\lim_k \|V_{\ep_k}\|(\Omega')=\lim_k\mathcal E^{\ep_k}(u_{\ep_k},\Omega') \quad \mbox{for any } \Omega'\Subset\Omega \mbox{ with } \|V\|(\partial \Omega')=0.
    \end{equation}
    \item As $\ep\to 0$,
    \begin{equation}\label{eq:varhausconv}
        \{|u_\ep|\leq 0.9\} \mbox{ converges locally in the Hausdorff distance sense to } {\rm spt}\, V.
    \end{equation}
\end{itemize}
\end{theorem}
We now focus on stable solutions.
\begin{remark}
    In the stable case, rescaling \Cref{thm:villegas} shows directly (and quantifies) the discrepancy decay.
\end{remark}
\begin{theorem}[\cite{Ton05}]\label{thm:TonStable}
    In the setting of \Cref{thm:HutchTon}, assume moreover that the $u_\ep$ are stable. Then, $V$ is also stable.\\
    Assume now furthermore that $\Omega\subset\R^2$, so that ${\rm supp}\, V$ is a union of straight segments. Then ${\rm supp}\, V$ is actually a union of non-intersecting straight lines.
\end{theorem}
The latter is a general property and holds also in higher dimensions, as shown by Tonegawa and Wickramasekera:
\begin{theorem}[{\cite[Proposition 3.2]{TonegawaWickramasekera2012}}]\label{prop:stablejunction}
    In the setting of \Cref{thm:HutchTon}, assume moreover that the $u_\ep$ are stable. Assume that $0\in\Omega$ and ${\rm spt}\|V\|=C\cap \Omega$, where $C\subset\R^n$ is a cone with $n-2$ directions of translation invariance. Then $C$ is a hyperplane.
\end{theorem}
\begin{remark}
    Cones $C$ with $n-2$ directions of translation invariance are, choosing an appropriate Euclidean coordinate frame, of the form $C_0\times\R^{n-2}$, where $C_0\subset\R^2$ is a union of half-lines intersecting at the origin. In other words, they are unions of half-hyperplanes meeting along a common boundary.
\end{remark}
The above are all the results we will use in the article. The following remark is in order:

\begin{remark}
The motivation for showing \Cref{prop:stablejunction} in \cite{TonegawaWickramasekera2012} is the following:
Wickramasekera developed a deep regularity theory for codimension $1$ stable integral varifolds in \cite{Wic14}, which shows that they cannot have branch points (i.e. flat singularities) as long as they do not have classical singularities\footnote{$V$ is said to have a classical singularity at $x\in {\rm sing}\,V$ if there is some $r>0$ such that ${\rm spt}\,V\cap B_r(x)$ is a $C^{1,\alpha}$ perturbation of a union of half-hyperplanes meeting along a common boundary.} either. In fact, such varifolds need to be then of optimal regularity (i.e. as regular as in the case of area-minimisers). This theory can then (by the property in \Cref{prop:stablejunction}) be applied in particular to limits of stable A--C solutions, which shows:
\begin{theorem}[\cite{TonegawaWickramasekera2012, Wic14}]\label{thm:tonwick}
    In the setting of \Cref{thm:HutchTon}, assume moreover that the $u_\ep$ are stable. Then, ${\rm sing}\,V$ is empty if $n\leq 7$, discrete if $n=8$, and of Hausdorff codimension at least $8$ if $n\geq 9$.
\end{theorem}
\end{remark}
We believe it is worth highlighting---as explained in \Cref{sec:critsol}---that our proof of \Cref{mainthm} does \emph{not} use the powerful \Cref{thm:tonwick}.

\section{A self-contained theory for stable solutions}\label{sec:stabselfcont}

In our current arguments, the results from \cite{HT00, TonegawaWickramasekera2012} are used only in two places:

The first instance is at the end of the proof of \Cref{prop:curvestindeps}. There, we need to argue that a minimal hypersurface, obtained as a $C^{2,\alpha}$ limit of (the level sets of) stable $\ep$-A--C solutions which satisfy the sheeting assumptions in balls of radius one, is also stable. We appealed to \Cref{thm:TonStable} for efficiency, but we could alternatively easily pass \eqref{eq:SZineq1} to the limit to recover the stability inequality for minimal hypersurfaces, using the $C^{2,\alpha}$ convergence and the fact that $|{\rm I\negthinspace I}_{\{u=u(x)\}}|^2(x)\leq\mathcal A^2(u)(x)$.

The second instance is in the proof of Propositions \ref{prop:W} and \ref{prop:deltamod2}, where we used Theorems \ref{thm:HutchTon}, \ref{thm:TonStable} and \ref{prop:stablejunction} to provide a short and direct argument. We explain in what follows how to avoid their use completely.

The fact that $K_*\in \N$, which is obained in the proof of \Cref{prop:W} via \Cref{thm:HutchTon}, is not needed at that point. We can instead proceed up until \Cref{sec:graphdec} without knowing that \( K_* \) is an integer, and only then---when proving \Cref{prop:graphdec}---{\it deduce} the integrality of \( K_* \) in a straightforward way: By \eqref{eq:gwiongown12}, we find a large ``clean'' ball with ${\bf M}_{\delta\rho}(\bar x)=K_*+o_k(1)$. Moreover, applying \Cref{lem:addproperties}, we find $N\in\N$ with ${\bf M}_{\delta\rho}(\bar x)=N+o_k(1)$. Combining both facts, obviously $K_*\in\N$. Such approach perfectly aligns with the ``a-priori philosophy'' that is present throughout the work.

The other properties in \Cref{prop:W,prop:deltamod2} can be obtained (without appealing to \cite{HT00, TonegawaWickramasekera2012}) as follows:
\begin{proof}[Alternative proof of \Cref{prop:W,prop:deltamod2} (sketch)]
    This alternative proof gives all the conclusions of  \Cref{prop:W,prop:deltamod2} except the integrality of $K_*$ (which as explained above can be deduced later in our argument).
    
    We adhere to the same notation and setup as in the original proof (\Cref{subsec:indlargesc}). Repeating the argument from Step 1 there, which does not use  \cite{HT00, TonegawaWickramasekera2012},  we obtain:
    \begin{equation}\label{pinch}
         K_*-\omega(\frac{1}{R})\leq \widetilde{\bf M}_{R}(\zz)\leq K_*\quad\mbox{for all}\quad \zz\in\cZ\,,
    \end{equation}
    where $\omega$ is a dimensional modulus of continuity. Notice that $\widetilde{\bf M}_{R}$ is the weighted excess from \Cref{def:heathexc}. 

     We then continue the proof verbatim, arguing by contradiction exactly as in the paragraph before Step 2---recall that we put $\widetilde u_i(x):=u_i(\zz_i+R_i x)$, and $\widetilde u_i$ is a solution of  $\widetilde\ep_i$-A--C with  $\widetilde \ep_i=\frac{1}{R_i}$. Letting $\mathcal Z(u_i)$ denote the bad centers of $u_i$, set $\widetilde Z_i : = \frac{1}{R_i}(\mathcal Z(u_i)-\zz_i) \cap B_1$, the sequence of rescaled (and translated) bad sets.

    \noindent {\bf Construction of a spine.} Define the  ``minimal spine dimension''
    \[
        d_* := {\rm min} \Big\{ d\in [0,n]\cap \mathcal N\quad |\quad   \liminf_{i\to \infty} \min_{L\in \mathrm{Gr}(d, n)}  \sup_{y\in \widetilde Z_i}{\rm dist}(y, L) =0\Big\},
    \]
    where $\mathrm{Gr}(d, n)$ denotes the Grassmannian of $d$-dimensional linear subspaces of $\R^n$.

    By definition of $d_*$ there exists a subsequence of $(u_{i_l})_{l\ge 0}$ such that for all $l$ we have: 
    \begin{itemize}
        \item There is $L\in \mathrm{Gr}(d_*, n)$ such that $\sup_{y \in \widetilde Z_{i_l}}{\rm dist}(y, L)\to 0$ as $l\to\infty$.
        \item There exist $d_*$ points in $\widetilde Z_{i_l}$, labeled $y^1_l,  y^2_l \dots, y^{d_*}_l$, such that the gram determinant of the $d_*$ vectors is bounded by below by the same positive number for all $l$.
    \end{itemize}

    By \eqref{pinch}, combined with the monotonicity formula (see \eqref{eq:weighmonM}) centered at $0$ and at $y^j_l$, we obtain
    \begin{equation*}
        \widetilde \eps_{i_l}\int_{B_2} (x \cdot \nabla \widetilde u_{i_l})^2  = o(1)  \quad \mbox{and} \quad    \widetilde \eps_{i_l}\int_{B_2} ((x-y^j_l) \cdot \nabla \widetilde u_{i_l})^2  = o(1) \quad \mbox{as }l\to \infty\,.
    \end{equation*}
    Given any $e\in L\cap B_{4}$, which we can then write as a linear combination of the $y_l^j$,  we deduce that
    \begin{equation}\label{trinv}
      \widetilde \eps_{i_l}\int_{B_2} (e \cdot \nabla \widetilde u_{i_l})^2  = o(1) \quad \mbox{and} \quad    \widetilde \eps_{i_l}\int_{B_2} ((x-e) \cdot \nabla \widetilde u_{i_l})^2  = o(1) \quad \mbox{as }l\to \infty.
    \end{equation}
    From now on, ${\bf \widetilde M}_r$ will denote the natural rescaled version of the original density (exactly as in \Cref{rmk:rescdens}). By the above, we can then deduce (for instance, arguing as in \cite[pp. 116-118]{Sim18}):
    \begin{lemma}\label{lem:89h1igbiob}
        Fix $x$ and $r>0$ such that $B_{2r}(x)\subset B_1\setminus L$. Then, the following hold:
        \begin{itemize}
            \item Dilation invariance: ${\bf \widetilde M}_r(x)={\bf \widetilde M}_r(x+\lambda (x-e))+o(1)$ for any $e\in L\cap B_1$ and $\lambda \in(0,1)$.
            \item Translation invariance: ${\bf \widetilde M}_r(x)={\bf \widetilde M}_r(x+e)+o(1)$ for any $e\in L\cap B_1$.
            \item Constancy in $L$: ${\bf \widetilde M}_r(0)={\bf \widetilde M}_r(e)+o(1)$ for any $e\in L\cap B_1$.
        \end{itemize}
    \end{lemma}
    
    \noindent {\bf Convergence to a hyperplane.} To conclude, we need to deduce that---up to a subsequence---there are hyperplanes $H_i$ such that, given $r>0$, for $i$ large enough the following hold:
    \begin{equation}\label{eq:782pibvwpfvb}
        \{|u|\leq 0.9\}\cap B_1\subset B_r(H_i)\qquad \quad\mbox{and}\qquad \quad {\bf \widetilde M}_r(u_i,x)\geq K_*-o(1)\quad\mbox{for every}\quad x\in H_i\cap B_1\,.
    \end{equation}
    Notice that \eqref{trinv} implies $d_*< n$: if we had $d_*=n$, then  \eqref{trinv} would imply $\widetilde \ep_{i_l}\int_{B_1} |\nabla \widetilde u_{i_l}|^2  = o(1)$ as $l\to \infty$, which readily gives a contradiction with lower density estimates.
    We distinguish then two cases $d_*=n-1$ and $d_* <n-1$.
    
    \noindent {\bf Case $d_*=n-1$.} Let $s>0$. By the third bullet in \Cref{lem:89h1igbiob} and ${\bf \widetilde M}_s(0)=K_*-o(1)$, we have ${\bf \widetilde M}_s(x)=K_*-o(1)$ $\forall x\in L\cap B_1$.  The second part of \eqref{eq:782pibvwpfvb} follows.
    
    To see the first one, let $r>0$. It is easy to see (by a cover of $L\cap B_1$ with disjoint balls of different radii) that $\mathcal E(u_i,B_{r/2}(L)\cap B_1)\geq K_*-o(1)$. Assume there were some $x\in\{|\widetilde u_{i_l}|\leq 0.9\}\cap B_{1/2}\setminus B_r(L)$ for contradiction; then
    $$K_*\geq \mathcal E(\widetilde u_{i_l},B_1)\geq \mathcal E(\widetilde u_{i_l},B_{r/2}(L)\cap B_1)+\mathcal E(\widetilde u_{i_l},B_{r/2}(x)\cap B_1)\geq K_*-o(1)+\mathcal E(\widetilde u_{i_l},B_{r/2}(x)\cap B_1)\,,$$
    which (by lower density estimates around $x$, see Step 1 in the proof of \Cref{lem:Hbdhtez}) is a contradiction for $\widetilde \eps_{i_l}>0$ small enough.
    
    \noindent {\bf Case $d_*<n-1$.} By \Cref{thm:goodimpsheet} we have the sheeting assumptions away from $L$ (as the bad set is contained in a $o(1)$ tube around $L$). In particular, the sets $\{\widetilde u_{i_l}=0\}$ enjoy $C^{2,\alpha}$ estimates. By  Arzel\`a--Ascoli---plus standard covering and diagonal arguments---we obtain some subsequence (not relabeled) converging (in $C^{2,\alpha}_{\rm loc}(B_1\setminus L)$) to a smooth minimal surface $\Sigma \subset B_1\setminus L$ (with $\partial \Sigma\subset L\cup\partial B_1$). By the argument at the beginning of the section, the $C^{2,\alpha}_{\rm loc}$ convergence implies moreover that $\Sigma$ is stable away from $L$.
    
    The following is immediate from \eqref{lem:1dapprox}:
    \begin{lemma}\label{lem:6713ifvoi}
        There is $R_0$ such that the following holds: Let $\delta>0$. Let $s,t$ be such that $|s|,|t|\leq 0.9$. Then, given $x\in\{\widetilde u_{i_l}=s\}\cap B_1\setminus B_{\delta}(L)$, we have---for $l$ large enough---that $\{\widetilde u_{i_l}=t\}\cap B_{R_0\widetilde \eps_{i_l}}(x)\neq \emptyset$ as well.\\
        In particular (by considering $s=0$), we deduce that $\{\widetilde u_{i_l}=t\}$ also converges both in the $C^{2,\alpha}$ and Hausdorff senses to $\Sigma$ in $B_1\setminus B_{\delta}(L)$ (with a rate independent of $t$).
    \end{lemma}
    In particular, $\Sigma$ must be nonempty: Otherwise, cover $B_1\cap B_r(L)$ with $\sim r^{-d_*}$ balls of radius $r>0$ small; the contributions of these balls (using that ${\bf M}$ is bounded) to the total energy in $B_1$  amount to at most $C r$. Moreover, by \Cref{lem:6713ifvoi}, $|u|\leq 0.9$ is contained in this cover (for $\widetilde \eps_{i_l}$ small enough), thus by \Cref{lem:expdecay} the energy of $\widetilde u_{i_l}$ decays exponentially away from these balls. Making $r$ small, the lower bound for the energy density in $B_1$ (in fact, it is almost $K_*$) yields a contradiction.\\

    Now, passing \eqref{trinv} to the limit, $\Sigma$ is conical and invariant in any direction $e\in L$. Thus, in an appropriate frame $\Sigma=C\times \R^{d_*}$, where $C\subset \R^{n-d_*}$ is a minimal cone which is smooth and stable outside the origin. We have then two subcases:

    {\bf Subcase 1:} If $d_*\le n-3$ (thus $3\leq n-d_*\leq 7$), $C$ (and thus $\Sigma$) is a hyperplane by \cite{Simons68}.
    
    {\bf Subcase 2:} If $ d_*=n-2$, $C\subset \R^2$ is a union of half-lines intersecting at the origin. We argue as follows:
    \begin{lemma}\label{lem:53dfitkdgo}
        There exists $U_{i_l}\subset B_1^{n-2}$, with ${\rm Vol}_{n-2}(U_{i_l})\geq c_0>0$, such that for every $x_0\in U_{i_l}$ and $t$ with $|t|\leq 0.9$ we have: Let $S_{y_0,t}:=\{x:(x,y_0)\in\{u=t\}\}\cap B_1^2$. Then $S_{y_0,t}\neq \emptyset$, and if $x\in S_{y_0,t}$ then $\{\widetilde u_{i_l}=t\}$ is a smooth hypersurface around $(x,y_0)$. Moreover, $|\nabla \widetilde u_{i_l}|\geq \frac{c}{\widetilde \eps_{i_l}}$.
    \end{lemma}
    \begin{proof}
        The fact that $S_{y_0,t}\neq \emptyset$ is immediate by applying \Cref{lem:6713ifvoi} around any point in $\Sigma\cap [B_1\setminus B_{\frac{1}{10}}(L)]\cap \{y=y_0\}\neq\emptyset$. Now, let $\Pi_{n-2}$ denote projection onto $\{0\}\times\R^{n-2}$, and define $U_{i_l}=B_1^{n-2}\setminus \Pi_{n-2}[B_{R_0\widetilde \eps_{i_l}}(\widetilde Z_i\cap  B_4)]$, $R_0$ universal to be chosen. By stability and Vitali, $B_{R_0\widetilde \eps_{i_l}}(\widetilde Z_i\cap B_4)$ can be covered with at most $C \widetilde \eps_{i_l}^{-(n-3)}$ balls of radius $\widetilde \eps_{i_l}$. Projecting these balls, ${\rm Vol}_{n-2}(\Pi_{n-2}[B_{R_0\widetilde \eps_{i_l}}(\widetilde Z_i\cap B_4)])\leq C\widetilde \eps_{i_l}^{-(n-3)}\widetilde \eps_{i_l}^{n-2}=C\widetilde \eps_{i_l}$, which can be made arbitrarily small as $\widetilde \eps_{i_l}\to 0$.
        
        Fixing $R_0$ (universal) large enough, for any $y_0\in U_{i_l}$ we can apply \Cref{thm:goodimpsheet} (appropriately rescaled) in $B_{R_0\widetilde \eps_{i_l}}(0,y_0)$, thus \eqref{eq:sheetassump} holds in $B_{\frac{R_0}{2}\widetilde \eps_{i_l}}(0,y_0)$. This gives $|\nabla \widetilde u_{i_l}|\geq \frac{c}{\widetilde \eps_{i_l}}$ as desired. In particular, if $|t|\leq 0.9$, then $\{\widetilde u_{i_l}=t\}$ is a smooth hypersurface around $S_{y_0,t}$.
    \end{proof}
    The rest is essentially the classical argument in \cite[pp. 785-787]{SS81}; we give a sketch of the proof, divided into two main lemmas. We have:
    \begin{lemma}
    Let $\sigma>0$. Then, for all $\widetilde \eps_{i_l}>0$ small enough, we have
        \begin{equation*}
            \inf_{t\in[-0.9,0.9],\,y_0\in U_{i_l}}\int_{S_{y_0,t}\cap B_\sigma^2} |{\rm I\negthinspace I}_{\{\widetilde u_{i_l}=t\}}|\,d\cH^{1}\leq \sigma^{1/2}. 
        \end{equation*}
    \end{lemma}
    \begin{proof}
        By Holder's inequality, \Cref{prop:SZstab} and the boundedness of ${\bf M}$ (on balls of radius $\sigma$), we easily see that
        $$\widetilde \eps_{i_l}\int_{B_\sigma^2\times B_1^{n-2}} \mathcal A|\nabla \widetilde u_{i_l}|^2\leq \left(\widetilde \eps_{i_l}\int_{B_1} \mathcal A^2|\nabla \widetilde u_{i_l}|^2\right)^{1/2}\left(\widetilde \eps_{i_l}\int_{B_\sigma^2\times B_1^{n-2}} |\nabla \widetilde u_{i_l}|^2\right)^{1/2}\leq C\sigma^{1/2}.$$
        On the other hand, by \Cref{lem:53dfitkdgo} and the coarea formula, we can estimate
    \begin{align*}
        \widetilde \eps_{i_l}\int_{B_\sigma^2\times B_1^{n-2}} \mathcal A|\nabla \widetilde u_{i_l}|^2&\geq \int_{\{|\widetilde u_{i_l}|\leq 0.9\}\cap (B_\sigma^2\times U_{i_l})} \mathcal A|\nabla \widetilde u_{i_l}|=\int_{-0.9}^{0.9} dt \int_{\{\widetilde u_{i_l}=t\}\cap (U_{i_l}\times B_\sigma^2)} \mathcal A \,d\cH^{n-1}\\
        &\geq\int_{-0.9}^{0.9} dt \int_{U_{i_l}}dy_0\int_{S_{y_0,t}\cap B_\sigma^2} |{\rm I\negthinspace I}_{\{\widetilde u_{i_l}=t\}}|\,d\cH^{1} \geq \inf_{t\in[-0.9,0.9],\,y_0\in U_{i_l}}\int_{S_{y_0,t}\cap B_\sigma^2} |{\rm I\negthinspace I}_{\{\widetilde u_{i_l}=t\}}|\,d\cH^{1}\,.
    \end{align*}
    \end{proof}
    On the other hand, we have:
    \begin{lemma}
        Assume that $\Sigma$ were not a hyperplane. Then, $\int_{S_{y_0,t}\cap B_\sigma^2} |{\rm I\negthinspace I}_{\{\widetilde u_{i_l}=t\}}|\,d\cH^{1}\geq c>0$ for every $\widetilde \eps_{i_l}>0$ small enough and $t\in[-0.9,0.9],\,y_0\in U_{i_l}$.
    \end{lemma}
    \begin{proof}
        By assumption, $\Sigma$ is a union of half-planes intersecting at $\{0\}^2\times \R^{n-2}$ and which form some angle. Moreover, $\{\widetilde u_{i_l}=t\}$ is smoothly embedded and it converges in $C^1_{loc}$ to $\Sigma$ away from $\{0\}^2\times \R^{n-2}$---thus, in particular, in $(B_\sigma^2\setminus B_{\sigma/2}^2)\times B_1^{n-2}$.
        
        By elementary considerations, for $\widetilde \eps_{i_l}>0$ small enough there are then $x_1,x_2\in B_{\sigma}^2\setminus B_{\sigma/2}^2$, both belonging to the same connected component of $S_{y_0,t}\cap B_\sigma^2$, and such that $|\nu_{\{\widetilde u_{i_l}=t\}}(x_1,y_0)-\nu_{\{\widetilde u_{i_l}=t\}}(x_2,y_0)|\geq c$; in other words, the normal vector to $\{\widetilde u_{i_l}=t\}$ needs to rotate by a definite positive angle along some connected component of $S_{y_0,t}$.
        
        Let $\gamma\subset S_{y_0,t}\cap B_\sigma^2$ be a curve segment connecting $(x_1,y_0)$ and $(x_2,y_0)$. Since $|\mathcal {\rm I\negthinspace I}_{\{u=t\}}|$ bounds any tangential derivative of the normal vector $\nu_{\{\widetilde u_{i_l}=t\}}$, the fundamental theorem of calculus then gives
        $$|\nu_{\{\widetilde u_{i_l}=t\}}(x_1,y_0)-\nu_{\{\widetilde u_{i_l}=t\}}(x_2,y_0)|=\int_{\gamma} \frac{d}{d\gamma} \nu_{\{u=t\}}\leq \int_{S_{y_0,t}\cap B_\sigma^2} |\mathcal {\rm I\negthinspace I}_{\{u=t\}}|\,.$$
    \end{proof}

    Combining the two preceding lemmas we deduce that $\Sigma$ must be a hyperplane, as otherwise fixing $\sigma>0$ small enough and then taking $\widetilde \eps_{i_l}>0$ small enough we get a contradiction.\\

    {\bf Conclusion.} In either of the cases $d_*=n-1$ and $d_* <n-1$, we found that $\Sigma$ is a hyperplane. The first property in \eqref{eq:782pibvwpfvb} immediately follows, by \Cref{lem:6713ifvoi}. To see the second one, let $r>0$. Since $\Sigma\setminus 
    B_{2r}(L)$ is connected, by \Cref{lem:addproperties} we find some $K\in\N$ such that ${\bf M}_r(x)=K+o(1)$ for any $x\in \Sigma\setminus B_{2r}(L)$. But then necessarily $K=K_*$: This follows easily since ${\bf M}_1(0)=K_*+o(1)$, and (as we already saw) $B_{s}(L)$ contributes at most $Cs$ to the total energy, which can be made arbitrarily small.
    
    Given that also ${\bf M}_r(y)=K_*+o(1)$ for any $y\in L$, we easily conclude (by combining different values of $r$) that ${\bf M}_r(x)=K_*+o(1)$ for {\it any} $x\in \Sigma$, which is precisely the second property in \eqref{eq:782pibvwpfvb}.
\end{proof}

\section{Proofs of some known or standard results}\label{app:standres}
\noindent {\bf Lemmas \ref{lem:Cacc-ineq-AC} and \ref{lem:expdecay}.}
\begin{proof}[Proof of \Cref{lem:Cacc-ineq-AC}]
This can be proved as in \cite[Remark 4.7]{Wang17}. Given a vector field $X\in C_c^1(\R^n;\R^n)$, multiplying \eqref{eq:aceqintro} by $X\cdot\nabla u$ and integrating by parts we get
    \begin{align*}
        \int \left[\frac{|\nabla u|^2}{2}+W(u)\right]{\rm div} X-(\nabla u)^T \cdot DX\cdot \nabla u=0\,,
    \end{align*}
    or (adding $\frac{|\nabla u|^2}{2}$ on both sides):
    \begin{align*}
        \int \Big[{\rm div} X-\sum_{1\leq i,j\leq n}\frac{u_i}{|\nabla u|}\frac{u_j}{|\nabla u|}\partial_i X^j\Big]|\nabla u|^2=\int \left[\frac{|\nabla u|^2}{2}-W(u)\right]{\rm div} X\,.
    \end{align*}
    We choose $X=x_{n}\eta^2e_n$, getting
    \begin{align*}
        \int \Big[\eta^2+2x_{n}\eta\eta_n-\frac{u_n}{|\nabla u|}\frac{u_n}{|\nabla u|}\eta^2-2\sum_{1\leq i\leq n}\frac{u_i}{|\nabla u|}\frac{u_n}{|\nabla u|}x_n\eta\eta_i\Big]|\nabla u|^2=\int \left[\frac{|\nabla u|^2}{2}-W(u)\right](\eta^2+2x_{n}\eta\eta_n)\,.
    \end{align*}
    The first and third terms add up to $|\nabla^{e_n'} u|^2\eta^2$. Moreover, the terms with $i<n$ in the sum can be estimated by 
    $$2\sum_{i<n} u_iu_nx_n\eta\eta_i\leq \frac{1}{2}\sum_{i<n} [u_i^2\eta^2+4u_n^2x_n^2\eta_i^2]\leq \frac{1}{2}|\nabla^{e_n'} u|^2\eta^2+2x_n^2|\nabla u|^2|\nabla \eta|^2\,.$$
    Bringing all remaining terms to the right hand side and using $\frac{|\nabla u|^2}{2}\le W(u)$ by \Cref{lem:modicaineq}, we obtain
    \begin{align*}
        \int |\nabla^{e_n'} u|^2\eta^2\leq 4\int x_n^2|\nabla u|^2|\nabla \eta|^2+C\int W(u)|x_n\eta\eta_n|
    \end{align*}
    as desired.
\end{proof}

\begin{proof}[Proof of \Cref{lem:expdecay}]
    Using \Cref{lem:modicaineq}, we can compute
    \begin{align*}
        \Delta (1-u^2)=-2[|\nabla u|^2+uW'(u)]\geq -2[2W(u)+uW'(u)]=-2(1-u^2)[\frac{1}{4}(1-u^2)-u^2].
    \end{align*}
    For $|u|\geq 0.85$ we find that $(1-u^2)$ is a strong subsolution, i.e.
    \begin{align*}
        \Delta (1-u^2)\geq c(1-u^2)
    \end{align*}
    (we can actually put $c=1$). By the maximum principle (using an exponential for comparison), this shows that
    \begin{align*}
        \sup_{B_r(x)} 1-u^2\leq e^{-c_0r} \quad \mbox{in any}\quad B_{2r}(x)\subset \{|u|\geq 0.85\}\,,
    \end{align*}
    which immediately gives the bound for $W(u)$; \Cref{lem:modicaineq} then bounds the gradient term as well.
\end{proof}
\noindent {\bf Lemma \ref{lem:linapprox}.}
We first need:
\begin{lemma}[\textbf{$W^{1,1}$ estimate}]\label{lem:calzygL1}
    Assume that $v\in C^2(B_{2R}')$, $B_{2R}'\subset \R^{n-1}$. Set $f:={\rm div} (A\nabla v)$, where $A\in C^1(B_{2R}')$ and $|A(x')-1|\leq 1/2$. Then, there exists $C$ depending only on $n$ such that
    $$R\fint_{B_R'}|\nabla v|\leq C\left(R^2\fint_{B_{2R}'} |f| + \fint_{B_{2R}'} |v|\right).$$
\end{lemma}

\begin{proof}[Proof of \Cref{lem:calzygL1}]
Since the estimate is scaling invariant we can (and do) assume that $R=1$. Consider $w$ which solves
\begin{equation*}
    \begin{cases}
        {\rm div}(Aw) &= f\quad \mbox{in}\ B_{2}'\\
        w &= 0 \quad \mbox{on}\ \partial B_{2}'\,.
    \end{cases}
\end{equation*}

Since $1/2\leq A(x')\leq 3/2$ is bounded and uniformly elliptic, \cite[Theorem 5.1]{LSW} shows that
$$\|w\|_{L^1(B_2)}+\|\nabla w\|_{L^1(B_2)}\leq C\|f\|_{L^1(B_2)}\,.$$
Moreover, $v-w$ satisfies ${\rm div}(A\nabla(v-w))=0$, therefore (applying Cauchy--Schwarz, the Caccioppoli inequality, and De Giorgi--Nash--Moser estimates)
$$\|\nabla(v-w)\|_{L^1(B_1')}\leq C\|\nabla(v-w)\|_{L^2(B_1')}\leq C\|v-w\|_{L^2(B_{3/2}')}\leq C\|v-w\|_{L^1(B_{2}')}\,.$$
Bounding $\|v-w\|_{L^1(B_1')}\leq \|v\|_{L^1(B_1')}+\|w\|_{L^1(B_1')}\leq \|v\|_{L^1(B_1')}+\|f\|_{L^1(B_2')}$, and adding up the gradient estimates for $w$ and $v-w$ above, we conclude the result.
\end{proof}

\begin{proof}[Proof of \Cref{lem:linapprox}]
We argue by compactness/contradiction. Suppose that the statement is not true; then, there are sequences $v_k,A_k$ satisfying the previous hypotheses for $\delta_k \downarrow 0$ but such that the conclusion fails for a certain $\lambda>0$.\\
Applying \Cref{lem:calzygL1} rescaled to the ball of radius $\rho$ we see that
$\rho\fint_{B_\rho'} |\nabla v_k|  \le \rho^{d+1/2}$, for $1\le \rho \le \tfrac{1}{2\delta_k}$. We can then apply Rellich--Kondrachov, deducing then that a subsequence (not relabeled) of the $v_k$ converges strongly in $L_{loc}^1$, and weakly in $W^{1,1}_{\rm loc}$, to some $v_{\infty}$, and with growth bound $\fint_{B_\rho} |v_{\infty}| \,dx \le \rho^{d+1/2}$.

Now, for any $\varphi\in C_c^2(\R^{n-1})$, the previous together with $|A_k(x')-1|\leq\delta_k\to 0$ give that
$$\int v_\infty \Delta\varphi=\lim_k\int v_k \Delta\varphi=\lim_k\int \nabla v_k \nabla\varphi=\lim_k\int A_k\nabla v_k \nabla\varphi=\lim_k\int {\rm div}(A_k\nabla v_k) \varphi=0\,,$$
thus $v_\infty$ is smooth and harmonic by the Weyl Lemma.\\
By the standard Liouville-type theorem for harmonic functions in $\R^{n-1}$ with polynomial growth, $v_\infty$ must be a harmonic polynomial $p_d$ of degree $\le d$. Moreover, from the growth bound with $\rho=1$ we see that $\|p_d\|_{L^1(B_1')}\le |B_1'|$. Since $v_k\to v_\infty=p_d$ in $L^1(B_1')$, this gives a contradiction for $k$ large enough.
\end{proof}

\section{Asymptotic behaviour of 1D periodic Allen--Cahn solutions}\label{sec:1DAC}
Let $\alpha\in[0,1]$ and
$$P(\alpha,R):= \inf\left\{\mathcal E\left (v,[0,\frac{1+\alpha}{4}\sqrt{2}\log{R}]\right ):v\in C_c^1([0,\frac{1+\alpha}{4}\sqrt{2}\log{R}])\text{ and } v(0)=0\right\}.
$$
\begin{proposition}\label{prop:1Dasympt}
There are universal constants $C>0$ and $R_0>1$ such that $P(\alpha,R)\geq \frac{1}{2}-CR^{-(1+\alpha)}$ for every $R\geq R_0$.
\end{proposition}

Considering
\begin{equation}\label{eq:aulgoagaog}
    P(R):= \inf\left\{\mathcal E\left (v,[0,\log{R}]\right ):v\in C_c^1([0,\log{R}])\text{ and } v(0)=0\right\},
\end{equation}
so that $P(\alpha, R)= P(R^{\frac{1+\alpha}{4}\sqrt{2}})
$,
it suffices to show:
\begin{proposition}\label{prop:qiogriogb}
There are universal constants $C>0$ and $R_0>1$ such that $P(R)\geq \frac{1}{2}-CR^{-2\sqrt{2}}$ for every $R\geq R_0$.
\end{proposition}
This should be read as: Any function with $v(0)=0$ needs to accumulate almost as much A--C energy as the 1D solution $\phi$ in long intervals. Indeed, we have:
\begin{lemma}\label{lem:pm12}
    $P(R)\leq \frac{1}{2}$.
\end{lemma}
\begin{proof}
    The monotone 1D solution $\phi$ is a competitor in \eqref{eq:aulgoagaog}, and it has $\mathcal E(\phi,[0,\infty))=\frac{1}{2}$, since $\mathcal E(\phi,(-\infty,\infty))=1$ and $\phi$ is antisymmetric.
\end{proof}
We argue in several steps.
\begin{lemma}
    There is $R_0>1$ such that, if $R>R_0$, the following hold:\\
    The infimum in \eqref{eq:aulgoagaog} is attained by an A--C solution $v\in C_c^1([0,\log{R}])$, satisfying $v(0)=0$, $v'(0)>0$, and $v'(\log{R})=0$. Moreover, $v$ is the restriction of a (not renamed) global, periodic  A--C solution $v:\R\to(-1,1)$ with quarter-period $\frac{T}{4} =\log{R}$. Furthermore, $v$ is nondecreasing in $[0,\frac{T}{4}]$.
\end{lemma}
\begin{proof}
    It is elementary to see that there is an $v$ which attains the infimum. Since $v$ solves a minimisation problem, with zero Dirichlet condition at $0$ and no constraint at $\log{R}$, we get that $v(0)=0$ and $v'(\log{R})=0$. Moreover, $v'(0)\neq 0$ (as otherwise $v\equiv 0$, but then $P(R)=\int_0^{\log R} W(0)= \frac{1}{4}\log R$, which is a contradiction for $R$ large enough with \Cref{lem:pm12}), so that $v'(0)>0$ up to perhaps considering $-v$ instead.
    
    Since $v$ is a minimizer of the A--C energy in $[0,\log R]$ among all functions with $v(0)=0$ and $v'(0)>0$, and since the potential $W$ is decreasing in $[0,1]$ and $W(1)=0$, we obtain that:
    \begin{itemize}
        \item $v$ is nondecreasing in $[0, \log R]$ (as otherwise the smallest nondecreasing function above $v$ would have less energy).
        \item $v\le 1$ (as otherwise the function $\min (v,1)$ would have less energy).
    \end{itemize}

    Moreover, a symmetrisation by hand alternating even and odd reflections (using that $v'(\log R)=0$) immediately shows that $v$ can be extended periodically, with quarter period $\frac{T}{4}=\log{R}$ for some $k\in\N$.  
\end{proof}
\begin{lemma} 
    In $[0,\frac{T}{4}]$, $v'=\sqrt{2W(v)-2\lambda}$  for some $\lambda\in(0,\frac 14)$, so that $v$ coincides with the solution to
    \begin{equation}\label{eq:lamb1dsolcases}
        \begin{cases}
            u_\lambda''=W'(u_\lambda) \\
            u_\lambda(0)=0 \quad \mbox{and} \quad u_\lambda'(0)=\sqrt{\frac{1}{2}-2\lambda}.\\
            
        \end{cases}
    \end{equation}
    Moreover, $v$ has amplitude $C_\lambda := \sqrt{1-2\sqrt{\lambda}}$, attained at $\frac{T}{4}$.
\end{lemma}
\begin{proof}
    Since \[
\frac{d}{dt} \left( v'^2 - 2 W(v) \right) = 2 v' v'' - 2 v' W'(v) = 0
\]
and $v'> 0$ in $[0,\frac{T}{4})$,
$v' = \sqrt{2 W(u) - 2\lambda}$ for some $\lambda \in \mathbb{R}$ in this interval. Since $|v'|\leq \sqrt{2W(v)}$ by \Cref{lem:modicaineq}, $\lambda>0$. If $\lambda = 0$ we would get $v=\phi$, but $\phi'$ never vanishes; since $|W|\leq \frac{1}{4}$ we deduce that $\lambda\in(0,1/4)$. Moreover, since $v$ attains its maximum $\max v$ at  $\frac{T}{4}$ and $v'(\frac{T}{4})=0$, we see that $0=\sqrt{2W(\max v)-2\lambda}$, which since $W(v)=\frac{1}{4}(1-v^2)^2$ shows that $\max v = v(T/4) =\sqrt{1-\sqrt{4\lambda}}$.
\end{proof}

We will also need the following
\begin{lemma}
    We have
    \begin{align*}
    P(R)\geq \frac{1}{2}-\frac{1}{\sigma_0}\int_{C_\lambda}^1\sqrt{2W(s)}\,ds.
\end{align*}
\end{lemma}
\begin{proof}
Recall that $\frac{T_\lambda}{4}=\log R$. By $a^2+b^2\geq 2ab$ and coarea, we can bound
\begin{align*}
    P(R)=\frac{1}{\sigma_0}\int_0^{\frac{T_\lambda}{4}}\left[\frac{1}{2}|\nabla u_\lambda|^2+W(u_\lambda)\right]\geq \frac{1}{\sigma_0}\int_0^{\frac{T_\lambda}{4}}|\nabla u_\lambda|\sqrt{2W(u_\lambda)}=\frac{1}{\sigma_0}\int_0^{C_\lambda}\sqrt{2W(s)}\,ds\,,
\end{align*}
which since $\sigma_0=\int_{-1}^1 \sqrt{2W(s)}\,ds=2\int_{0}^1 \sqrt{2W(s)}\,ds$ gives
\begin{align*}
    P(R)\geq \frac{1}{2}-\frac{1}{\sigma_0}\int_{C_\lambda}^1\sqrt{2W(s)}\,ds\,.
\end{align*}
\end{proof}
We can now prove \Cref{prop:qiogriogb}.
\begin{proof}[Proof of \Cref{prop:qiogriogb}]
Let $\delta= \delta_\lambda$ be such that $C_\lambda  = u_\lambda(T_\lambda/4)= 1-\delta$. We want to see that $\delta$ is small.

Recall that \( u_\lambda' = \sqrt{2W(u_\lambda) - \lambda} \).
Dividing by $\sqrt{2W(u_\lambda) - \lambda}$, integrating from $0$ to $\frac{T_\lambda}{4}$, and setting $s=u(r)$ gives
\begin{equation}
     \frac{T_{\lambda}}{4}= \int_0^{\frac{T_\lambda}{4}} \frac{u'(r)}{\sqrt{2W(u(r)) - 2\lambda}}\,dr=\int_0^{1-\delta} \frac{ds}{\sqrt{2W(s) - 2W(1-\delta)}}\,,
\end{equation}
where we used that  (since $u_\lambda'(T_\lambda/4)=0$) we have $2\lambda = 2W(u_\lambda(T_\lambda/4))=2W(1-\delta)$. 

Changing variables $s\xrightarrow[]{} 1-s$, bounding $2-s\geq  2-\delta$, and changing variables $s\xrightarrow[]{} \delta s$, we can estimate
\[
\begin{split}
\frac T 4&=\frac{1}{\sqrt{2}}\int_0^{1-\delta}  \frac{ds }{\sqrt{W(s)-W(1-\delta)}}  = \frac{2}{\sqrt{2}}\int_\delta^1 \frac{ds}{\sqrt{(2-s)^2s^2 - (2-\delta)^2\delta^2}} 
\\
&\le \frac{\sqrt{2}}{2-\delta}\int_\delta^1 \frac{ds}{\sqrt{s^2 - \delta^2}}= \frac{\sqrt{2}}{2-\delta}\int_1^{1/\delta} \frac{ds}{\sqrt{s^2 - 1}}\,.
\end{split}
\]
Therefore, since $\frac{1}{\sqrt{s^2 - 1}}=\frac{1}{s}+O(\frac{1}{s^2})$ for $s\geq 2$, we conclude that
$$
\frac T 4\le \frac{\sqrt{2}}{2-\delta}\int_2^{1/\delta} \frac{ds}{s} +O(1)\le \frac{1}{\sqrt{2}} \log \frac 1 \delta  +O(1)\,.
$$
Since $T/4=\log R$, exponentiating on both sides we deduce that $\delta \leq CR^{-\sqrt{2}}$.\\

Now that we have bounded $\delta$, since 
\begin{align*}
    P(R)\geq \frac{1}{2}-\int_{C_\lambda}^1\sqrt{2W(s)}\,ds,
\end{align*}
it remains to estimate $\int_{C_\lambda}^1\sqrt{2W(s)}\,ds$. From $\sqrt{2W(s)}=\frac{1}{\sqrt{2}}(1-s^2)\leq C(1-s)$ for $s\in[0,1]$, we deduce that
\begin{align*}
    \int_{C_\lambda}^1\sqrt{2W(s)}\,ds&\leq C\int_{C_\lambda}^1 (1-s)\,ds=\frac{C}{2} (1-C_\lambda)^2 = \frac{C}{2}\delta^2 \le CR^{-2\sqrt 2},
\end{align*}
and the proposition follows.
\end{proof}


\addcontentsline{toc}{section}{\large References}
\printbibliography

\end{document}